\documentclass[12pt,centertags,oneside]{amsart}
\usepackage{a4}
\usepackage{amsmath,amstext,amsthm,amssymb,amsfonts,amscd}
\usepackage{mathrsfs}
\usepackage[colorlinks=true]{hyperref}
\usepackage{graphicx}
\usepackage[T1]{fontenc}
\usepackage[latin1]{inputenc}
\usepackage{color,tocvsec2}
\usepackage{typearea}
\usepackage{charter}
\usepackage{enumerate}
\usepackage{lscape}
\usepackage[all]{xy}

\usepackage{multicol}        
\makeindex             

\usepackage[a4paper,width=16.2cm,top=2.5cm,bottom=2.5cm]{geometry}

\theoremstyle{plain}
\newtheorem{lem}{Lemme}[section]
\newtheorem{thm}[lem]{Theorem}
\newtheorem{prop}[lem]{Proposition}
\newtheorem{cor}[lem]{Corollary}
\newtheorem{as}[lem]{Assumption}

\theoremstyle{definition}
\newtheorem{defin}[lem]{Definition}

\theoremstyle{remark}
\newtheorem{re}[lem]{Remark}

\numberwithin{equation}{section}
\numberwithin{figure}{section}

\newcommand{\cE}{\mathcal{E}}
\newcommand{\cF}{\mathcal{F}}

\newcommand{\cO}{\mathcal{O}}

\newcommand{\cS}{\mathcal{S}}

\newcommand{\cZ}{\mathcal{Z}}

\newcommand{\bZ}{\mathbf{Z}}
\newcommand{\bR}{\mathbf{R}}
\newcommand{\bC}{\mathbf{C}}

\newcommand{\fb}{\mathfrak{b}}
\newcommand{\fg}{\mathfrak{g}}
\newcommand{\fh}{\mathfrak{h}}
\newcommand{\fk}{\mathfrak{k}}
\newcommand{\fm}{\mathfrak{m}}
\newcommand{\fn}{\mathfrak{n}}
\newcommand{\fp}{\mathfrak{p}}
\newcommand{\fq}{\mathfrak{q}}
\newcommand{\ft}{\mathfrak{t}}
\newcommand{\fz}{\mathfrak{z}}
\newcommand{\fu}{\mathfrak{u}}
\newcommand{\fa}{\mathfrak{a}}

\DeclareMathOperator{\Trs}{\mathrm Tr_s}
\DeclareMathOperator{\Tr}{\mathrm Tr}
\DeclareMathOperator{\Ad}{\mathrm Ad}

\DeclareMathOperator{\vol}{\mathrm vol}

\DeclareMathOperator{\ad}{\mathrm ad}
\DeclareMathOperator{\Sp}{\mathrm{Sp}}

\renewcommand{\Re}{\mathrm{Re}\,}

\DeclareMathOperator{\End}{\mathrm End}
\DeclareMathOperator{\Hom}{\mathrm Hom}

\DeclareMathOperator{\rk}{\mathrm rk}
\DeclareMathOperator{\GL}{\mathrm GL}
\DeclareMathOperator{\SO}{\mathrm SO}

\newcommand{\<}{\langle}
\renewcommand{\>}{\rangle}
\newcommand{\ol}{\overline}
\newcommand{\ul}{\underline}

\newcommand{\p}{\partial}

\renewcommand{\(}{\left(}
\renewcommand{\)}{\right)}
\renewcommand{\[}{\left[}
\renewcommand{\]}{\right]}
\renewcommand{\l}{\leqslant}
\newcommand{\g}{\geqslant}
\newcommand{\e}{\epsilon}

\newcommand{\bbS}{\mathbb{S}}

\newcommand\overmat[3][0pt]{%
  \makebox[0pt][l]{$\smash{\overbrace{\phantom{%
    \begin{matrix}\phantom{\rule{0pt}{#1}}#3\end{matrix}}}^{\text{#2}}}$}#3}

\begin{document}

\title{Analytic torsion, dynamical zeta functions, and  the Fried 
conjecture}
\author{Shu Shen}
\address{D\'epartement de Math\'ematique B\^atiment 425, Universit\'e Paris-Sud, 91405 Orsay, France.}
\email{shu.shen@math.u-psud.fr}
\thanks {}


\dedicatory{}

\begin{abstract}
We prove the equality of the analytic 
torsion and  the value at zero of a Ruelle dynamical
zeta function associated with an acyclic unitarily
flat vector bundle on a closed locally symmetric reductive
manifold. This solves a conjecture of Fried. This article should be 
read in conjunction with an earlier paper by   Moscovici and Stanton. 
\end{abstract}

\maketitle
\tableofcontents

\settocdepth{subsection}
\section{Introduction}
The purpose of this article is to prove the equality of the analytic 
torsion and the value at zero of a Ruelle dynamical
zeta function associated with an acyclic unitarily
flat vector bundle on a closed locally symmetric reductive
manifold, which completes a gap in the proof given by Moscovici and Stanton \cite{MStorsion} and
solves a conjecture of Fried \cite{Friedconj}.  

Let $Z$ be a smooth closed manifold. Let $F$ be a complex  vector bundle equipped with a flat Hermitian metric $g^F$.
Let $H^\cdot(Z,F)$ be the cohomology of sheaf of locally flat sections of $F$.
We assume $H^\cdot(Z,F)=0$. 

The Reidemeister torsion has been introduced  by Reidemeister \cite{ReidemeisterTorsion}. 
 It is a positive real number one obtains via the combinatorial complex with values in $F$ associated with a triangulation of $Z$, which can be shown not to depend on the triangulation.

%

Let $g^{TZ}$ be a Riemannian metric on $TZ$. Ray and Singer \cite{RSTorsion} constructed the analytic torsion $T(F)$
as a spectral invariant of the Hodge Laplacian associated with $g^{TZ}$ and $g^F$.
They showed that if $Z$ is an even dimensional oriented manifold, 
then $T(F)=1$. Moreover, if $\dim Z$ is odd, then  $T(F)$ does not 
depend on the metric data. 

In \cite{RSTorsion}, Ray and Singer conjectured an equality between the Reidemeister torsion and the analytic torsion, which was later proved by Cheeger \cite{Ch79} and M\"uller \cite{Muller78}. 
Using the Witten deformation, Bismut and Zhang \cite{BZ92} gave an 
extension of the Cheeger-M\"uller Theorem which is valid for 
arbitrary flat vector bundles.



From the dynamical side, in \cite[Section 3]{MilnorZcover}, Milnor  pointed out a remarkable similarity between the Reidemeister torsion and
the Weil zeta function.
A quantitative description of their relation was formulated by Fried \cite{FriedRealtorsion} when $Z$ is a closed oriented hyperbolic manifold. Namely, he showed that the value at zero of the Ruelle dynamical zeta function, constructed using the closed geodesics in $Z$ and the holonomy of $F$, is equal to  $T(F)^2$. In \cite[p. 66, Conjecture]{Friedconj}, Fried conjectured that a similar result holds true for general closed locally homogeneous manifolds.


In this article, we prove the Fried conjecture for odd dimensional\footnote{The even dimensional case is trivial.} closed locally symmetric
reductive manifolds. More precisely, we show that the  dynamical zeta function is
meromorphic on $\bC$, holomorphic at $0$, and that its  value at $0$ is equal to $T(F)^2$.

The proof of the above result by Moscovici-Stanton \cite{MStorsion}, based
on the Selberg trace formula and harmonic analysis on reductive groups, does not seem to be complete. We tried to give the proper argument to
make it correct. Our proof is based on the explicit formula given by Bismut for semisimple orbital
integrals \cite[Theorem 6.1.1]{B09}.

The results contained in this article was announced in 
\cite{Shen_Fried_Conj_CRAS}. See also Ma's talk \cite{Ma_bourbaki} at 
S\'eminaire Bourbaki for an introduction. 

Now, we will describe our results in more details, and explain the techniques used in their proof.

\subsection{The analytic torsion}
Let $Z$ \index{Z@$Z$}   be a smooth closed manifold, and let $F$\index{F@$F$} be a complex flat vector bundle on $Z$.

Let $g^{TZ}$ \index{G@$g^{TZ}$} be a Riemannian metric on $TZ$, and let $g^F$\index{G@$g^{F}$} be a Hermitian metric on $F$. To $g^{TZ}$ and $g^F$, we can associate an $L^2$-metric on $\Omega^\cdot(Z,F)$ \index{O@$\Omega^\cdot(Z,F)$}, the space of differential forms with values in $F$. Let  $\Box^Z$ \index{B@$\Box^Z$} be the  Hodge Laplacian acting on $\Omega^\cdot(Z,F)$. By  Hodge theory, we have a canonical isomorphism
\begin{align}\label{eq:Hodgein}
  \ker\Box^Z\simeq H^\cdot(Z,F).
\end{align}

Let $\(\Box^{Z}\)^{-1}$ be the inverse of $\Box^Z$ acting on the 
orthogonal space to $ \ker \Box^Z$.
Let $N^{\Lambda^\cdot(T^*Z)}$ \index{N@$N^{\Lambda^\cdot(T^*Z)}$} be the number operator of $\Lambda^\cdot(T^*Z)$, i.e., multiplication by $i$ on $\Omega^i (Z, F )$.
 Let $\Trs$ denote  the supertrace.
For $s\in\bC$, $\Re(s)>\frac{1}{2}\dim Z$, set
\begin{align}
  \theta(s)=-\Trs\[N^{\Lambda^\cdot(T^*Z)}\(\Box^{Z}\)^{-s}\].\index{T@$\theta(s)$}
\end{align}
By \cite{Seeley66},  $\theta(s)$ has a meromorphic extension to $\bC$, which is holomorphic at $s=0$. The analytic torsion is a positive real number given by
\begin{align}
T(F)=\exp(\theta'(0)/2).\index{T@$T(F)$}
\end{align}
Equivalently, $T(F)$ is given by the following weighted product of the zeta regularized determinants
\begin{align}\label{eq:inn}
T(F)=\prod_{i=1}^{\dim Z}\det\(\Box^Z|_{\Omega^i(Z,F)}\)^{(-1)^ii/2}.
\end{align}

\subsection{The dynamical zeta function }
Let us recall Fried's general definition of the formal dynamical zeta function  associated to a geodesic flow \cite[Section 5]{Friedconj}.

Let $(Z,g^{TZ})$ be a connected manifold with nonpositive sectional curvature. Let $\Gamma=\pi_1(Z)$\index{G@$\Gamma$} be the fundamental group of $Z$, and let $[\Gamma]$ be the set of the conjugacy classes of $\Gamma$. We identify $[\Gamma]$ \index{G@$[\Gamma]$}with the free homotopy space of $Z$. 
For $[\gamma]\in [\Gamma]$ \index{G@$[\gamma]$}, let $B_{[\gamma]}$\index{B@$B_{[\gamma]}$} be the set of closed geodesics, parametrized by $[0,1]$, in the class $[\gamma]$. The map
$x_\cdot\in B_{[\gamma]}\to (x_0,\dot{x}_0/|\dot{x}_0|)$
induces an identification between $\coprod_{[\gamma]\in [\Gamma]-\{1\}}B_{[\gamma]}$ and the fixed points of the geodesic flow at time $t=1$ acting on the unit tangent bundle  $SZ$. Then, $B_{[\gamma]}$ is equipped with the  induced topology,  is connected and compact. Moreover, all the elements in $B_{[\gamma]}$ have the same length $l_{[\gamma]}$.\index{L@$l_{[\gamma]}$} Also, the Fuller index ${\rm ind}_F(B_{[\gamma]}) \in \mathbf{Q}$ is well defined (c.f. \cite[Section 4]{Friedconj}).
Given a finite dimensional representation $\rho$ of $\Gamma$, for $\sigma\in \bC$, the formal dynamical zeta function is then defined by\index{R@$R_\rho(\sigma)$}
\begin{align}\label{eq:RFuller}
R_{\rho}(\sigma)=\exp\(\sum_{[\gamma]\in [\Gamma]- \{1\}} \Tr[\rho(\gamma) ]\mathrm{ind}_F(B_{[\gamma]})e^{-\sigma l_{[\gamma]}}\).\end{align}
Note that our definition is the inverse of the one introduced by Fried \cite[P. 51]{Friedconj}.

The Fuller index can be made explicit in many case. If $[\gamma]\in 
[\Gamma]- \{1\}$, the group $\bbS^1$ acts locally freely on 
$B_{[\gamma]}$ by rotation. Assume that the $B_{[\gamma]}$ are  
smooth manifolds. This is the case if $(Z,g^{TZ})$ has a negative sectional curvature or if $Z$ is locally symmetric.  Then $\bbS^1\backslash B_{[\gamma]}$ is an orbifold. Let $\chi_{\rm orb}(\bbS^1\backslash B_{[\gamma]})\in \mathbf{Q}$ \index{C@$\chi_{\rm orb}$} be the orbifold Euler characteristic \cite{SatakeGaussB}. Denote by \index{M@$m_{[\gamma]}$}
\begin{align}
m_{[\gamma]}=\left|\ker\big(\bbS^1\to {\rm 
Diff}(B_{[\gamma]})\big)\right|\in \mathbf{N}^{*}
\end{align}
the multiplicity of a generic element in $B_{[\gamma]}$. By \cite[Lemma 5.3]{Friedconj}, we have
\begin{align}\label{eq:Fuller=Euler}
\mathrm{ind}_F(B_{[\gamma]})=\frac{\chi_{\rm orb}(\bbS^1\backslash B_{[\gamma]})}{m_{[\gamma]}}.
\end{align}
By \eqref{eq:RFuller} and \eqref{eq:Fuller=Euler}, the formal dynamical zeta function is then given by
\begin{align}\label{eq:ZetaE}
R_{\rho}(\sigma)=\exp\(\sum_{[\gamma]\in [\Gamma]- \{1\}} \Tr[\rho(\gamma) ]\frac{\chi_{\rm orb}(\bbS^1\backslash B_{[\gamma]})}{m_{[\gamma]}}e^{-\sigma l_{[\gamma]}}\).
\end{align}
We will say that the formal dynamical zeta function is well defined if $R_\rho(\sigma)$ is  holomorphic for $\Re(\sigma)\gg1$ and extends meromorphically to $\sigma\in \bC$.

Observe  that if $(Z,g^{TZ})$ is of  negative sectional curvature,  then $B_{[\gamma]}\simeq \mathbb{S}^1$ and
\begin{align}
\chi_{\rm orb}(\bbS^1\backslash B_{[\gamma]})=1.
\end{align}
In this case, $R_\rho(\sigma)$ was recently shown to be well defined  by Giulietti-Liverani-Pollicott \cite{GLP2013} and Dyatlov-Zworski \cite{DyatlovZworski}. Moreover, Dyatlov-Zworski \cite{zworski_zero} showed that, if $(Z,g^{TZ})$ is a negatively curved surface, the order of the zero of $R_\rho(\sigma)$ at $\sigma=0$ is related to  the genus of $Z$.

\subsection{The Fried conjecture}
Let us briefly recall  Fried's results in \cite{FriedRealtorsion}.
Assume
$Z$ is an odd dimensional   connected orientable closed hyperbolic manifold.
%
%
%
Take $r\in \mathbf{N}$. Let $\rho:\Gamma\to U(r)$ be a unitary representation of the fundamental group $\Gamma$.
Let $F$ be the unitarily flat vector bundle on $Z$ associated to $\rho$.

Using the Selberg trace formula, Fried  \cite[Theorem 
3]{FriedRealtorsion} showed that there exist explicit constants 
$C_\rho\in \bR^{*}$ and $r_\rho\in \mathbf{Z}$ such that as  $\sigma\to 0$,
\begin{align}
R_\rho(\sigma)=C_\rho T(F)^2\sigma^{r_\rho}+\cO(\sigma^{r_{\rho}+1}).
\end{align}
Moreover, if $H^\cdot(Z,F)=0$, then
\begin{align}
&C_\rho=1,&r_\rho=0,
\end{align}
so that
\begin{align}\label{eq:Fintro}
R_\rho(0)=  T(F)^2.
\end{align}

%
%
%
%
%
%

In \cite[P. 66, Conjecture]{Friedconj}, Fried  conjectured that the same holds true when $Z$ is a general closed locally  homogeneous manifold. 

\subsection{The $V$-invariant} In this and in the following subsections, we give a formal proof of \eqref{eq:Fintro} using the $V$-invariant of Bismut-Goette \cite{BGdeRham}. 

Let $S$ be a closed manifold equipped with an action of a compact Lie group  $L$, with Lie algebra $\mathfrak{l}$.
If $a\in \mathfrak{l}$, let $a^S$ be the corresponding vector field on $S$. Bismut-Goette \cite{BGdeRham} introduced the $V$-invariant $V_a(S)\in \bR$.  
%

Let $f$ be an $a^S$-invariant Morse-Bott function on $S$. Let 
$B_f\subset S$ be  the critical submanifold. Since $a^S|_{B_f}\in TB_f$, $V_a(B_f)$ is also well defined.
By \cite[Theorem 4.10]{BGdeRham}, $V_a(S)$ and $V_a(B_f)$ are related by a simple formula.


\subsection{Analytic torsion and the  $V$-invariant}
Let us argue  formally.
%
Let $LZ$ be the free loop space of $Z$ equipped with the canonical $\bbS^1$-action. Write $LZ=\coprod_{[\gamma]\in [\Gamma]}(LZ)_{[\gamma]}$ as a disjoint union of its connected components. Let $a$ be the generator of the  Lie algebra of $\bbS^1$ such that $\exp(a)=1$.
  As explained in \cite[Equation (0.3)]{B05}, 
 if $F$ is a unitarily flat vector bundle on $Z$ such that $H^\cdot(Z,F)=0$, at least formally, we have
\begin{align}\label{eq:V=T}
  \log T(F)=-\sum_{[\gamma]\in [\Gamma]}\Tr[\rho(\gamma)]V_a\big((LZ)_{[\gamma]}\big).
\end{align}

%

Suppose that  $(Z,g^{TZ})$ is an  odd dimensional  connected closed manifold of nonpositive sectional curvature, and   suppose that the energy functional
 \begin{align}
  E:x_\cdot\in LZ\to \frac{1}{2}\int_0^1|\dot{x}_s|^2ds
\end{align}
on $LZ$ is  Morse-Bott. The critical set of $E$ is just 
$\coprod_{[\gamma]\in [\Gamma]}B_{[\gamma]}$, and  all the critical 
points  are local minima. Applying  \cite[Theorem 4.10]{BGdeRham}  to 
the infinite  dimensional manifold $(LZ)_{[\gamma]}$ equipped with the $\mathbb{S}^1$-invariant Morse-Bott functional $E$, we has the formal identity
\begin{align}\label{eq:16}
V_a\big((LZ)_{[\gamma]}\big)=V_a\big(B_{[\gamma]}\big).
\end{align}
Since $B_{[1]}\simeq Z$ is formed of the trivial closed geodesics, by 
the definition of the $V$-invariant,
\begin{align}
V_a(B_{[1]})=0.
\end{align}
By \cite[Proposition 4.26]{BGdeRham}, if $[\gamma]\in [\Gamma]- \{1\}$, then
\begin{align}\label{eq:17}
V_a(B_{[\gamma]})=-\frac{\chi_{\rm orb}(\mathbb{S}^1\backslash B_{[\gamma]})}{2m_{[\gamma]}}.
\end{align}
By \eqref{eq:V=T}, \eqref{eq:16}-\eqref{eq:17}, we get a formal identity
\begin{align}\label{eq:18}
\log T(F)=\frac{1}{2}\sum_{[\gamma]\in [\Gamma]- \{1\}}\Tr[\rho(\gamma)]\frac{\chi_{\rm orb}(\mathbb{S}^1\backslash B_{[\gamma]})}{m_{[\gamma]}},
\end{align}
which is formally equivalent to \eqref{eq:Fintro}.
\subsection{The main result of the article}
Let $G$ \index{G@$G$} be a linear connected real reductive group \cite[p. 3]{Knappsemi}, and let $\theta$\index{T@$\theta$} be the  Cartan involution. Let
$K$\index{K@$K$} be the maximal compact subgroup of $G$ of the points of $G$ that are fixed by $\theta$.
Let $\fk$ \index{K@$\fk$}and  $\fg$\index{G@$\fg$} be the  Lie algebras of $K$ and $G$, and let $\fg=\fp\oplus \fk$ be the Cartan decomposition. Let $B$\index{B@$B$} be a nondegenerate bilinear symmetric form
on $\fg$ which is invariant under the adjoint action of $G$ and $\theta$. Assume that $B$ is positive on $\fp$ and negative on
$\fk$. Set $X = G/K$.\index{X@$X$} Then $B$ induces a Riemannian metric $g^{TX}$ on the tangent bundle $TX = G \times_K \fp$, such that
$X$ is of nonpositive sectional curvature.

Let $\Gamma\subset G$\index{G@$\Gamma$} be a discrete torsion-free 
cocompact subgroup of $G$. Set $Z=\Gamma\backslash X$.  Then $Z$ is a 
closed locally symmetric manifold with $\pi_1(Z)=\Gamma$. Recall that 
$\rho:\Gamma\to \mathrm{U}(r)$ \index{R@$\rho$} is a unitary 
representation of $\Gamma$, and that $F$ \index{F@$F$} is the  
unitarily flat vector bundle on $Z$ associated with $\rho$.  The main 
result of this article gives the solution of the Fried conjecture for 
$Z$. In particular, this conjecture is valid for all the closed locally symmetric space of the noncompact type.

\begin{thm}\label{thm:01}
Assume $\dim Z$ is odd. The dynamical zeta function  $R_\rho(\sigma)$ 
is  holomorphic for $\Re(\sigma)\gg1$ and extends meromorphically to 
$\sigma\in \bC$. Moreover, there exist explicit constants $C_\rho\in \bR^{*}$ and $r_\rho\in \mathbf{Z}$ (c.f. \eqref{eq:Cr}) such that, when  $\sigma\to 0$,
\begin{align}\label{eq:t01}
  R_\rho(\sigma)=C_\rho T(F)^2\sigma^{r_\rho}+\cO\(\sigma^{r_\rho+1}\).
\end{align}
If $H^\cdot(Z,F)=0$, then
\begin{align}\label{eq:t02}
  &C_\rho=1,&r_\rho=0,
\end{align}
so that
\begin{align}
R_\rho(0)=T(F)^2.
\end{align}
\end{thm}
Let $\delta(G)$ \index{D@$\delta(G)$} be the nonnegative integer  
defined by the difference between the complex ranks of $G$ and $K$. 
Since $\dim Z$ is odd, $\delta(G)$ is odd.
If $\delta(G)\neq1$, Theorem \ref{thm:01} is originally due to 
Moscovici-Stanton \cite{MStorsion} and was recovered by Bismut \cite{B09}. 
Indeed, it was proved in \cite[Corollary 2.2, Remark 3.7]{MStorsion} 
or  \cite[Theorem 7.9.3]{B09} that $T(F) = 1$ and  $\chi_{\rm orb}\(\mathbb{S}^1\backslash B_{[\gamma]}\)=0$ for all $[\gamma]\in [\Gamma]- \{1\}$.

Remark that both of the above two proofs use the Selberg trace 
formula. However, in the evaluation of the geometric side of the 
Selberg trace formula and of orbital integrals, Moscovici-Stanton 
relied on Harish-Chandra's Plancherel theory, while Bismut  used his explicit formula  \cite[Theorem 6.1.1]{B09} obtained via  the hypoelliptic Laplacian.

Our proof of Theorem \ref{thm:01}  relies on Bismut's formula.

%
%
%

\subsection{Our results on $R_\rho(\sigma)$}
%
Assume that  $\delta(G)=1$. To show that $R_\rho(\sigma)$ extends as 
a meromorphic function on $\bC$ when $Z$ is hyperbolic, Fried  \cite{FriedRealtorsion} showed that $R_\rho(\sigma)$ is  an alternating product of certain Selberg zeta functions.
Moscovici-Stanton's idea was to introduce the more general Selberg zeta functions and to get a similar formula for $R_\rho(\sigma)$.

Let us recall some facts about reductive group $G$ with $\delta(G)=1$. In this case, there exists a unique (up to conjugation) standard parabolic subgroup $Q\subset G$ with Langlands decomposition $Q=M_QA_QN_Q$ such that $\dim  A_Q=1$. Let $\fm, \fb, \fn$ be the  Lie algebras of $M_Q, A_Q, N_Q$.
Let $\alpha\in \fb^*$ be  such that for $a\in \fb$, $\ad(a)$ acts on 
$\fn$ as a scalar $\<\alpha,a\>\in \bR$ (c.f. Proposition \ref{prop:nnam}).
Let $M$ be the connected component of identity of $M_Q$.  Then $M$ is 
a connected reductive group with  maximal compact subgroup $K_M=M\cap K$ and with Cartan decomposition $\fm=\fp_\fm\oplus \fk_\fm$. We have an identity of real $K_M$-representations
\begin{align}\label{eq:introdp}
  \fp\simeq\fp_\fm\oplus\fb \oplus \fn.
\end{align}

An observation  due to Moscovici-Stanton is that $\chi_{\rm orb}(\mathbb{S}^1\backslash B_{[\gamma]})\neq0$ only if $\gamma$ can be conjugated by an element of $G$ into $A_QK_M$. For $\sigma\in \bC$, we define the formal Selberg zeta function by
\begin{align}
Z_{j}(\sigma)=\exp\(-\sum_{[\gamma]\in 
[\Gamma]-\{1\}}\Tr[\rho(\gamma)]\frac{\chi_{\rm 
orb}(\mathbb{S}^1\backslash 
B_{[\gamma]})}{m_{[\gamma]}}\frac{\Tr^{\Lambda^j(\fn^*)}\[\Ad(k^{-1})\]}{\big|\det(1-\Ad(e^ak^{-1}))|_{\fn\oplus \theta\fn}\big|^{1/2}}e^{-\sigma l_{[\gamma]}}\),
\end{align}
where $a\in \fb,k\in K_M$ are such that $\gamma$ can be conjugated to 
$e^ak^{-1}$. We remark that $l_{[\gamma]}=|a|$.
To show the meromorphicity of $Z_{j}(\sigma)$, Moscovici-Stanton tried to identify $Z_{j}(\sigma)$ with the geometric side of the zeta regularized determinant of the resolvent of some elliptic operator acting on some vector bundle on $Z$.  However, the vector bundle used in \cite{MStorsion}, whose construction involves the adjoint representation of $K_M$ on $\Lambda^i(\fp_\fm^*)\otimes \Lambda^i(\fn^*)$, does not live on $Z$, but only on $\Gamma\backslash G/K_{M}$.

We complete this gap  by showing that such an
objet exists as a virtual vector bundle on $Z$ in the sense of $K$-theory. More precisely,
let $RO(K), RO(K_M)$ \index{R@$RO(K_M),RO(K)$} be the real 
representation rings  of $K$ and $K_M$. We can verify that the 
restriction $RO(K)\to RO(K_M)$ is injective. Note that  
$\fp_\fm,\fn\in RO(K_M)$. In Subsection \ref{sec:varsi}, using the classification theory of real simple Lie algebras, we show  $\fp_\fm,\fn$ are in the image of $RO(K)$.  For $0\l j \l \dim \fn$,  let $E_j=E^+_j-E^-_j\in RO(K)$ such that the following identity in  $RO(K_M)$ holds:
\begin{align}\label{eq:in22}
 \Big(\sum^{\dim \fp_\fm}_{i=0} (-1)^i\Lambda^i(\fp_\fm^*)\Big)\otimes \Lambda^j(\fn^*)=E_j|_{K_M}.
\end{align}

Let $\cE_j=G\times_K E_j$ be a $\mathbf{Z}_2$-graded  vector bundle 
on $X$. It descends to a $\mathbf{Z}_2$-graded  vector bundle $\cF_j$ 
on $Z$. Let $C_j$ be a  Casimir operator of $G$ action on 
$C^\infty(Z,\cF_j\otimes_\bR F)$. In Theorem \ref{thm:detfor}, we 
show that there are $\sigma_j\in \bR$ and  an odd polynomial $P_j$ 
such that if $\Re(\sigma)\gg 1$,  $Z_j(\sigma)$ is holomorphic and
\begin{align}\label{eq:detfoin}
 Z_{j}(\sigma)={\rm det}_{\rm gr}\(C_j+\sigma_j+\sigma^2\)\exp\big(r\vol(Z) P_j(\sigma)\big),
\end{align}
where $\det_{\rm gr}$ is the zeta regularized $\mathbf{Z}_2$-graded determinant.
In particular, $Z_{j}(\sigma)$ extends meromorphically to $\bC$. 

By a direct calculation of linear algebra, we have
\begin{align}\label{eq:RbyZin}
  R_\rho(\sigma)=\prod_{j=0}^{\dim \fn}Z_{j}\(\sigma+\big(j-\frac{\dim\fn}{2}\big)|\alpha|\)^{(-1)^{j-1}},
\end{align}
from which we get the meromorphic extension of $R_\rho(\sigma)$.
Note that the meromorphic function
\begin{align}
T(\sigma)=\prod_{i=1}^{\dim Z}\det\(\sigma+\Box^Z|_{\Omega^i(Z,F)}\)^{(-1)^ii}\index{T@$T(\sigma)$}
\end{align}
has a Laurent expansion near $\sigma=0$,
\begin{align}\label{eq:RbyZinsss}
T(\sigma)=T(F)^2\sigma^{\chi'(X,F)}+\cO(\sigma^{\chi'(X,F)+1}),
\end{align}
where $\chi'(X,F)$ is the derived Euler number (c.f. \eqref{eq:Eulderi}). 
Note also that the Hodge Laplacian $\Box^Z$ coincides with the Casimir operator acting on $\Omega^\cdot(Z,F)$.
The Laurent expansion \eqref{eq:t01} can be deduced form 
\eqref{eq:detfoin}-\eqref{eq:RbyZinsss}  and the identity in $RO(K)$,
\begin{align}
\sum_{i=1}^{\dim \fp}(-1)^{i-1}i\Lambda^i(\fp^*)=\sum^{\dim \fn}_{j=0} (-1)^jE_j.
\end{align}

\subsection{Proof of Equation \eqref{eq:t02}}
To understand how the acyclicity of $F$ is reflected in the function $R_\rho (\sigma)$, we need some deep results of representation theory. Let $\widehat{p}: \Gamma\backslash G \to Z$ be the natural projection.
The enveloping algebra of $U(\fg)$ acts on $C^\infty(\Gamma\backslash G, \widehat{p}^*F)$.
Let $\cZ(\fg)$ be the center of $U(\fg)$.
 Let $V^\infty\subset C^\infty(\Gamma\backslash G,\widehat{p}^*F)$ be the subspace of $C^\infty(\Gamma\backslash G,\widehat{p}^*F)$ on which the action of $\cZ(\fg)$ vanishes, and let $V$ be the closure of $V^\infty$ in $L^2(\Gamma\backslash G, \widehat{p}^*F)$.
Then $V$ is a unitary representation of $G$.
The compactness  of $\Gamma\backslash G$ implies that $V$ is
a finite sum of irreducible unitary representations of $G$.
By  standard arguments \cite[Ch VII, Theorem 3.2, Corollary 3.4]{BW}, the cohomology $H^\cdot(Z,F)$ is canonically isomorphic to  the $(\fg,K)$-cohomology
$H^\cdot(\fg,K;V)$ of $V$.

 Vogan-Zuckerman \cite{VoganZuckerman} and Vogan \cite{Vogan2} classified all irreducible unitary representations with nonzero  $(\fg,K)$-cohomology. On the other hand, Salamanca-Riba \cite{Salamanca} showed that any irreducible  unitary representation with vanishing $\cZ(\fg)$-action is in the class specified by  Vogan and Zuckerman, which means that it possesses nonzero $(\fg,K)$-cohomology.

By the above considerations, the acyclicity of $F$ is equivalent to $V=0$.  This is  essentially the algebraic ingredient in the proof of \eqref{eq:t02}. Indeed, in Corollary \ref{cor:forr1}, we give a formula for the constants $C_\rho$ and $r_\rho$, obtained by Hecht-Schmid formula \cite{HechtSchmid} with the help of the $\fn$-homology of $V$.

\subsection{The organization of the article}
This article  is organized as follows. In Sections \ref{sec:inv}, we 
recall the definitions of certain  characteristic forms and of the analytic torsion.


In Section \ref{Sec:Pred}, we introduce the reductive groups and the fundamental rank $\delta(G)$ of $G$.

In Section \ref{sec:sel},  we introduce the symmetric space. We recall basic principles for the Selberg trace formula, and we state formulas by Bismut \cite[Theorem 6.1.1]{B09} for semisimple orbital integrals. We recall the proof given by Bismut \cite[Theorem 7.9.1]{B09} of a vanishing result of the analytic torsion $T(F)$ in the case $\delta(G)\neq1$, which is originally due to Moscovici-Stanton \cite[Corollary 2.2]{MStorsion}.

In Section \ref{sec:dyn}, we introduce the dynamical zeta function 
$R_\rho(\sigma)$, and we state  Theorem \ref{thm:01} as Theorem 
\ref{thm:01z}. We prove Theorem \ref{thm:01} when $\delta(G)\neq 1$ or 
when $G$ has 
noncompact center. 

%

Sections \ref{sec:5}-\ref{sec:proofthm1} are devoted to establish 
Theorem \ref{thm:01}  when $G$ has compact center and when $\delta(G)=1$.  

In Section \ref{sec:5}, we introduce   geometric objects 
associated with such reductive groups $G$.

In Section \ref{Sec:qusi},  we introduce  Selberg zeta 
functions, and we prove that $R_{\rho}(\sigma)$ extends 
meromorphically, and we establish  Equation \eqref{eq:t01}. 
%
%

Finally, in Section
\ref{sec:proofthm1}, after recalling some constructions and results of representation theory,
we prove that \eqref{eq:t02} holds.

In the whole paper, we use the superconnection formalism of Quillen \cite{Quillensuper} and \cite[Section 1.3]{BGV}.
If $A$ is a $\bZ_2$-graded algebra, if $a,b\in A$,  the supercommutator $[a, b]$ is
given by
 \begin{align}
   [a,b]=ab-(-1)^{\deg a \deg b}ba.
 \end{align}
If $B$ is another $\bZ_2$-graded algebra, we denote by 
$A\widehat{\otimes} B$ the super tensor product algebra of $A$ and $B$.
If $E = E^+ \oplus E^-$ is a $\bZ_2$-graded vector space, the algebra $\End(E)$ is $\bZ_2$-graded. If $\tau = \pm 1$ on
$E^\pm$, if $a \in \End(E)$, the supertrace $\Trs[a]$ is defined by
\begin{align}
 \Trs[a]=\Tr[\tau a].
\end{align}
We make the convention that $\mathbf{N}=\{0,1, 2,\cdots\}$ and $\mathbf{N}^*=\{1,2,\cdots\}$.

\subsection{Acknowledgement}
I am deeply grateful to Prof. J.-M. Bismut for sharing his insights on this problem. Without his help,
this paper would never have been written. I am indebted to Prof. W. M\"uller for his detailed explanation
on the paper \cite{MStorsion}. I would also like to thank Prof. L. Clozel and Prof. M. Vergne for very useful discussions.
 I appreciate the encouragement and constant support of Prof. X. Ma. I am also indebted to  referees  who read the manuscript very carefully and offered detailed suggestions for its improvement.

This article was prepared while I was visiting Max-Planck-Institut f\"ur Mathematik at Bonn and the
Humboldt-Universit\"at zu Berlin. I would like to thank both institutions for their hospitality.

This research has been financially supported by Collaborative Research Centre ``Space-Time-Matte'' (SFB 647), funded by the German Research Foundation (DFG).

\settocdepth{subsection}
\section{Characteristic forms and analytic torsion}\label{sec:inv}
The purpose of this section is to recall some basic constructions and properties  of characteristic forms and the analytic torsion.  

This section is organized as follows. 
In Subsection \ref{sec:Euler}, we recall the construction of the 
Euler  form, the $\widehat{A}$-form and the Chern character form.

In Subsection \ref{sec:regdet}, we introduce the regularized determinant.

Finally, in Subsection \ref{sec:anator}, we recall the definition of the analytic torsion of flat vector bundles.

\subsection{Characteristic forms}\label{sec:Euler}
If $V$ is a real or complex vector space of dimension $n$, we denote 
by $V^*$ the dual space and by 
$\Lambda^\cdot(V)=\sum_{i=0}^n\Lambda^i(V)$ its exterior algebra.
Let $Z$ be a  smooth manifold. If $V$ is a vector bundle on $Z$, we 
denote by $\Omega^\cdot(Z,V)$ the space of smooth differential forms 
with values in $V$. When $V=\bR$, we write $\Omega^\cdot(Z)$ instead.



%

Let $E$ be a real Euclidean vector bundle of rank $m$ with a metric 
connection $\nabla^E$. Let $R^E=\nabla^{E,2} $ be the curvature of 
$\nabla^{E}$. It is a $2$-form with values in antisymmetric endomorphisms of $E$.

If $A$ is an antisymmetric matrix, denote by $\mathrm{Pf}[A]$ the 
Pfaffian \cite[eq. (3.3)]{BZ92} of $A$. Then $\mathrm{Pf}[A]$ is a polynomial function of $A$, which is a square root of $\det[A]$.
Let $o(E)$ be the orientation line of $E$.  The Euler  form $e\(E,\nabla^{E}\)$ of $\(E,\nabla^E\)$ is given by
\begin{align}\label{eq:eE}
  e\(E,\nabla^{E}\)=\mathrm{Pf}\[\frac{R^{E}}{2\pi}\]\in \Omega^{m}\big(Z,o(E)\big). \index{E@$e(E,\nabla^E)$}
\end{align}
 If $m$ is odd, then $e(E,\nabla^{E})=0$.

For $x\in \bC$, set \index{A@$\widehat{A}$}
\begin{align}\label{eq:defAhat}
  \widehat{A}(x)=\frac{x/2}{\sinh(x/2)}.
\end{align}
The form $\widehat{A}\(E,\nabla^E\)$ of $\(E,\nabla^E\)$ is given by\index{A@$\widehat{A}(E,\nabla^E)$}
\begin{align}\label{eq:Ahatg}
  \widehat{A}\(E,\nabla^E\)=\[\det\(\widehat{A}\(-\frac{R^E}{2i\pi}\)\)\]^{1/2}\in \Omega^\cdot(Z).\end{align}

If $E'$ is a complex Hermitian vector bundle equipped with a metric connection 
$\nabla^{E'}$  with curvature $R^{E'}$, the Chern character form $\mathrm{ch}\(E',\nabla^{E'}\)$ of $(E',\nabla^{E'})$ is given by
\begin{align}\label{eq:chg}
  \mathrm{ch}\(E',\nabla^{E'}\)=\Tr\[\exp\(-\frac{R^{E'}}{2i\pi}\)\]\in \Omega^\cdot(Z).\index{C@$\mathrm{ch}(E',\nabla^{E'})$}
\end{align}

The differential forms $e\(E,\nabla^E\), \widehat{A}\(E,\nabla^E\)$ 
and $\mathrm{ch}\(E',\nabla^{E'}\)$ are closed. They are the 
Chern-Weil representatives  of the Euler class of $E$, the $\widehat{A}$-genus of $E$ and the Chern character of $E'$.

\subsection{Regularized determinant}\label{sec:regdet}
Let $(Z,g^{TZ})$ be a smooth  closed Riemannian manifold of dimension $m$. Let $(E,g^{E})$ be
a Hermitian vector bundle on $Z$. The metrics $g^{TZ}$, $g^E$ induce an $L^2$-metric  on $C^\infty(Z,E)$.

Let $P$ be a second order elliptic differential operator acting on $C^\infty(Z,E)$.  Suppose that $P$ is  formally self-adjoint and nonnegative.  Let $P^{-1}$ be the inverse of $P$ acting on the orthogonal space to $ \ker(P)$. For $\Re(s)>m/2$, set
\begin{align}\label{eq:zetaP}
\theta_P(s)=-\Tr\[\(P^{-1}\)^s\].\index{T@$\theta_P(s)$}
\end{align}
By \cite{Seeley66} or \cite[Proposition 9.35]{BGV}, $\theta(s)$ has a meromorphic extension to $s\in \bC$ which is holomorphic at $s=0$.  The regularized determinant  of $P$ is defined as
\begin{align}\label{eq:Mch}
\det(P)=\exp\big(\theta'_P(0)\big).\index{D@$\det(P)$}
\end{align}

Assume now that $P$ is  formally self-adjoint and bounded from below. Denote by $\Sp(P)$ the spectrum  of $P$.
For $\lambda\in \Sp(P)$, set
\begin{align}
m_P(\lambda)=\dim_{\bC} \ker(P-\lambda)\index{M@$m_P(\lambda)$}
\end{align}
its multiplicity.  If  $\sigma\in \bR$ is such that $P + \sigma> 0$, 
then $\det(P + \sigma)$\index{D@$\det(P+\sigma)$} is defined by 
\eqref{eq:Mch}. Voros \cite{Voros} has shown that the function $\sigma\to \det(P + \sigma)$, defined for $\sigma\gg1$, extends holomorphically  to $\bC$ with zeros at $\sigma=-\lambda$ of the order $m_{P}(\lambda)$, where $\lambda\in \Sp(P)$.

\subsection{Analytic torsion}\label{sec:anator}
Let $Z$ \index{Z@$Z$} be a smooth connected closed manifold of 
dimension $m$ \index{M@$m$} with fundamental group $\Gamma$.\index{G@$\Gamma$}
Let $F$ \index{F@$F$} be a complex flat vector bundle on $Z$ of rank $r$. \index{R@$r$} Equivalently, $F$ can be obtained via a complex representation $\rho:\Gamma\to \GL_r(\bC)$.\index{R@$\rho$}

Let $H^\cdot(Z,F)=\oplus_{i=0}^m H^i(Z,F)$ \index{H@$H^\cdot(Z,F)$} be the cohomology of the sheaf of locally flat
sections of $F$.
We define the Euler number and the derived Euler number by
\begin{align}\label{eq:Eulderi}
&\chi(Z,F)=\sum_{i=0}^m(-1)^i\dim_{\bC}H^i(Z,F),&  \chi'(Z,F)=\sum_{i=1}^m(-1)^ii\dim_\bC H^i(Z,F).\index{C@$\chi(Z,F),\chi'(Z,F)$}
\end{align}

Let $\(\Omega^\cdot(Z,F),d^Z\)$ \index{O@$\Omega^\cdot(Z,F)$} \index{D@$d^Z$} be the de Rham complex of smooth sections of $\Lambda^\cdot(T^*Z)\otimes_\bR F$ on $Z$. We have a canonical isomorphism of vector spaces
\begin{align}
  H^\cdot\(\Omega^\cdot(Z,F),d^Z\)\simeq H^\cdot(Z,F).
\end{align}
In the sequel, we will also consider the trivial line bundle $\bR$. We denote simply by $H^\cdot(Z)$ and
$\chi(Z)$ \index{C@$\chi(Z)$} the corresponding objects. Note that, in this case, the complex dimension  in  \eqref{eq:Eulderi} should be replaced by the real dimension.

Let $g^{TZ}$ \index{G@$g^{TZ}$} be a Riemmanian  metric  on $TZ$, and let $g^F$ \index{G@$g^{F}$} be a Hermitian metric on $F$. They induce an $L^2$-metric $\<\, ,\>_{\Omega^\cdot(Z,F)}$ on  $\Omega^\cdot(Z,F)$. Let $d^{Z,*}$ be the formal adjoint of $d^Z$ with respect to $\<\, ,\>_{\Omega^\cdot(Z,F)}$. Put
\begin{align}\label{eq:DiracHodge}
  &D^Z=d^Z+d^{Z,*},&\Box^Z=D^{Z,2}=\[d^Z,d^{Z,*}\].\index{D@$d^{Z,*}$}\index{D@$D^Z$}  \end{align}
Then,  $\Box^Z$ \index{B@$\Box^Z$}
 is a  formally self-adjoint nonnegative second order elliptic operator acting on  $\Omega^\cdot(Z,F)$. 
By  Hodge theory, we have a canonical isomorphism of vector spaces
\begin{align}\label{eq:Hodge}
  \ker \Box^Z\simeq H^\cdot(Z,F).
\end{align}
%

\begin{defin}
  The analytic torsion of $F$ is a positive real number defined by
\begin{align}
   T\(F,g^{TZ},g^F\)=\prod_{i=1}^{m}\det\(\Box^Z|_{\Omega^i(Z,F)}\)^{(-1)^ii/2}.
\end{align}
\end{defin}

Recall that the flat vector bundle $F$ carries a flat metric $g^{F}$ 
if and only if the holonomy representation $\rho$ factors through 
$\mathrm{U}(r)$. In this case, $F$ is said to be unitarily flat.  If $Z$ is an even dimensional orientable manifold and if $F$ is unitarily flat with a flat metric $g^{F}$, by Poincar\'e duality,  $T(F,g^{TZ},g^F)=1$. If $\dim Z$ is odd and if $H^\cdot(Z,F)=0$, by \cite[Theorem 4.7]{BZ92}, then $T(F,g^{TZ},g^F)$ does not depend on  $g^{TZ}$ or $g^{F}$.
 In the sequel,  we write instead
\begin{align}\label{eq:tauZF}
  T(F)=T\(F,g^{TZ},g^F\).\index{T@$T(F)$}
\end{align}

By  Subsection \ref{sec:regdet},
\begin{align}\label{eq:Tsigma}
T(\sigma)=\prod_{i=1}^{\dim Z}\det\(\sigma+\Box^Z|_{\Omega^i(Z,F)}\)^{(-1)^ii}\index{T@$T(\sigma)$}
\end{align}
is meromorphic on $\mathbf{C}$.  When $\sigma\to 0$, we have
\begin{align}\label{eq:tors}
T(\sigma)=T(F)^2\sigma^{\chi'(Z,F)}+\cO(\sigma^{\chi'(Z,F)+1}).
\end{align}

\section{Preliminaries on reductive groups}\label{Sec:Pred}
The purpose of this section is to recall some basic facts about reductive groups. 

This section is organized as follows. In Subsection  \ref{sec:red}, we introduce the  reductive group $G$.  

In Subsection \ref{sec:semi}, we introduce the semisimple elements of $G$, and we recall some related constructions.

In Subsection \ref{sec:cart}, we  recall  some properties of Cartan subgroups of $G$. We introduce a nonnegative integer $\delta(G)$, which has paramount importance in the whole article. We  also recall Weyl's integral formula on reductive groups. 

Finally, in Subsection \ref{sec:regular}, we introduce the regular elements of $G$. 


\subsection{The reductive group}\label{sec:red}
Let $G$\index{G@$G$} be a linear connected real reductive group \cite[p. 3]{Knappsemi}, that means $G$ is a closed connected group of real matrices that is stable under transpose.  Let $\theta\in \mathrm{Aut}(G)$ \index{T@$\theta$} be the Cartan involution.  Let $K$ \index{K@$K$} be the maximal compact subgroup of $G$ of the points of $G$
that are fixed by $\theta$.

Let $\fg$ $\index{G@$\fg$}$ be the Lie algebra of $G$, and let 
$\fk\subset \fg$ \index{K@$\fk$} be the Lie algebra of $K$. The Cartan involution $\theta$ acts naturally as a Lie algebra automorphism of $\fg$. Then $\fk$ is the eigenspace of $\theta$ associated with the eigenvalue $1$. Let $\fp$ \index{P@$\fp$} be the eigenspace with the eigenvalue $-1$, so that
\begin{align}\label{eq:cartan1}
  \fg=\fp\oplus\fk.
\end{align}
Then we have
\begin{align}\label{eq:kpkp}
  &[\fk,\fk]\subset \fk, &[\fp,\fp]\subset \fk,&& [\fk,\fp]\subset \fp.
\end{align}
Put
\begin{align}\label{eq:mpnk}
  &m=\dim\fp,&n=\dim\fk.\index{M@$m$}\index{N@$n$}
\end{align}
By \cite[Proposition 1.2]{Knappsemi}, we have the diffeomorphism
\begin{align}\label{eq:cartan2}
 (Y,k)\in \fp\times K\to  e^Y k\in G.
\end{align}

Let $B$ \index{B@$B$} be a real-valued  nondegenerate bilinear symmetric form on $\fg$ which is invariant under the adjoint action $\Ad$ of $G$ on $\fg$, and also under $\theta$. Then \eqref{eq:cartan1} is an orthogonal splitting of $\fg$ with respect to $B$. We assume $B$ to be positive on $\fp$, and negative on $\fk$. The form $\<\cdot,,\cdot\>=-B(\cdot,\theta\cdot)$ defines an $\Ad(K)$-invariant scalar product on $\fg$ such that the splitting \eqref{eq:cartan1} is still orthogonal. We denote by $|\cdot|$ \index{1@$\lvert\cdot\rvert$}the corresponding norm.

Let $Z_G\subset G$\index{Z@$Z_G$} be the center of $G$ with Lie algebra $\fz_\fg\subset \fg$.\index{Z@$\fz_\fg$}
Set
\begin{align}\label{eq:zzpzk}
&  \fz_{\fp}=\fz_{\fg}\cap \fp,&\fz_{\fk}=\fz_{\fg}\cap \fk.\index{Z@$\fz_{\fp},\fz_{\fk}$}
\end{align}
By \cite[Corollary 1.3]{Knappsemi}, $Z_G$ is reductive. As in \eqref{eq:cartan1} and \eqref{eq:cartan2}, we have the Cartan decomposition
\begin{align}\label{eq:ZG}
&\fz_\fg=\fz_{\fp}\oplus \fz_\fk,& Z_G=\exp (\fz_\fp) (Z_G\cap K).
\end{align}

Let $\fg_\bC=\fg\otimes_\bR\bC$ \index{G@$\fg_\bC$}be the complexification of $\fg$ and let $\fu=\sqrt{-1}\fp\oplus \fk$\index{U@$\fu$}
be the compact form of $\fg$. Let $G_\bC$ \index{G@$G_\bC$} and 
$U$\index{U@$U$} be the connected group of complex matrices 
associated with the Lie algebras $\fg_\bC$ and $\fu$.
By \cite[Propositions 5.3 and 5.6]{Knappsemi}, if $G$ has compact center, i.e., its center $Z_G$ is compact, then $G_\bC$ is a linear connected complex reductive group with maximal compact subgroup $U$.


Let $U(\fg)$ \index{U@$U(\fg)$} be the enveloping algebra of $\fg$,  and let $\mathcal{Z}(\fg)\subset U(\fg)$ \index{Z@$\mathcal{Z}(\fg)$} be the center of $U(\fg)$.
Let $C^\fg\in U(\fg)$ be the Casimir element. If $e_1,\cdots,e_m$ is an orthonormal basis of $\fp$, and if $e_{m+1},\cdots,e_{m+n}$ is an  orthonormal basis of $\fk$, then
\begin{align}\label{eq:Cg}
  C^\fg=-\sum_{i=1}^{m}e^2_i+\sum_{i=m+1}^{n+m}e_i^{2}.\index{C@$C^{\fg}$}
\end{align}
Classically, $C^\fg\in \mathcal{Z}(\fg)$.

We define $C^{\fk}$\index{C@$C^{\fk}$} similarly. Let $\tau$ \index{T@$\tau$}be a finite dimensional representation of $K$ on $V$. We denote by $C^{\fk,V}$ or $C^{\fk,\tau}\in \End(V)$ the corresponding Casimir operator acting on $V$, so that
\begin{align}\label{eq:ckkckp}
C^{\fk,V}=C^{\fk,\tau}=\sum_{i=m+1}^{m+n}\tau(e_i)^2.\index{C@$C^{\fk,V},C^{\fk,\tau}$}
\end{align}




\subsection{Semisimple elements}\label{sec:semi}
If $\gamma\in G$, we denote by $Z(\gamma)\subset G$ \index{Z@$Z(\gamma)$} the centralizer of $\gamma$ in $G$, and by $\fz(\gamma)\subset \fg$ \index{Z@$\fz(\gamma)$}its Lie algebra. If $a\in \fg$, let $Z(a)\subset G$  \index{Z@$Z(a)$} be the stabilizer of $a$ in $G$, and let $\fz(a)\subset \fg$ \index{Z@$\fz(a)$} be its Lie algebra. 

An element $\gamma \in G$ is said to be semisimple if $\gamma$ can be conjugated to $e^ak^{-1}$ such that
\begin{align}\label{eq:asr}
   &a\in \fp, &k\in K, &&\Ad(k)a=a.
\end{align}
%
Let $\gamma=e^ak^{-1}$ such that \eqref{eq:asr} holds.  By \cite[eq. (3.3.4), (3.3.6)]{B09}, we have
\begin{align}\label{eq:Zr}
 & Z(\gamma)=Z(a)\cap Z(k),&\fz(\gamma)=\fz(a)\cap \fz(k).
\end{align}
Set
\begin{align}\label{eq:prkr}
 & \fp(\gamma)=\fz(\gamma)\cap\fp,& \fk(\gamma)=\fz(\gamma)\cap \fk.\index{P@$\fp(\gamma)$}\index{K@$\fk(\gamma)$}
\end{align}
From \eqref{eq:Zr} and \eqref{eq:prkr}, we get
\begin{align}\label{eq:fzr}
  \fz(\gamma)=\fp(\gamma)\oplus\fk(\gamma).
\end{align}


By \cite[Proposition 7.25]{KnappLie}, $Z(\gamma)$ is a reductive  subgroup of $G$ with maximal compact subgroup $K(\gamma)=Z(\gamma)\cap K$, and with Cartan decomposition  \eqref{eq:fzr}.  Let $Z^0(\gamma)$ be the connected component of the identity in $Z(\gamma)$. Then $Z^0(\gamma)$ \index{Z@$Z^0(\gamma)$} is a reductive subgroup of $G$, with maximal compact subgroup $Z^0(\gamma)\cap K$. Also, $ Z^0(\gamma)\cap K$ coincides with $K^0(\gamma)$,  the connected component of the identity in $K(\gamma)$.

An element $\gamma\in G$ is said to be elliptic if $\gamma$ is conjugated to an element of $K$. Let $\gamma\in G$ be semisimple and nonelliptic. Up to conjugation, we can assume $\gamma=e^ak^{-1}$ such that \eqref{eq:asr} holds and that $a\neq0$.
By \eqref{eq:Zr}, $a\in \fp(\gamma)$.
Let $\fz^{a,\bot}(\gamma), \fp^{a,\bot}(\gamma)$ 
\index{Z@$\fz^{a,\bot}(\gamma)$}\index{P@$\fp^{a,\bot}(\gamma)$} be 
respectively the orthogonal spaces to $a$ in $\fz(\gamma), 
\fp(\gamma)$, so that
\begin{align}\label{eq:zabot}
  \fz^{a,\bot}(\gamma)=\fp^{a,\bot}(\gamma)\oplus \fk(\gamma).
\end{align}
Moreover, $ \fz^{a,\bot}(\gamma)$ is a Lie algebra. Let $Z^{a,\bot,0}(\gamma)$ be the connected subgroup of $Z^0(\gamma)$ that is associated with the Lie algebra $\fz^{a,\bot}(\gamma)$.
By \cite[eq. (3.3.11)]{B09}, $Z^{a,\bot,0}(\gamma)$ is reductive with maximal compact subgroup $K^0(\gamma)$ with Cartan decomposition \eqref{eq:zabot}, and
\begin{align}\label{eq:Z0a}
  Z^0(\gamma)=\mathbf{R}\times Z^{a,\bot,0}(\gamma),
\end{align}
so that $e^{ta}$ maps into $t|a|$.


\subsection{Cartan subgroups}\label{sec:cart}
A  Cartan subalgebra of $\fg$ is a maximal  abelian subalgebra of $\fg$. A Cartan subgroup of $G$ is  the centralizer of a Cartan subalgebra.

By \cite[Theorem 5.22]{Knappsemi}, there is only a finite number of 
nonconjugate (via $K$) $\theta$-stable Cartan subalgebras $\fh_1$, $\cdots$, $\fh_{l_0}$.\index{H@$\fh_i,H_i$} \index{L@$l_0$} Let $H_1$, $\cdots$, $H_{l_0}$ be the corresponding Cartan subgroup. Clearly, the Lie algebra of $H_i$ is $\fh_i$.  Set
\begin{align}
&  \fh_{i\fp}=\fh_i\cap \fp,&\fh_{i\fk}=\fh_i\cap \fk.\index{H@$ \fh_{i\fp}, \fh_{i\fk}$}
\end{align}
We call $\dim \fh_{i\fp}$ the noncompact dimension of $\fh_i$.
By \cite[Theorem 5.22 (c)]{Knappsemi} and \cite[Proposition 7.25]{KnappLie}, $H_i$ is an abelian  reductive group with maximal compact subgroup $ H_i\cap K$, and with  Cartan decomposition
\begin{align}\label{eq:Hcartan}
  &\fh_i=\fh_{i\fp}\oplus \fh_{i\fk}, &H_i=\exp(\fh_{i\fp})( H_i\cap K).
\end{align}
 Note that in general, $H_i$ is not necessarily  connected.

Let $ W(H_i,G)$ be the  Weyl group. If $N_K(\fh_i)\subset K$ and $Z_K(\fh_i)\subset K$ are   the normalizer and centralizer of $\fh_i$ in $K$, then
\begin{align}\label{eq:Weylgroup}
  W(H_i,G)=N_K(\fh_i)/Z_K(\fh_i).\index{W@$W(H_i,G)$}
\end{align}

In the whole paper, we fix a maximal torus $T$ \index{T@$T$} of $K$. Let $\ft\subset \fk$ be the Lie algebra of $T$. Set
\begin{align}\label{eq:defb}
  \fb=\{Y\in \fp:[Y, \ft]=0\}.\index{B@$\fb$}
\end{align}
By \eqref{eq:zzpzk} and \eqref{eq:defb}, we have 
\begin{align}\label{eq:defbzp11}
\fz_{\fp}\subset\fb.
\end{align}

Put
\begin{align}\label{eq:carhH}
  \fh=\fb\oplus \ft.\index{H@$\fh$}
\end{align}
By \cite[Theorem 5.22 (b)]{Knappsemi}, $\fh$ is the $\theta$-stable Cartan subalgebra of $\fg$ with  minimal noncompact dimension.
Also, every  $\theta$-stable Cartan subalgebra  with minimal 
noncompact dimension is conjugated  to $\fh$ by an element of $K$. Moreover, the corresponding Cartan subgroup $H\subset G$ of $G$ is connected, so that
\begin{align}\label{eq:H=BT}
  H=\exp(\fb)T.
\end{align}
We may  assume that $\fh_1=\fh$ and $H_1=H$.

Note that the complexification $\fh_{i\bC}=\fh_i\otimes_\bR\bC$  of $\fh_i$ is a Cartan subalgebra of $\fg_\bC$.
All the $\fh_{i\bC}$ are conjugated by  inner automorphisms of $\fg_\bC$. Their common complex dimension $\dim_\bC \fh_{i\bC}$ is called the complex rank $\rk_\bC(G)$ \index{R@$\mathrm{rk}_\bC$}of $G$.

\begin{defin}\label{def:dg}
Put\index{D@$\delta(G)$}
\begin{align}\label{eq:dg}
  \delta(G)=\rk_\bC(G)-\rk_\bC(K)\in \mathbf{N}.
\end{align}
\end{defin}
By \eqref{eq:defb} and \eqref{eq:dg}, we have
\begin{align}\label{eq:dg=db}
  \delta(G)=\dim \fb.
\end{align}
Note that $m-\delta(G)$ is even.
We will see that $\delta(G)$ plays an important role in our article.

\begin{re}
If $\fg$ is a real reductive Lie algebra, then $\delta(\fg)\in \mathbf{N}$\index{D@$\delta(\fg)$} can be defined in the same way as in \eqref{eq:dg=db}. Since $\fg$ is reductive, by \cite[Corolary 1.56]{KnappLie}, we have
\begin{align}\label{eq:g=zgg}
  \fg=\fz_\fg\oplus[\fg,\fg],
\end{align}
where $[\fg,\fg]$ is a semisimple Lie algebra. By \eqref{eq:ZG} and \eqref{eq:g=zgg}, we have
\begin{align}\label{eq:gzgg}
 \delta(\fg)=\dim\fz_\fp+\delta([\fg,\fg]).
\end{align}
\end{re}

%

\begin{prop}\label{prop:dgdz}
   The element $\gamma\in G$ is semisimple if and only if $\gamma$ can be conjugated into $ \cup_{i=1}^{l_0}H_i$. In this case,
\begin{align}\label{eq:dgdz}
  \delta(G)\l \delta\big(Z^0(\gamma)\big).
\end{align}
The two sides of \eqref{eq:dgdz} are equal  if and only if $\gamma$ can be conjugated into $H$.
\end{prop}
\begin{proof}If $\gamma\in H_i$, by the Cartan decomposition \eqref{eq:Hcartan}, there exist
$a\in \fh_{i\fp}$ and $k\in K\cap H_i$ such that  $\gamma=e^ak^{-1}$. 
Since $H_{i}$ is the centralizer of $\fh_{i}$, we have  
$\Ad(\gamma)a=a$. Therefore, $\Ad(k)a=a$, so that  $\gamma$ is semisimple.

Assume that $\gamma\in G$ is semisimple and such that \eqref{eq:asr} holds. We claim that
\begin{align}\label{eq:rk1}
 \rk_\bC(G)=  \rk_{\bC}\big(Z^0(\gamma)\big).
\end{align}
Indeed, let $\fh'\subset \fg$ be a $\theta$-invariant Cartan subalgebra of $\fg$ containing $a$. Then, $\fh'\subset \fz(a)$. It implies
\begin{align}\label{eq:GZ0k}
  \rk_\bC(G)=\rk_\bC\big(Z^0(a)\big).
\end{align}
By choosing a maximal torus $T$ containing $k$, by \eqref{eq:carhH}, we have $\fh\subset \fz(k)$. Then
\begin{align}\label{eq:GZ0K}
  \rk_\bC(G)=\rk_\bC\big(Z^0(k)\big).
\end{align}
If we replace $G$ by $Z^0(a)$ in \eqref{eq:GZ0K}, by \eqref{eq:Zr}, we get
\begin{align}\label{eq:ZZ0k}
  \rk_\bC(Z^0(a))=\rk_\bC\big(Z^0(\gamma)\big).
\end{align}
By \eqref{eq:GZ0k} and \eqref{eq:ZZ0k}, we get \eqref{eq:rk1}.

%
%

Let $\fh(\gamma)\subset \fz(\gamma)$ \index{H@$\fh(\gamma)$} be the 
$\theta$-invariant Cartan subalgebra defined as in \eqref{eq:carhH} when $G$ is replaced by $Z^0(\gamma)$. By \eqref{eq:rk1}, $\fh(\gamma)$ is also a Cartan subalgebra of $\fg$. Moreover, $\gamma$ is an element of the Cartan subgroup of $G$ associated to $\fh(\gamma)$. In particular, $\gamma$ can be conjugated into some $H_i$.

By the minimality of noncompact dimension of $\fh$, we have
\begin{align}\label{eq:bsbr}
 \delta(G)= \dim \fh\cap \fp \l \dim \fh(\gamma)\cap \fp=\delta\big(Z^0(\gamma)\big),
\end{align}
which completes  the proof of \eqref{eq:dgdz}.

It is obvious  that if $\gamma$ can be conjugated into $H$, the 
equality in \eqref{eq:bsbr} holds. If the equality holds in \eqref{eq:bsbr}, by the uniqueness  of the Cartan subalgebra with minimal noncompact dimension, there is $k'\in K$ such that
\begin{align}
  \Ad(k')\fh(\gamma)=\fh,
\end{align}
which  implies that $k'\gamma k^{\prime,-1}\in H$. The proof of our proposition  is completed.
\end{proof}


Now we recall the Weyl integral formula on $G$, which will be used in Section \ref{sec:proofthm1}.
Let $dv_{H_i}$ and $dv_{H_i\backslash G}$ \index{D@$dv_{H_i},dv_{H_i\backslash G}$} be respectively the Riemannian volumes on $H_i$ and $H_i\backslash G$ induced by $-B(\cdot,\theta\cdot)$. By \cite[Theorem 8.64]{KnappLie},
  for a nonnegative measurable function $f$ on $G$, we have
\begin{multline}\label{eq:weylnoncom}
  \int_{g\in G}f(g)dv_G=\sum_{i=1}^{l_0}\frac{1}{|W(H_i,G)|}\\
\int_{\gamma\in H_i}\(\int_{g\in H_i\backslash G}f(g^{-1}\gamma g)dv_{ H_i\backslash G} \)\left|\mathrm{det}\big(1-\Ad(\gamma)\big)|_{\fg/\fh_i}\right|dv_{H_i}.
\end{multline}
\subsection{Regular elements}\label{sec:regular}
For  $0\l j\l m+n-\rk_\bC(G)$, let  $D_j$ be the analytic function on $G$ such that for $t\in \bR$ and $\gamma\in G$, we have
\begin{align}\label{eq:detr}
  \mathrm{det}\big(t+1-\Ad(\gamma)\big)|_{\fg}=t^{\rk_\bC(G)}\bigg(\sum^{m+n-\rk_\bC(G)}_{j=0} D_j(\gamma)t^j\bigg).
\end{align}
If $\gamma\in H_i$, then
\begin{align}
  D_{0}(\gamma)=\mathrm{det}\big(1-\Ad(\gamma)\big)|_{\fg/\fh_i}.
\end{align}

We call $\gamma\in G$ regular if $D_{0}(\gamma)\neq0$. Let $G'\subset G$ \index{G@$G'$} be the subset of regular elements of $G$. Then $G'$ is open in $G$, such that $G- G'$ has zero measure with respect to the Riemannian volume $dv_G$ on $G$ induced by $-B(\cdot,\theta\cdot)$. For $1\l i\l l_0$, set
\begin{align}\label{eq:Mul1}
&H_i'= H_i\cap G', &G_i'=\bigcup_{g\in G} g^{-1}H'_ig.\index{H@$H_i'$}\index{G@$G_i'$}
\end{align}
By \cite[Theorem 5.22 (d)]{Knappsemi}, $G_i'$ is open, and we have the disjoint union
\begin{align}
  G'=\coprod_{1\l i\l l_0} G_i'.
\end{align}


\section{Orbital integrals and Selberg trace formula}\label{sec:sel}
 The purpose of this section is to recall  Bismut's  semisimple orbital integral formula  \cite[Theorem 6.1.1]{B09}
and the Selberg trace formula.
%

This section is organized as follows. In Subsections  \ref{sec:sym},
we introduce  the Riemannian  symmetric space $X=G/K$, and we give a formula for its Euler form.

In Subsection \ref{sec:orb}, we recall the definition of  semisimple orbital integrals.

In Subsection \ref{sec:J},  we recall  Bismut's explicit formula for the semisimple orbital integrals associated to the heat operator of the Casimir element.

In Subsection \ref{sec:Gamma}, we introduce a discrete torsion-free cocompact subgroup $\Gamma$ of $G$. We state the Selberg trace formula. 

Finally, in Subsection \ref{sec:t=0}, we recall  Bismut's proof of a 
vanishing result on the analytic torsion in the case $\delta(G)\neq 1$, which is originally due to Moscovici-Stanton \cite{MStorsion}.

\subsection{The symmetric space}\label{sec:sym}
We use the notation of Section \ref{Sec:Pred}.
 Let $\omega^\fg$ be the canonical left-invariant $1$-form on $G$ with values in $\fg$, and let $\omega^\fp, \omega^\fk$ be
its components in $\fp, \fk$, so that
\begin{align}\label{eq:wgkp}
  \omega^\fg=\omega^\fp+\omega^\fk.\index{O@$ \omega^\fg,\omega^\fp,\omega^\fk$}
\end{align}

Let $X=G/K$ \index{X@$X$} be the  associated symmetric space. Then
\begin{align}
p:G\to X=G/K\index{p@$P$}
\end{align}
is a $K$-principle bundle, equipped with the connection form $\omega^\fk$. By \eqref{eq:kpkp} and \eqref{eq:wgkp}, the curvature of $\omega^\fk$ is given by
\begin{align}\label{eq:Ok}
  \Omega^\fk=-\frac{1}{2}\[\omega^\fp,\omega^\fp\].\index{O@$\Omega^\fk$}
\end{align}

Let $\tau$ \index{T@$\tau$} be a finite dimensional orthogonal representation of $K$ on the real Euclidean space $E_\tau$.\index{E@$E_\tau$}
Then  $\cE_\tau=G\times_K E_\tau$\index{E@$\cE_\tau$}
is a real Euclidean vector bundle
on $X$, which is naturally equipped with a Euclidean connection 
$\nabla^{\cE_\tau}$. The space of smooth sections 
$C^\infty({X},\cE_\tau)$ on ${X}$ can be identified with the 
$K$-invariant subspace $C^\infty(G,E_\tau)^K$ of smooth 
$E_\tau$-valued functions on $G$. Let $C^{\fg,X,\tau}$ 
\index{C@$C^{\fg,X,\tau}$}be the Casimir element  of $G$ acting on 
$C^\infty(X,\cE_\tau)$. Then $C^{\fg,X,\tau}$ is a  formally self-adjoint  second order elliptic differential operator which is bounded from below.

Observe that $K$ acts isometrically on $\fp$. Using the above construction, the tangent bundle
$TX=G\times_K\fp$
 is equipped with a Euclidean metric $g^{TX}$ \index{G@$g^{TX}$} and a Euclidean connection $\nabla^{TX}$.\index{N@$\nabla^{TX}$}
Also, $\nabla^{TX}$ is the Levi-Civita connection on $(TX,g^{TX})$ with curvature $R^{TX}$. 
Classically, $(X,g^{TX})$ is a Riemannian manifold of nonpositive sectional curvature.
For $x,y\in X$, we denote by $d_X(x,y)$ \index{D@$d_X(x,y)$} the Riemannian distance on $X$. 

If $E_\tau=\Lambda^\cdot(\fp^*)$, then $C^\infty(X,\cE_\tau)=\Omega^\cdot(X)$. In this case, we write $C^{\fg,X}=C^{\fg,X,\tau}$. By \cite[Proposition 7.8.1]{B09}, $C^{\fg,X}$ \index{C@$C^{\fg,X}$} coincides with the Hodge Laplacian acting on
$\Omega^\cdot(X)$.

Let us state a formula for $e\(TX,\nabla^{TX}\)$. Let $o(TX)$ \index{O@$o(TX)$} be the orientation line of $TX$.
Let $dv_{X}$ \index{D@$dv_X$} be the $G$-invariant Riemnnian volume form on $X$. If $\alpha\in \Omega^\cdot\big(X,o(TX)\big)$ is
of maximal degree and $G$-invariant, set $[\alpha]^{\max} \in \bR$ such that
\begin{align}\label{eq:amax}
\alpha=[\alpha]^{\max}dv_{X}.\index{A@$[\alpha]^{\max}$}
\end{align}

Recall that if $G$ has compact center, then $U$ is the compact form 
of $G$. If  $\delta(G)=0$, by \eqref{eq:gzgg}, $G$ has compact center. In this case, $T$ are  maximal torus of both $U$ and $K$. Let $W(T,U),W(T,K)$ \index{W@$W(T,K)$} be respectively the Weyl group. Let $\vol(U/K)$ \index{V@$\vol(\cdot)$} be the volume of $U/K$ with respecte to the volume form induced by $-B$.

\begin{prop}\label{cor:Euler}
If $\delta(G)\neq0$, $ \[e\(TX,\nabla^{TX}\)\]^{\max}=0$. If $\delta(G)=0$,
\begin{align}\label{eq:eulWW}
  \[e\(TX,\nabla^{TX}\)\]^{\max}=(-1)^{\frac{m}{2}}\frac{|W(T,U)|/|W(T,K)|}{\vol(U/K)}.
\end{align}
\end{prop}
\begin{proof}
If $G$ has noncompact center (thus $\delta(G)\neq0$), it is trivial 
that  $ \[e\(TX,\nabla^{TX}\)\]^{\max}=0$. Assume now,  $G$ has 
compact center. By Hirzebruch proportionality (c.f. \cite[Theorem 
22.3.1]{Hirzebruch_top_methods_AG} for a proof for Hermitian symmetric spaces, and the 
proof for general case is identical), we have
\begin{align}\label{eq:eUK}
\[e\(TX,\nabla^{TX}\)\]^{\max}=(-1)^{\frac{m}{2}}\frac{\chi(U/K)}{\vol(U/K)}.
\end{align}
Proposition \ref{cor:Euler} is a consequence of \eqref{eq:eUK},   \cite[Theorem II]{Bott1965} and Bott's formula  \cite[p. 175]{Bott1965} and of the fact that $\delta(G)=\rk_\bC(U)-\rk_\bC(K)$.
%
%
\end{proof}

Let $\gamma\in G$ be a semisimple element as in \eqref{eq:asr}. Let
\begin{align}
X(\gamma)=Z(\gamma)/K(\gamma)\index{X@$X(\gamma)$}
\end{align}
be the associated symmetric space. Clearly,
\begin{align} \label{eq:Xr}
X(\gamma)=Z^0(\gamma)/K^0(\gamma).
\end{align}

Suppose that $\gamma$ is nonelliptic. Set
\begin{align}\label{eq:Xa0}
X^{a,\bot}(\gamma)=  Z^{a,\bot,0}(\gamma)/K^0(\gamma).\index{X@$X^{a,\bot}(\gamma)$}
\end{align}
%
%
By \eqref{eq:Z0a}, \eqref{eq:Xr} and \eqref{eq:Xa0}, we have
\begin{align}\label{eq:ZXM}
  X(\gamma)=\mathbf{R}\times X^{a,\bot}(\gamma),
\end{align}
so that the action $e^{ta}$ on $X(\gamma)$ is just the translation by $t|a|$ on $\bR$.

\subsection{The semisimple orbital integrals}\label{sec:orb}
Recall that $\tau$ is a finite dimensional orthogonal representation of $K$ on the real Euclidean space $E_\tau$, and that $C^{\fg,X,\tau}$ is the Casimir element of $G$ acting on $C^\infty(X,\cE_\tau)$.

Let $p_t^{X,\tau}(x,x')$ \index{P@$p_t^{X,\tau}(x,x')$} be the smooth kernel of $\exp(-tC^{\fg,X,\tau}/2)$
 with respect to the Riemannian volume $dv_X$ on $X$. 
Classically, for $t>0$, there exist $c>0$ and $C>0$ such that for  $x,x'\in X$,
\begin{align}\label{eq:Qt}
   \left|p_t^{X,\tau}(x,x')\right|\l C\exp\(-c\ d_X^{2}(x,x')\).
\end{align}


Set
\begin{align}
  p_t^{X,\tau}(g)=p_t^{X,\tau}(p1,pg).\index{P@$p_t^{X,\tau}(g)$}
\end{align}
For  $g\in G$ and $k,k'\in K$,  we have
\begin{align}\label{eq:qtautype}
   p_t^{X,\tau}(kgk')=\tau(k)p_t^{X,\tau}(g)\tau(k').
\end{align}
Also, we can recover $p_t^{X,\tau}(x,x')$ by
\begin{align}
  p_t^{X,\tau}(x,x')=p_t^{X,\tau}(g^{-1}g'),
\end{align}
where $g,g'\in G$ are such that $pg=x,pg'=x'$.

In the sequel, we do not distinguish   $p_t^{X,\tau}(x,x')$ and 
$p^{X,\tau}_t(g)$. We refer to both of them as being the smooth kernel of $\exp(-tC^{\fg,X,\tau}/2)$.


Let $dv_{K^0(\gamma)\backslash K}$ and $dv_{Z^0(\gamma)\backslash G}$ \index{D@$dv_{K^0(\gamma)\backslash K},dv_{Z^0(\gamma)\backslash G}$} be the  Riemannian volumes on $K^0(\gamma)\backslash K$ and $Z^0(\gamma)\backslash G$ induced by $-B(\cdot,\theta\cdot)$. Let $\vol(K^0(\gamma)\backslash K)$ \index{V@$\vol(\cdot)$} be the volume of $K^0(\gamma)\backslash K$ with respect to $dv_{K^0(\gamma)\backslash K}$.

\begin{defin}\label{def:orbital}Let $\gamma\in G$ be semisimple.   The orbital integral of $\exp(-tC^{\fg,X,\tau}/2)$ is defined by
\begin{align}\label{eq:TRrz}
  \Tr^{[\gamma]}\[\exp\(-tC^{\fg,X,\tau}/2\)\]=\frac{1}{\vol(K^0(\gamma)\backslash K)}\int_{g\in Z^0(\gamma)\backslash G}\Tr^{E_\tau}\[p_t^{X,\tau}(g^{-1}\gamma g)\]dv_{Z^0(\gamma)\backslash G}.\index{T@$\Tr^{[\gamma]}[\cdot]$}
\end{align}
\end{defin}
\begin{re}
  Definition \ref{def:orbital}  is equivalent to  \cite[Definition 4.2.2]{B09}, where the volume forms are normalized  such that $\vol(K^0(\gamma)\backslash K)=1$.
\end{re}

\begin{re}
As the notation $\Tr^{[\gamma]}$ indicates, the orbital integral only depends on the conjugacy class of $\gamma$ in $G$.
However, the notation $[\gamma]$ (c.f. Subsection \ref{sec:Gamma}) will 
be used later for the conjugacy class in the  discrete group $\Gamma$. 
\end{re}

\begin{re}
We will also consider the case where $E_\tau$ is a $\bZ_2$-graded or virtual  representation of $K$. We will use the notation $\Trs^{[\gamma]}[q]$ \index{T@$\Trs^{[\gamma]}[\cdot]$}
 when the trace on the right-hand side of \eqref{eq:TRrz} is replaced by the supertrace  on $E_\tau$.
\end{re}

\subsection{ Bismut's formula for semisimple orbital integrals}\label{sec:J}Let us recall the explicit formula for $\Tr^{[\gamma]}\[\exp(-tC^{\fg,X,\tau}/2)\]$, for any semisimple element $\gamma\in G$, obtained by Bismut \cite[Theorem 6.1.1]{B09}.
 %

Let $\gamma=e^ak^{-1}\in G$ be semisimple as in \eqref{eq:asr}.
Set
\begin{align}\label{eq:z0k0p0}
 &\fz_0=\fz(a), &\fp_0=\fz(a)\cap \fp,&&\fk_0=\fz(a)\cap\fk.\index{Z@$\fz_0$} \index{P@$\fp_0$} \index{K@$\fk_0$}
\end{align}
Then
\begin{align}
  \fz_0=\fp_0\oplus \fk_0.
\end{align}

By \eqref{eq:Zr}, \eqref{eq:prkr} and \eqref{eq:z0k0p0}, we have $\fp(\gamma)\subset\fp_0$ and $\fk(\gamma)\subset\fk_0$. Let $\fp_0^{\bot}(\gamma)$,
$\fk_0^{\bot}(\gamma)$, $\fz_0^\bot(\gamma)$ \index{Z@$\fz_0^{\bot}(\gamma)$} \index{P@$\fp_0^{\bot}(\gamma)$} \index{K@$\fk_0^{\bot}(\gamma)$}
be the orthogonal spaces of $\fp(\gamma)$, $\fk(\gamma)$, $\fz(\gamma)$ in $\fp_0$, $\fk_0$, $\fz_0$. Let
$\fp_0^{\bot}$, $\fk_0^{\bot}$, $\fz_0^\bot$ \index{Z@$\fz_0^{\bot}$} \index{P@$\fp_0^{\bot}$} \index{K@$\fk_0^{\bot}$} be the orthogonal spaces of $\fp_0$, $\fk_0$, $\fz_0$ in $\fp$, $\fk$, $\fz$.
Then we have
\begin{align}\label{eq:z0kopobot}
  &\fp=\fp(\gamma)\oplus \fp_0^\bot(\gamma)\oplus \fp^\bot_0,
  &\fk=\fk(\gamma)\oplus \fk_0^\bot(\gamma)\oplus \fk^\bot_0.
\end{align}

Recall that $\widehat{A}$ is the  function defined in \eqref{eq:defAhat}.
\begin{defin} For $Y\in \fk(\gamma)$, put
\begin{multline}\label{eq:J}
  J_{\gamma}({Y})=\frac{1}{\left|\det\big(1-\Ad(\gamma)\big)|_{\fz_0^\bot}\right|^{1/2}}\frac{\widehat{A}\big(i\ad({Y})|_{\fp(\gamma)}\big)}{\widehat{A}\big(i\ad(Y)|_{\fk(\gamma)}\big)}\\
  \[\frac{1}{\det\big(1-\Ad(k^{-1})\big)|_{\fz_0^\bot(\gamma)}}\frac{\det\big(1-\exp(-i\ad(Y))\Ad(k^{-1})\big)|_{\fk_0^\bot(\gamma)}}{\det\big(1-\exp(-i\ad(Y))\Ad(k^{-1})\big)|_{\fp_0^\bot(\gamma)}}\]^{1/2}.\index{J@$J_{\gamma}$}
\end{multline}
\end{defin}
As explained in \cite[Section 5.5]{B09}, there is a natural choice for the square root in \eqref{eq:J}.
Moreover, $J_{\gamma}$ is an  $\Ad\big(K^0(\gamma)\big)$-invariant analytic function on $\fk(\gamma)$, and there exist $c_\gamma>0, C_\gamma>0$, such that for $Y\in \fk(\gamma)$,
\begin{align}\label{eq:Jrexp}
   |J_{\gamma}(Y)|\l C_\gamma \exp{(c_\gamma|Y|)}.
\end{align}
By \eqref{eq:J}, we have
\begin{align}\label{eq:J1}
  J_1(Y)=\frac{\widehat{A}(i\ad(Y)|_{\fp})}{\widehat{A}(i\ad(Y)|_{\fk})}.
\end{align}

For $Y\in \fk(\gamma)$, let $dY$ be the Lebesgue measure on $\fk(\gamma)$ induced by $-B$. Recall that $C^{\fk,\fp}$ and $C^{\fk,\fk}$ are defined in \eqref{eq:ckkckp}.
The main result of \cite[Theorem 6.1.1]{B09} is the following.

\begin{thm}\label{thm:Bis}
For $t>0$, we have
\begin{multline}\label{eq:trrJr}
\Tr^{[\gamma]}\[\exp\(-tC^{\fg,X,\tau}/2\)\]=\frac{1}{(2\pi t)^{\dim \fz(\gamma)/2}}\exp\(-\frac{|a|^2}{2t}+\frac{t}{16}\Tr^{\fp}[C^{\fk,\fp}]+\frac{t}{48}\Tr^{\fk}\[C^{\fk,\fk}\]\)\\
  \int_{Y\in \fk(\gamma)}J_\gamma(Y)
  \Tr^{E_\tau}\[\tau\(k^{-1}\)\exp(-i{\tau(Y)})\]\exp\(-|{Y}|^2/2t\)dY.
\end{multline}
\end{thm}

\subsection{A discrete subgroup of $G$}\label{sec:Gamma}
Let $\Gamma\subset G$ \index{G@$\Gamma$} be a discrete torsion-free 
cocompact subgroup of $G$. By \cite[Lemma 1]{Selberg60}, $\Gamma$ 
contains the identity element and nonelliptic semisimple elements. 
Also, $\Gamma$ acts isometrically on the left on $X$. This action lifts to all the homogeneous Euclidean vector bundles $\cE_\tau$ constructed in Subsection \ref{sec:sym}, and preserves the corresponding connections.

Take $Z=\Gamma\backslash X=\Gamma \backslash G/K$. \index{Z@$Z$}Then $Z$ is a connected closed orientable Riemannian locally symmetric manifold with nonpositive sectional curvature. Since $X$ is contractible,
  $\pi_1(Z)=\Gamma$ and $X$ is the universal cover of $Z$. We denote by $\widehat{p}:\Gamma\backslash G\to Z$ and $\widehat{\pi}:X\to Z$ \index{P@$\widehat{p},\widehat{\pi}$}the natural projections, so that the diagram
\begin{align}
\begin{aligned}
\xymatrix{
G \ar[d]^p \ar[r] &\Gamma \backslash G\ar[d]^{\widehat{p}}\\
X \ar[r]^{\widehat{\pi}} &Z}
\end{aligned}
\end{align}
commutes.

The Euclidean vector bundle $\cE_\tau$ descends to a Euclidean vector bundle $\cF_\tau=\Gamma \backslash\cE_\tau$ \index{F@$\cF_\tau$} on $Z$. Take $r\in \mathbf{N}^*$. \index{R@$r$} Let $\rho:\Gamma\to \mathrm{U}(r)$ be a unitary representation of $\Gamma$.  \index{R@$\rho$} \index{F@$F$} Let
 $(F,\nabla^F,g^F)$ be the unitarily flat vector bundle on $Z$ associated to $\rho$. 
  Let $C^{\fg,Z,\tau,\rho}$  be the Casimir element of $G$ acting on $C^\infty(Z,\cF_{\tau}\otimes_\bC F)$.
As in Subsection \ref{sec:sym},  when $E_\tau=\Lambda^\cdot(\fp^*)$, we write $C^{\fg,Z,\rho}=C^{\fg,Z,\tau,\rho}$. Then,
\begin{align}\label{eq:D=C}
  \Box^Z=C^{\fg,Z,\rho}.
\end{align}

Recall that  $p_t^{X,\tau}(x,x')$ is the smooth kernel of $\exp(-tC^{\fg,X,\tau}/2)$ with respect to $dv_X$. 



\begin{prop}\label{prop:heatconv}
 There exist $c>0$, $C>0$ such that for $t>0$ and $x\in X$, we have
\begin{align}\label{eq:Fheat2}
   \sum_{\gamma\in \Gamma- \{1\}} \left| p_t^{X,\tau}(x, \gamma x)\right|\l C\exp\(-\frac{c}{t}+Ct\).
\end{align}
\end{prop}
\begin{proof}
%
By \cite[Remark p.1, Lemma 2]{Milnor_fundgroup} or \cite[eq. 
(3.19)]{MaMar_cover}, there is $C>0$ such that for all $r\g 0$,  
$x\in X$, we have 
\begin{align}\label{eq:N<er}
	\big|\big\{\gamma\in \Gamma: d_{X}(x,\gamma x)\l 
	r\big\}\big|\l Ce^{Cr}. 
\end{align}
	
%
%
%
%

We claim that  there exist $c>0, C>0$ and $N\in \mathbf{N}$ such that for  $t>0$ and $x,x'\in X$, we have
\begin{align}\label{eq:estheat}
 \left|  p_t^{X,\tau}(x, x')\right|\l \frac{C}{t^N}\exp\(-c\frac{d_X^2(x, x')}{t}+Ct\).
\end{align}
Indeed, if $\tau=\mathbf{1}$, $p_t^{X,\mathbf{1}}(x, x')$ is the heat 
kernel for the Laplace-Beltrami operator. In this case, 
\eqref{eq:estheat} is a consequence of  the Li-Yau estimate 
\cite[Corollary 3.1]{Li_Yau} and of the fact that $X$ is a symmetric 
space.  For general $\tau$, using the It\^o formula as in \cite[eq.  (12.30)]{BZ92}), we can show that there is $C>0$ such that
\begin{align}
\left|  p_t^{X,\tau}(x, x')\right|\l C e^{Ct} p_t^{X,\mathbf{1}}(x, x'),
\end{align}
from which we get \eqref{eq:estheat}\footnote{See \cite[Theorem 
4]{MaMar_cover} for another proof of \eqref{eq:estheat} using finite propagation speed of 
solutions of hyperbolic equations.}.

Note that there exists $c_0>0$ such that for all $\gamma\in 
\Gamma-\{1\}$ and $x\in X$,
\begin{align}\label{eq:dxrx}
  d_X(x,\gamma x)\g c_0.
\end{align}
By \eqref{eq:estheat} and \eqref{eq:dxrx}, there exist $c_1>0$, $c_2>0$ and $C>0$ such that
for $t>0$, $x\in X$ and $\gamma \in \Gamma-\{1\} $, we have
\begin{align}\label{eq:hk20}
 \left|  p_t^{X,\tau}(x, \gamma x)\right|\l C 
 \exp\(-\frac{c_1}{t}-c_2\frac{d_X^2(x,\gamma x)}{t}+Ct\).
\end{align}

By \eqref{eq:N<er} and\eqref{eq:hk20},  for $t>0$ and $x\in X$, we have
\begin{align}\label{eq:hk1}
\begin{aligned}
  \sum_{\gamma\in \Gamma-\{1\}} \left|p_t^{X,\tau}(x, \gamma 
  x)\right|&\l   C\sum_{\gamma\in \Gamma} 
  \exp\(-\frac{c_1}{t}-c_2\frac{d_X^2(x,\gamma x)}{t}+Ct\)\\
&=c_2C\exp\(-\frac{c_1}{t}+Ct\)\sum_{\gamma\in 
\Gamma}\int^{\infty}_{d_X^2(x,\gamma x)/t}\exp(-c_2r)dr\\
&  =c_2C\exp\(-\frac{c_1}{t}+Ct\)\int_0^\infty \big|\big\{\gamma\in 
\Gamma:   d_X(x,\gamma x)\l \sqrt{rt}\big\}\big|\exp(-c_2r)dr\\
&\l C'\exp\(-\frac{c_1}{t}+Ct\)  \int_0^\infty \exp\(-c_2r+C\sqrt{rt}\)dr.
\end{aligned}
\end{align}
From  \eqref{eq:hk1}, we get \eqref{eq:Fheat2}. The proof of our proposition is completed.
%
%
%
%
%
%
\end{proof}

For $\gamma\in \Gamma$, set
\begin{align}\label{eq:Gr}
\Gamma(\gamma)=Z(\gamma)\cap \Gamma.\index{G@$\Gamma(\gamma)$}
\end{align}
Let $[\gamma]$ be the conjugacy class of $\gamma$ in $\Gamma$.\index{G@$[\gamma]$}
Let $[\Gamma]$ be the set of all the conjugacy classes of $\Gamma$.\index{G@$[\Gamma]$}

The following proposition is \cite[Lemma 2]{Selberg60}. We include a proof for the sake of completeness.


\begin{prop}\label{prop:cocompact}
If $\gamma\in \Gamma$,   then $\Gamma(\gamma)$ is cocompact in $Z(\gamma)$.
\end{prop}
\begin{proof}
Since $\Gamma$ is discrete, $[\gamma]$ is closed in $G$. The inverse image of $[\gamma]$ by the continuous  map $g\in G\to g\gamma g^{-1}\in G$ is $\Gamma\cdot Z(\gamma)$. Then $\Gamma\cdot Z(\gamma)$ is closed in $G$. Since $\Gamma\backslash G$ is compact,
the closed subset $\Gamma\backslash \Gamma\cdot Z(\gamma)\subset \Gamma\backslash G$ is then compact.

The group $Z(\gamma)$ acts  transitively on the right on 
$\Gamma\backslash \Gamma \cdot Z(\gamma)$. The stabilizer at $[1]\in 
\Gamma\backslash \Gamma\cdot Z(\gamma)$ is $\Gamma(\gamma)$. Hence $ 
\Gamma(\gamma)\backslash Z(\gamma)\simeq \Gamma\backslash \Gamma 
\cdot Z(\gamma)$ is compact. The proof of our proposition is completed.
\end{proof}

Let $\vol\big(\Gamma(\gamma)\backslash X(\gamma)\big)$ be the volume of $\Gamma(\gamma)\backslash X(\gamma)$
with respect to the volume form  induced by
$dv_{X(\gamma)}$. Clearly, $\vol\big(\Gamma(\gamma)\backslash X(\gamma)\big)$ depends only on the conjugacy class $[\gamma]\in [\Gamma]$.

By the property of heat kernels on compact manifolds, the operator $\exp\(-tC^{\fg,Z,\tau,\rho}/2\)$ is  trace class. Its trace is given by the Selberg trace formula: 
\begin{thm}\label{thm:sel}
There exist $c>0$, $C>0$ such that
  for $t>0$, we have
  \begin{align}\label{eq:h11exp}
    \sum_{[\gamma]\in [\Gamma]- \{1\}} \vol\big(\Gamma(\gamma)\backslash X(\gamma)\big)  \left|\Tr^{[\gamma]}\[\exp\(-tC^{\fg,X,\tau}/2\)\]\right|\l C\exp\(-\frac{c}{t}+Ct\).
  \end{align}
  For $t>0$, the following identity holds:
\begin{align}\label{eq:sel}
  \Tr\[\exp\(-tC^{\fg,Z,\tau,\rho}/2\)\]=\sum_{[\gamma]\in [\Gamma]} \vol\big(\Gamma(\gamma)\backslash X(\gamma)\big) \Tr[\rho(\gamma)] \Tr^{[\gamma]}\[\exp(-tC^{\fg,X,\tau}/2)\].
\end{align}
\end{thm}
\begin{proof}
	Let $F\subset X$ be a fundamental domain of $Z$ in $X$.  By \cite[eq. (4.8.11), (4.8.15)]{B09}, we have
\begin{align}\label{eq:sel2}
 \sum_{\gamma'\in [\gamma]}\int_{x\in F}\Tr^{E_\tau}\[p_t^{X,\tau}(x,\gamma' x)\]dx=\vol\big(\Gamma(\gamma)\backslash X(\gamma)\big) \Tr^{[\gamma]}\[\exp(-tC^{\fg,X,\tau}/2)\].
\end{align}
By  \eqref{eq:Fheat2} and \eqref{eq:sel2}, we get \eqref{eq:h11exp}. 
The proof of \eqref{eq:sel} is well known (c.f. \cite[Section 4.8]{B09}). 
\end{proof}

\subsection{A formula for $ 
\Trs^{[\gamma]}\[N^{\Lambda^\cdot(T^*X)}\exp\(-tC^{\fg,X}/2\)\]$}\label{sec:t=0}

Let $\gamma=e^ak^{-1}\in G$ be semisimple such that \eqref{eq:asr} holds. Let $\ft(\gamma)\subset \fk(\gamma)$ \index{T@$\ft(\gamma)$} be a Cartan subalgebra of $\fk(\gamma)$.
Set
\begin{align}\label{eq:defbr}
  \fb(\gamma)=\{Y\in \fp: \Ad(k)Y=Y, [Y,\ft(\gamma)]=0\}.\index{B@$ \fb(\gamma)$}
\end{align}
Then,
\begin{align}\label{eq:abr}
  a\in \fb(\gamma).
\end{align}
By definition, $\dim \fp-\dim \fb(\gamma)$ is even.

Since $k$ centralizes $\ft(\gamma)$, by \cite[Theorem 4.21]{Knappsemi}, there is $k'\in K$ such that
\begin{align}
  &k'\ft(\gamma)k^{\prime-1}\subset \ft, &k'kk^{\prime-1}\in T.
\end{align}
Up to a conjugation on $\gamma$, we can assume directly that $\gamma=e^ak^{-1}$ with
\begin{align}\label{eq:trkpo}
& \ft(\gamma)\subset \ft,&k\in T.
\end{align}
By \eqref{eq:defb}, \eqref{eq:defbr}, and \eqref{eq:trkpo}, we have
\begin{align}\label{eq:bbr}
  \fb\subset \fb(\gamma).
\end{align}

\begin{prop}\label{prop:dbdbr}
   A semisimple element $\gamma\in G$ can be conjugated into $H$ if and only if
\begin{align}
  \label{eq:b=br}
\dim \fb= \dim \fb(\gamma).
\end{align}
\end{prop}
\begin{proof}
If $\gamma\in H$, then $\ft(\gamma)=\ft$. By \eqref{eq:defbr}, we get $\fb=\fb(\gamma)$, which implies \eqref{eq:b=br}.

Recall that $\fh(\gamma)\subset \fz(\gamma)$ 
is defined as in \eqref{eq:carhH}, when $G$ is replaced by $Z^0(\gamma)$ and $\ft$ is replaced by $\ft(\gamma)$. It is a $\theta$-invariant Cartan subalgebra of both $\fg$ and $\fz(\gamma)$. Let
$\fh(\gamma)=\fh(\gamma)_\fp\oplus \fh(\gamma)_\fk$ be the Cartan decomposition. Then,
\begin{align}\label{eq:hp}
 & \fh(\gamma)_ \fp=\{Y\in \fp(\gamma): 
 [Y,\ft(\gamma)]=0\}=\fb(\gamma)\cap \fp(\gamma),&\fh(\gamma)_\fk=\ft(\gamma).
\end{align}
From \eqref{eq:dgdz} and \eqref{eq:hp}, we get
\begin{align}\label{eq:bhpbr}
  \dim \fb\l \dim  \fh(\gamma)_\fp \l \dim \fb(\gamma).
\end{align}
By \eqref{eq:bhpbr}, if $\dim \fb= \dim \fb(\gamma)$, then $\dim \fb= \dim  \fh(\gamma)_\fp$.
By Proposition \ref{prop:dgdz}, $\gamma$ can be conjugated into $H$. 
The proof of our proposition is completed.
\end{proof}

The following Proposition extends \cite[Theorem 7.9.1]{B09}.
\begin{thm}\label{prop:vanishT}
 Let $\gamma\in G$ be semisimple such that $\dim \fb(\gamma)\g 2 $. For $Y\in \fk(\gamma)$, we have
\begin{align}\label{eq:vanishT}
  \Trs^{\Lambda^\cdot(\fp^*)}\[N^{\Lambda^\cdot(\fp^*)}\Ad(k^{-1})\exp(-i\ad(Y))\]=0.
\end{align}
In particular, for $t>0$, we have
\begin{align}\label{eq:vanishT2}
    \Trs^{[\gamma]}\[N^{\Lambda^\cdot(T^*X)}\exp\(-tC^{\fg,X}/2\)\]=0.
  \end{align}
\end{thm}
\begin{proof}
Since the left-hand side of \eqref{eq:vanishT} is 
$\Ad\big(K^0(\gamma)\big)$-invariant, it is enough to show 
\eqref{eq:vanishT} for $Y\in \ft(\gamma)$. If $Y\in \ft(\gamma)$, by 
\cite[eq. (7.9.1)]{B09}, we have
\begin{align}\label{eq:moddtor0}
  \Trs^{\Lambda^\cdot(\fp^*)}\[N^{\Lambda^\cdot(\fp^*)}\Ad(k^{-1})\exp(-i\ad(Y))\]=\frac{\p}{\p b}\Big|_{b=0}\mathrm{det}\(1-e^b\Ad(k)\exp(i\ad(Y))\)|_{\fp}.
\end{align}
Since $\dim \fb(\gamma)\g2$, by \eqref{eq:moddtor0}, we get \eqref{eq:vanishT} for $Y\in \ft(\gamma)$.

By \eqref{eq:trrJr} and \eqref{eq:vanishT}, we get \eqref{eq:vanishT2}.
The proof of our theorem  is completed.
\end{proof}

In this way, Bismut \cite[Theorem 7.9.3]{B09} recover  \cite[Corollary 2.2]{MStorsion}.

\begin{cor}\label{cor:ms}Let $F$ be a unitarily flat vector bundle on $Z$. Assume that  $\dim Z$ is odd and $\delta(G)\neq 1$. Then for any $t>0$, we have
\begin{align}\label{eq:BMS}
    \Trs\[N^{\Lambda^\cdot(T^*Z)}\exp\(-t\Box^Z/2\)\]=0.
  \end{align}
In particular,
  \begin{align}\label{eq:tau=0}
    T(F)=1.
  \end{align}
\end{cor}
\begin{proof}Since $\dim Z$ is odd, $\delta(G)$ is odd. Since 
	$\delta(G)\neq 1$,  $\delta(G)\g3$.  By \eqref{eq:bbr}, $\dim 
	\fb(\gamma)\g\delta(G)\g3$,  so  \eqref{eq:BMS} is a consequence of \eqref{eq:D=C}, \eqref{eq:sel} and \eqref{eq:vanishT2}.
\end{proof}

Suppose that $\delta(G)=1$. Up to sign, we fix an element  $a_{1}\in \fb$ such 
that $B(a_{1},a_{1})=1$. As in Subsection \ref{sec:semi}, set  
\index{M@$M$} \index{K@$K_{M}$}
\begin{align}\label{eq:MPMKM1}
&M=Z^{a_{1},\bot,0}(e^{a_{1}}),& K_{M}=K^{0}(e^{a_{1}}),
\end{align}
and \index{M@$\fm$} \index{P@$\fp_{\fm}$} \index{K@$\fk_{\fm}$}
\begin{align}\label{eq:mpkd1}
&\fm=\fz^{a_{1},\bot}(e^{a_{1}}), 
&\fp_{\fm}=\fp^{a_{1},\bot}(e^{a_{1}}), &&\fk_{\fm}=\fk(e^{a_{1}}).
\end{align}
As in Subsection \ref{sec:semi}, $M$  
is a connected reductive group with Lie algebra $\fm$, with maximal 
compact subgroup $K_{M}$, and  with Cartan decomposition 
$\fm=\fp_{\fm}\oplus\fk_{\fm}.$  Let \index{X@$X_{M}$}
\begin{align}\label{eq:XM}
X_{M}=M/K_{M}
\end{align}
be the corresponding symmetric space.  By definition, $T\subset M$ is a 
compact Cartan subgroup. Therefore $\delta(M)=0$, and $\dim \fp_{\fm}$ is even.

Assume  that $\delta(G)=1$ and that $G$ has  noncompact center, so 
that $\dim \fz_{\fp}\g1$.  By 
\eqref{eq:defbzp11},  we find that 
$a_{1}\in\fz_{\fp}$, so that  $Z^{0}(a_{1})=G$. By \eqref{eq:Z0a} and \eqref{eq:ZXM}, we have
\begin{align}\label{eq:G=RM}
	&G=\bR\times M, &K=K_{M}, &&X=\bR\times X_{M}. 
\end{align}

Let $\gamma\in G$ be a semisimple element such that $\dim 
\fb(\gamma)=1$. By Proposition \ref{prop:dbdbr}, we may assume that $\gamma=e^{a}k^{-1}$ with $a\in \fb$ and 
$k\in T$. 
\begin{prop}
	We have 
 	\begin{align}\label{eq:dgn11}
 		\Trs^{[1]}\[N^{\Lambda^{\cdot}(T^{*}X)}\exp\(-tC^{\fg, 
		X}/2\)\]=-\frac{1}{\sqrt{2\pi 
		t}}\[e\(TX_{M},\nabla^{TX_{M}}\)\]^{\max}.
 	\end{align}
	If $\gamma=e^{a}k^{-1}$ with $a\in \fb$, $a\neq0$, and $k\in T$, then 
	\begin{align}\label{eq:dgn1}
		\Tr^{[\gamma]}\[N^{\Lambda^{\cdot}(T^{*}X)}\exp\(-tC^{\fg,X}/2\)\]=-\frac{1}{\sqrt{2\pi t}}e^{-|a|^{2}/2t}\[e\(TX^{a,\bot}(\gamma),\nabla^{TX^{a,\bot}(\gamma)}\)\]^{\max}.
	\end{align}
\end{prop}
\begin{proof}
%
%
	By \eqref{eq:G=RM}, for $\gamma=e^{a}k^{-1}$ with $a\in \fb$ and $k\in T$, we have 
	\begin{align}\label{eq:dgnc2}
		\Tr_{s}^{[\gamma]}\[N^{\Lambda^{\cdot}(T^{*}X)}\exp(-tC^{\fg,X}/2)\]=-\Tr^{[e^{a}]}\[\exp(t\Delta^{\bR}/2)\]\Tr_{s}^{[k^{-1}]}\[\exp(-tC^{\fm,X_{M}}/2)\],
   \end{align}
where $\Delta^{\bR}$ is the   Laplace-Beltrami operator acting on $C^{\infty}(\bR)$. 
 
Clearly, 
\begin{align}\label{eq:obR}
\Tr^{[e^{a}]}\[\exp(t\Delta^{\bR}/2)\]=\frac{1}{\sqrt{2\pi 
t}}e^{-|a|^{2}/2t}.
\end{align}
By \cite[Theorem 7.8.13]{B09}, we have 
\begin{align}\label{eq:dgnc1}
\Tr_{s}^{[1]}\[\exp(-tC^{\fm,X_{M}}/2)\]=\[e\(TX_{M},\nabla^{TX_{M}}\)\]^{\max}.
\end{align}
and
\begin{align}\label{eq:dgnc12}
\Tr_{s}^{[k^{-1}]}\[\exp(-tC^{\fm,X_{M}}/2)\]=\[e\(TX^{a,\bot}(\gamma),\nabla^{TX^{a,\bot}(\gamma)}\)\]^{\max}.
\end{align}
By \eqref{eq:dgnc2}-\eqref{eq:dgnc12}, we get \eqref{eq:dgn11} and \eqref{eq:dgn1}, 
which completes the proof of our proposition. 
\end{proof}


\section{The solution to Fried conjecture}\label{sec:dyn}
We use the notation in Sections \ref{Sec:Pred} and \ref{sec:sel}. 
Also, we assume that $\dim \fp$ is odd. The purpose of this section 
is to introduce the Ruelle dynamical zeta function on  $Z$ and to
state our main result, which contains the solution of the Fried 
conjecture in the case of locally symmetric spaces. 

This section is organized as follows.
In Subsection \ref{sec:Vinva}, we describe the closed geodesics on $Z$.


In Subsection \ref{sec:main}, we define the dynamical zeta function 
and state Theorem \ref{thm:01z}, which is the main result of the 
article. 

Finally, in Subsection \ref{Sec:dg=1}, we establish Theorem 
\ref{thm:01z} when $G$ has noncompact center and $\delta(G)=1$.

%
%
%

\subsection{The space of closed geodesics}\label{sec:Vinva}
 By \cite[Proposition 5.15]{DuistermaatKolkVaradarajan}, the set of  
 nontrivial closed geodesics  on  $Z$ consists of a disjoint union of smooth connected closed submanifolds
\begin{align}
\coprod_{[\gamma]\in [\Gamma]-[1]}B_{[\gamma]}.
\end{align}
Moreover, $B_{[\gamma]}$ is diffeomorphic to $\Gamma(\gamma)\backslash X(\gamma)$. All the  elements of $B_{[\gamma]}$ have the same length $|a|>0$, if $\gamma$ can be conjugated to $e^ak^{-1}$ as in \eqref{eq:asr}. \index{B@$B_{[\gamma]}$} Also, the geodesic flow induces a canonical locally free action of $\mathbb{S}^1$ on $B_{[\gamma]}$, so that $\mathbb{S}^1\backslash B_{[\gamma]}$ is a closed orbifold. The $\mathbb{S}^1$-action is not necessarily effective. Let
\begin{align}
m_{[\gamma]}=\left|\ker\big(\bbS^1\to {\rm 
Diff}(B_{[\gamma]})\big)\right|\in \mathbf{N}^{*}
\end{align}
be the generic multiplicity.\index{M@$m_{[\gamma]}$}

Following \cite{SatakeGaussB}, if $S$ is a closed Riemannian orbifold 
with Levi-Civita connection $\nabla^{TS}$, then $e\big(TS,\nabla^{TS}\big)\in \Omega^{\dim S}\big(S,o(TS)\big)$ is still well define, and the  Euler characteristic  $\chi_{\rm orb}(S)\in \mathbf{Q}$ is given by \index{C@$\chi_{\rm orb}$}
\begin{align}
\chi_{\rm orb}(S)=\int_S e\big(TS,\nabla^{TS}\big).
\end{align}


\begin{prop}\label{prop:v=e} For $\gamma\in \Gamma-\{1\}$, the following identity holds:
\begin{align}\label{eq:V=e}
  \frac{\chi_{\rm orb}(\mathbb{S}^1\backslash B_{[\gamma]})}{m_{[\gamma]}}=\frac{\vol\big(\Gamma(\gamma)\backslash X(\gamma)\big)}{|a|}\[e\(TX^{a,\bot}(\gamma),\nabla^{TX^{a,\bot}(\gamma)}\)\]^{\max}.
\end{align}
\end{prop}
\begin{proof}
Take $\gamma\in \Gamma- \{1\}$. We can assume that $\gamma=e^ak^{-1}$ as in \eqref{eq:asr} with $a\neq0$. By \eqref{eq:Zr} and \eqref{eq:Gr}, for $t\in \bR$, $e^{ta}$ commutes with  elements of $\Gamma(\gamma)$. Thus,
$e^{ta}$  acts on the left on $\Gamma(\gamma)\backslash X(\gamma)$. Since $e^a=\gamma k$,
$\gamma\in \Gamma(\gamma)$, $k\in K(\gamma)$ and $k$ commutes  with 
elements of $Z(\gamma)$, we see that $e^a$ acts as identity on 
$\Gamma(\gamma)\backslash X(\gamma)$. This induces an  
$\bR/\mathbf{Z}\simeq \mathbb{S}^1 $ action on $\Gamma(\gamma)\backslash X(\gamma)$ which coincides with the $\mathbb{S}^1$-action on $B_{[\gamma]}$. Therefore,
\begin{align}\label{eq:orb1}
\chi_{\rm orb}(\mathbb{S}^1\backslash B_{[\gamma]})=\vol(\mathbb{S}^1\backslash B_{[\gamma]})\[e\(TX^{a,\bot}(\gamma),\nabla^{TX^{a,\bot}(\gamma)}\)\]^{\max}
\end{align}
and
\begin{align}\label{eq:orb2}
\frac{\vol(\mathbb{S}^1\backslash B_{[\gamma]})}{m_{[\gamma]}}=\frac{\vol\big(\Gamma(\gamma)\backslash X(\gamma)\big)}{|a|}.
\end{align}
By \eqref{eq:orb1} and \eqref{eq:orb2}, we get \eqref{eq:V=e}. The 
proof of our proposition  is completed.
\end{proof}

\begin{cor}\label{cor:V=0}
Let $\gamma\in \Gamma- \{1\}$.  If $\dim \fb(\gamma)\g 2$, then
\begin{align}\label{eq:V=0}
  \chi_{\rm orb}\big(\mathbb{S}^1\backslash B_{[\gamma]}\big)=0.
\end{align}
\end{cor}
\begin{proof}
By Propositions \ref{cor:Euler} and \ref{prop:v=e}, it is enough to show that
  \begin{align}\label{eq:d1ab}
   \delta(Z^{a,\bot,0}(\gamma))\g 1.
  \end{align}
By \eqref{eq:Z0a} and \eqref{eq:dgdz}, we have
  \begin{align}\label{eq:dza}
    \delta(Z^{a,\bot,0}(\gamma))=\delta(Z^0(\gamma))-1\g \delta(G)-1.
  \end{align}
  Recall $\dim \fp$ is odd, therefore $\delta(G)$ is odd.  If $\delta(G)\g 3$,  by \eqref{eq:dza}, we get \eqref{eq:d1ab}.
If $\delta(G)=1$, then $\dim \fb(\gamma)\g 2>\delta(G)$.
By Propositions \ref{prop:dgdz} and \ref{prop:dbdbr}, the inequality 
in \eqref{eq:dza} is strict, which implies \eqref{eq:d1ab}. The proof 
of our corollary  is completed.
\end{proof}

\begin{re}\label{re:20}
By Theorem \ref{prop:vanishT} and Corollary \ref{cor:V=0}, both $\Trs^{[\gamma]}\[N^{\Lambda^\cdot(T^*X)}\exp\(-tC^{\fg,X}/2\)\]$ and $\chi_{\rm orb}\big(\mathbb{S}^1\backslash B_{[\gamma]}\big)$ vanish when $\dim \fb(\gamma)\g2$.
\end{re}

\subsection{Statement of the main result}\label{sec:main}
Recall that $\rho:\Gamma\to \mathrm{U}(r)$ is a unitary 
representation of $\Gamma$ and that $(F,\nabla^F,g^F)$ is the 
unitarily  flat vector bundle on $Z$ associated with $\rho$.
\begin{defin}\label{def:Rr}
The Ruelle dynamical zeta function $R_\rho(\sigma)$ is said to be well defined, if the following properties hold:\index{X@$\Xi_{\rho}$}
\begin{enumerate}
  \item For $\sigma \in \bC, \Re(\sigma)\gg1$, the sum \begin{align}\label{eq:Rrho}
  \Xi_\rho(\sigma)=\sum_{[\gamma]\in [\Gamma]-\{1\}}\frac{\chi_{\rm orb}\big(\mathbb{S}^1\backslash B_{[\gamma]}\big)}{m_{[\gamma]}}\Tr[\rho(\gamma)]e^{-\sigma|a|}
\end{align}
 converges to a holomorphic function.
  \item The function $R_\rho(\sigma)=\exp\big(\Xi_\rho(\sigma)\big)$ \index{R@$R_\rho(\sigma)$}has a meromorphic extension to $\sigma\in \bC$.
\end{enumerate}
\end{defin}
If  $\delta(G)\neq1$, by Corollary \ref{cor:V=0},
\begin{align}\label{eq:Rr=0}
  R_\rho(\sigma)\equiv1.
\end{align}

The main result of this article is the solution of the Fried conjecture. We restate Theorem \ref{thm:01} as follows.
\begin{thm}\label{thm:01z}
The dynamical zeta function  $R_\rho(\sigma)$ is  well defined. There 
exist explicit constants $C_\rho\in \bR^{*}$ and $r_\rho\in \mathbf{Z}$ (c.f. \eqref{eq:Cr}) such that, when  $\sigma\to 0$, we have
\begin{align}\label{eq:t01z}
  R_\rho(\sigma)=C_\rho T(F)^2\sigma^{r_\rho}+\cO\(\sigma^{r_\rho+1}\).
\end{align}
If $H^\cdot(Z,F)=0$, then
\begin{align}\label{eq:t02z}
  &C_\rho=1,&r_\rho=0,
\end{align}
so that
\begin{align}\label{eq:t03z}
R_\rho(0)=T(F)^2.
\end{align}
\end{thm}
\begin{proof}
When $\delta(G)\neq 1$,  Theorem \ref{thm:01z} is a consequence of 
\eqref{eq:tau=0} and \eqref{eq:Rr=0}. When  $\delta(G)=1$ and when
$G$ has noncompact center, we will show  Theorem \ref{thm:01z}
in Subsection \ref{Sec:dg=1}. When  $\delta(G)=1$ and when
$G$ has compact center, we will show that $R_\rho(\sigma)$ is well 
defined such that \eqref{eq:t01z} holds in Section \ref{Sec:qusi}, and we will show \eqref{eq:t02z} in Section \ref{sec:proofthm1}.
\end{proof}

\subsection{Proof of Theorem \ref{thm:01z} when $G$ has noncompact center and 
$\delta(G)=1$}\label{Sec:dg=1}
We assume that $\delta(G)=1$ and that $G$  has 
noncompact center. Let us show the  following refined version of Theorem \ref{thm:01z}. 

\begin{thm}\label{thm:dg=1Rd1}
 There is $\sigma_0>0$ \index{S@$\sigma_0$}such that
\begin{align}\label{eq:seed1}
   \sum_{[\gamma]\in  [\Gamma]- \{1\}}\frac{\left| \chi_{\rm orb}\big(\bbS^1\backslash B_{[\gamma]}\big)\right|}{m_{[\gamma]}}e^{-\sigma_0|a|}<\infty.
\end{align}	
The dynamical zeta function $R_{\rho}(\sigma)$ extends 
meromorphically to $\sigma\in \bC$ such that
\begin{align}
	\label{eq:R=Tdg=1}
	R_{\rho}(\sigma)=\exp\({r 
	\vol(Z)\[e\(TX_{M},\nabla^{TX_{M}} 
	\)\]^{\max}\sigma}\)T\(\sigma^{2}\).
\end{align}
If $\chi'(Z,F)=0$, then $R_{\rho}(\sigma) $ is holomorphic at $\sigma=0$ and 
\begin{align}\label{eq:RTd1}
	R_{\rho}(0)=T(F)^{2}.
\end{align}	
\end{thm}
\begin{proof}
Following \eqref{eq:zetaP}, for $(s,\sigma)\in \bC\times \bR$ such 
that $\Re(s)> m/2$ and $\sigma>0$, put
\begin{multline}\label{eq:T12d1}
  \theta_{\rho}(s,\sigma)=-\Tr\[N^{\Lambda^{\cdot}(T^{*}Z)}\(C^{\fg,Z,\rho}+\sigma\)^{-s}\]\\
  =-\frac{1}{\Gamma(s)}\int_0^\infty \Trs\[N^{\Lambda^{\cdot}(T^{*}Z)}\exp\big(-t (C^{\fg,Z,\rho}+\sigma)\big)\]t^{s-1}dt.
\end{multline}
Let us show that there is $\sigma_{0}>0$ such that 
 \eqref{eq:seed1} holds true and that  for 
 $\sigma>\sigma_{0}$, we have  
 \begin{align}\label{eq:Zjdd1}
\Xi_{\rho}(\sigma)=\frac{\p }{\p s}\theta_{\rho}(0,\sigma^2)  
+r\vol(Z)\[e\(TX_{M},\nabla^{TX_{M}}\)\]^{\max}\sigma.
\end{align}

By \eqref{eq:dgn11}, 
for $(s,\sigma)\in \bC\times \bR$ such that $\Re(s)>1/2$ and 
$\sigma>0$, the function
\begin{align}\label{eq:t1d1}
  \theta_{\rho,1}(s,\sigma)=&-\frac{r\vol(Z)}{\Gamma(s)}\int_0^\infty 
  \Trs^{[1]}\[N^{\Lambda^{\cdot}(T^{*}X)}\exp\(-t\big(C^{\fg,X}+\sigma\big)\)\]t^{s-1}dt
\end{align}
is well defined so that 
\begin{align}\label{eq:t1exd1}
  \theta_{\rho,1}(s,\sigma)=\frac{r\vol(Z)}{2\sqrt{\pi}}\[e\(TX_{M},\nabla^{TX_{M}}\)\]^{\max}\frac{\Gamma(s-1/2)}{\Gamma(s)}\sigma^{1/2-s}.
\end{align}
Therefore, for $\sigma>0$ fixed, the function $s\to 
\theta_{\rho,1}(s,\sigma)$ has a meromorphic extension to $s\in\bC $ 
which is holomorphic at $s=0$ so that
\begin{align}\label{eq:t1fid1}
  \frac{\p}{\p s} \theta_{\rho,1}\(0,\sigma\)=-r\vol(Z)\[e\(TX_{M},\nabla^{TX_{M}}\)\]^{\max}\sigma^{1/2}.
\end{align}

For $(s,\sigma)\in \bC\times \bR$ such that 
$\Re(s)>m/2$ and $\sigma>0$, set 
\begin{align}\label{eq:M1d1}
  \theta_{\rho,2}(s,\sigma)=\theta_{\rho}(s,\sigma)-\theta_{\rho,1}(s,\sigma).
\end{align}
By \eqref{eq:vanishT2},   \eqref{eq:dgn1}, \eqref{eq:V=e}, and 
\eqref{eq:V=0}, for $[\gamma]\in [\Gamma]-\{1\}$, we have 
\begin{multline}\label{eq:thoghmazeta4d1}
\vol\big(\Gamma(\gamma)\backslash 
X(\gamma)\big)\Trs^{[\gamma]}\[N^{\Lambda^{\cdot}(T^{*}X)}\exp\(-tC^{\fg,X}\)\]\\
=- \frac{1}{2\sqrt{\pi t}}\frac{ 
\chi_{\rm orb}\big(\bbS^1\backslash 
B_{[\gamma]}\big)}{m_{[\gamma]}}|a|\exp\(-\frac{|a|^2}{4t}\).
\end{multline}
By \eqref{eq:h11exp} and \eqref{eq:thoghmazeta4d1}, there exist 
$C_1>0$, $C_2>0$, and $C_3>0$ such that for $t>0$, we have
\begin{align}\label{eq:thoghmazetad1}
  \sum_{[\gamma]\in [\Gamma]- \{1\}}\frac{\left| \chi_{\rm 
  orb}\big(\bbS^1\backslash B_{[\gamma]}\big)\right|}{m_{[\gamma]}}|a|\exp\(-\frac{|a|^2}{4t}\)\l C_1\exp\(-\frac{C_2}{t}+C_3t\).
\end{align}
Take $\sigma_0=\sqrt{2C_3}$. 
Since $\rho$ is unitary, by 
\eqref{eq:sel}, \eqref{eq:T12d1},     \eqref{eq:M1d1}, and 
\eqref{eq:thoghmazetad1},
for $(s,\sigma)\in 
\bC\times \bR$ such that $\Re(s)>m/2$ and $\sigma\g\sigma_{0}$, we have 
\begin{multline}\label{eq:t2d1}
\theta_{\rho,2}\(s,\sigma^{2}\)=\frac{1}{2\sqrt{\pi }\Gamma(s)}\int_0^\infty
\sum_{[\gamma]\in [\Gamma]- \{1\}}\Tr[\rho(\gamma)] 
\frac{ \chi_{\rm orb}\big(\bbS^1\backslash 
B_{[\gamma]}\big)}{m_{[\gamma]}}|a|\exp\(-\frac{|a|^2}{4t}-\sigma^{2} 
t\)t^{s-3/2}dt.
\end{multline}
Moreover, for $\sigma\g\sigma_{0}$ fixed, the function
$s\to \theta_{\rho,2}(s,\sigma^{2})$ extends holomorphically  to $ \bC$, so that
\begin{align}\label{eq:massid1}
  \frac{\p}{\p s} \theta_{\rho,2}(0,\sigma^{2})=\frac{1}{2\sqrt{\pi}}\int_0^\infty\sum_{[\gamma]\in [\Gamma]-\{1\}}\Tr[\rho(\gamma)]\frac{\chi_{\rm orb}\big(\bbS^1\backslash B_{[\gamma]}\big)}{m_{[\gamma]}}
  |a|\exp\(-\frac{|a|^2}{4t}-\sigma^{2} t\)\frac{dt}{t^{3/2}}.
\end{align}
Using  the formula\footnote{We give a proof of 
\eqref{eq:elect} when  $B_{1}=B_{2}=1$. Indeed, we have
$$\int_0^\infty\exp\(-\frac{1}{t}-t\)\frac{dt}{t^{{3/2}}}=\frac{1}{2}\int_0^\infty\exp\(-\frac{1}{t}-t\)\(\frac{1}{t^{3/2}}+\frac{1}{t^{1/2}}\)dt.$$
Using the change of variables  $u=t^{1/2}-t^{-1/2}$, we get \eqref{eq:elect}.} 
 that for $B_1>0,B_2\g0$,
\begin{align}\label{eq:elect}
  \int_0^\infty\exp\(-\frac{B_1}{t}-B_2t\)\frac{dt}{t^{{3/2}}}=\sqrt{\frac{\pi}{B_1}}\exp\big(-2\sqrt{B_1B_2}\big),
\end{align}
by \eqref{eq:thoghmazetad1}, \eqref{eq:massid1}, and by Fubini 
Theorem,  we get 
\eqref{eq:seed1}. Also, for $\sigma\g\sigma_{0}$, we have  
\begin{align}\label{eq:t2befud1}
   \frac{\p}{\p s} \theta_{\rho,2}(0,\sigma^2)
=\Xi_{\rho}(\sigma).
\end{align}

By \eqref{eq:t1fid1}, \eqref{eq:M1d1}, and \eqref{eq:t2befud1},  we 
get $\eqref{eq:Zjdd1}$. By taking the exponentials, we get 
\eqref{eq:R=Tdg=1} for $\sigma\g \sigma_{0}$. Since the right-hand 
side of \eqref{eq:R=Tdg=1} is meromorphic on $\sigma\in \bC$,  then 
$R_{\rho}$ has a meromorphic extension to $ \bC$. By \eqref{eq:tors} 
and \eqref{eq:R=Tdg=1}, we get \eqref{eq:RTd1}.
The proof of  
our theorem  is completed.
\end{proof}

\section{Reductive groups $G$ with compact center and $\delta(G)=1$}\label{sec:5}
In this section, we assume that $\delta(G)=1$ and that $G$ has  
compact  center. 
The purpose of this section is to introduce some geometric objects 
associated with  $G$. Their proprieties are proved by  algebraic 
arguments based on the classification of real simple Lie 
algebras $\fg$ with $\delta(\fg)=1$. The results of this section will 
be used in Section \ref{Sec:qusi}, in order to evaluate certain 
orbital integrals. 


This section is organized as follows.
In Subsection \ref{sec:dG=1},  we introduce  a splitting $\fg=\fb\oplus \fm\oplus \fn\oplus 
\ol{\fn}$, associated with  the action of $\fb$ on $\fg$.

In Subsection \ref{sec:Yb}, we construct a natural compact Hermitian 
symmetric space $Y_\fb$, which will be used  in the calculation of 
orbital integrals in Subection \ref{sec:evapri}.

In Subsection \ref{sec:varsi}, we state one key result, which says 
that  the  action of $K_{M}$ 
on  $\fn$  lift to $K$. The purpose of the following subsections is to prove this result. 



 In Subsection \ref{sec:class}, we state a  classification result of 
 real simple Lie algebras $\fg$ with $\delta(\fg)=1$, which asserts 
 that they just contain 
 $\mathfrak{sl}_3(\bR)$ and $\mathfrak{so}(p,q)$ with $pq>1$ odd. This 
 result has already been used by Moscovici-Stanton \cite{MStorsion}.
 
 In Subsections \ref{subsec:SL3} and \ref{subsec:SOPQ}, we study the 
 Lie
 groups $\mathrm{SL}_3(\bR)$ and $\SO^0(p,q)$ with $pq>1$ odd, and 
 the structure of the associated Lie groups $M$, $K_{M}$. 
 
In Subsection \ref{sec:i}, we study the connected component $G_*$ of 
the identity of the isometry  group of $X=G/K$. We show that $G_*$ has a factor $\mathrm{SL}_3(\bR)$ or $\SO^0(p,q)$ with $pq>1$ odd.

Finally, in Subsections \ref{sec:proofnn1}-\ref{sec:G1=0}, we show 
several unproven results stated   in Subsections 
\ref{sec:dG=1}-\ref{sec:varsi}. Most of the 
results  are shown  case by case for the group $\mathrm{SL}_3(\bR)$ 
and $\SO^0(p,q)$ with $pq>1$ odd. We prove the corresponding results for general $G$  using a natural morphism $i_G:G\to G_*$.

\subsection{A splitting  of $\fg$}\label{sec:dG=1}
We use the notation in \eqref{eq:MPMKM1}-\eqref{eq:XM}. Let 
$Z(\fb)\subset G$ 
be the stabilizer of $\fb$ in $G$, and 
let $\fz(\fb)\subset \fg$ be its Lie algebra. 

We define $\fp(\fb)$, $\fk(\fb)$, $\fp^{\bot}(\fb)$, 
$\fk^{\bot}(\fb)$, $\fz^{\bot}(\fb)$ in an obvious way as in Subsection \ref{sec:semi}.  
By  \eqref{eq:mpkd1}, we have
\begin{align}\label{eq:tao1}
&\fp(\fb)=\fb\oplus\fp_{\fm},&\fk(\fb)=\fk_{\fm}.
\end{align}
Also,
\begin{align}\label{eq:mpk1}
 & \fp=\fb\oplus\fp_\fm\oplus\fp^\bot(\fb),&\fk=\fk_\fm\oplus\fk^\bot(\fb).
\end{align}
Let $Z^{0}(\fb)$ \index{Z@$Z^0(\fb)$} be the connected component of 
the identity in $Z(\fb)$. By \eqref{eq:Z0a}, we have
\begin{align}\label{eq:Z0UB}
  Z^0(\fb)=\bR\times M.
\end{align}
The group $K_M$ acts trivially on $\fb$. It also acts on $\fp_\fm$, 
$\fp^\bot(\fb)$, $\fk_\fm$ and $\fk^{\bot}(\fb)$, and  preserves the 
splittings \eqref{eq:mpk1}.




Recall that we have fixed $a_{1}\in \fb$ such that 
$B(a_{1},a_{1})=1$. The choice of $a_{1}$ fixes an orientation of 
$\fb$. 
Let $\fn\subset \fz^{\bot}(\fb)$ be the direct sum of the eigenspaces 
of $\ad(a_{1})$ with the positive eigenvalues.
Set $\ol{\fn}=\theta\fn$. Then $\ol{\fn}$ is the direct sum of the eigenspaces with  negative eigenvalues, and
\begin{align}\label{eq:zbnn}
\fz^{\bot}(\fb)=\fn\oplus \ol{\fn}.\index{N@$\fn,\ol{\fn}$}
\end{align}
Clearly, $Z^0(\fb)$ acts on $\fn$ and $\ol{\fn}$ by  adjoint action. 
Since $K_M$ is fixed by $\theta$, we have  isomorphisms of representations of $K_M$,
\begin{align}\label{eq:n=pbK}
  &X\in \fn\to X-\theta X\in  \fp^\bot(\fb),&X \in \fn\to X+\theta X\in  \fk^\bot(\fb).
\end{align}
In the sequel, if $f\in \fn$, we denote $\ol{f}=\theta f\in \ol{\fn}$.

By \eqref{eq:mpk1} and \eqref{eq:n=pbK}, we have $\dim \fn=\dim 
\fp-\dim \fp_{\fm}-1$. Since $\dim \fp$ is odd and since $\dim \fp_\fm$ is even,  $\dim \fn$ is even. Set
\begin{align}
  l=\frac{1}{2}\dim \fn.\index{L@$l$}
\end{align}
Note that since $G$ has compact center, we have $\fb\not\subset \fz_{\fg}$. 
Therefore,  $\fz^{\bot}(\fb)\neq0$ and $l>0$.

\begin{re}Let $\mathfrak{q}\subset \fg$ \index{Q@$\fq$}be the direct 
	sum of the eigenspaces of $\ad(a_{1})$ with  nonnegative 
	eigenvalues. Then $\mathfrak{q}$ is a proper parabolic subalgebra of $\fg$, with Langlands decomposition $\mathfrak{q}= \fm\oplus \fb\oplus\fn$ \cite[Section VII.7]{KnappLie}. Let $Q\subset G$ be the corresponding parabolic subgroup of $G$, and let $Q=M_QA_QN_Q$ be the corresponding Langlands decomposition. Then $M$ is the connected component of the identity in $M_Q$, and $\fb,\fn$ are  the Lie algebras of $A_Q$ and $N_Q$.
\end{re}

\begin{prop}\label{prop:nnam1}
Any element of $\fb$ acts on $\fn$ and $\ol{\fn}$ as a scalar, i.e., there exists $\alpha\in \fb^*$\index{A@$\alpha$}  such that for $a\in \fb$, $f\in \fn$, we have
\begin{align}\label{eq:ab1}
  &[a,f]=\<\alpha,a\>f, &[a,\ol{f}]=-\<\alpha,a\>\ol{f}.
\end{align}
\end{prop}
\begin{proof}
The proof of Proposition \ref{prop:nnam1}, based on the 
classification theory of real simple Lie algebras, will be given in 
Subsection \ref{sec:proofnn1}.
\end{proof}
Let $a_0\in \fb$ \index{A@$a_0$} be such that
\begin{align}\label{eq:ab3}
\<\alpha,a_0\>=1.
\end{align}

\begin{prop}\label{prop:nnam}
We have
\begin{align}\label{eq:nnam}
&[\fn,\ol{\fn}]\subset \fz(\fb),&[\fn,\fn]=[\ol{\fn},\ol{\fn}]=0.
\end{align}
Also,
\begin{align}\label{eq:Bnn=0}
&  B|_{\fn\times \fn=0}, &B|_{\ol{\fn}\times \ol{\fn}}=0.
\end{align}
\end{prop}
\begin{proof}
By \eqref{eq:ab1}, $a\in \fb$ acts  on $[\fn,\ol{\fn}]$, 
$[\fn,\fn]$, and $[\ol{\fn},\ol{\fn}]$ by multiplication by $0$, 
$2\<\alpha,a\>$, and $-2\<\alpha,a\>$. Equation \eqref{eq:nnam} 
follows. 

If $f_1,f_2\in \fn$, by \eqref{eq:ab1} and \eqref{eq:ab3}, we have
\begin{align}\label{eq:Bf12}
  B(f_1,f_2)=B([a_0,f_1],f_2)=-B(f_1,[a_0,f_2])=-B(f_1,f_2).
\end{align}
From \eqref{eq:Bf12}, we get the first equation of \eqref{eq:Bnn=0}. 
We obtain  the second equation of \eqref{eq:Bnn=0} by the same 
argument. The 
proof of our proposition is completed.
\end{proof}

\begin{re}Clearly, we have
\begin{align}\label{eq:fzbot1}
&\[\fz(\fb),\fz(\fb)\]\subset \fz(\fb).
\end{align}
Since $\fz(\fb)$ preserves $B$ and since $\fz^\bot(\fb)$ is the orthogonal space to $\fz(\fb)$ in $\fg$ with respect to $B$, we have
\begin{align}
 \[\fz(\fb),\fz^\bot(\fb)\]\subset \fz^{\bot}(\fb)
\end{align}
By \eqref{eq:zbnn} and \eqref{eq:nnam}, we get
\begin{align}\label{eq:fzbot}
 \[\fz^{\bot}(\fb),\fz^\bot(\fb)\]\subset \fz(\fb).
\end{align}
We note the similarity between \eqref{eq:kpkp} and   equations \eqref{eq:fzbot1}-\eqref{eq:fzbot}. In the sequel, We call such a pair $(\fz,\fz(\fb))$ a symmetric pair.
\end{re}

For $k\in K_M$, let $M(k)\subset M$ \index{M@$M(k),M^0(k)$}be the  
centralizer of $k$ in $M$, and let $\fm(k)$ \index{M@$\fm(k)$}be its 
Lie algebra. Let $M^0(k)$ be the connected component of the identity 
in $M(k)$. Let  $\fp_\fm(k)$ \index{P@$\fp_\fm(k)$}and 
$\fk_\fm(k)$\index{K@$\fk_\fm(k)$} be the analogues  of $\fp(\gamma)$ and $\fk(\gamma)$ in \eqref{eq:prkr}, so that
\begin{align}\label{eq:mkpk}
  \fm(k)=\fp_\fm(k)\oplus \fk_\fm(k).
\end{align}
Since $k$ is elliptic in $M$,  $M^0(k)$ is  reductive with maximal compact subgroup $K_M^0(k)=M^0(k)\cap K$ and with Cartan decomposition \eqref{eq:mkpk}. Let
\begin{align}
  X_M(k)=M^0(k)/K^0_M(k)\index{X@$X_M(k)$}
\end{align}
 be the corresponding symmetric space. Note that $\delta\big(M^0(k)\big)=0$ and $\dim X_M(k)$ is even.

Clearly, if $\gamma=e^ak^{-1}\in H$ with $a\in \fb$, $a\neq0$, $k\in T$, then
\begin{align}\label{eq:Kr=kMk}
  &\fp(\gamma)=\fp_\fm(k),&\fk(\gamma)=\fk_\fm(k),&&Z^{a,\bot,0}(\gamma)=M(k),&&&K^0(\gamma)=K^0_M(k).
\end{align}
\begin{prop}\label{prop:det1/2fi}
For $\gamma=e^ak^{-1}\in H$ with $a\in \fb,a\neq0,k\in T$, we have
\begin{align}\label{eq:det1/2fi}
\begin{aligned}
  \left|\mathrm{det}\big(1-\Ad(\gamma)\big)|_{\fz_0^\bot}\right|^{1/2}=&\sum_{j=0}^{2l}(-1)^j\Tr^{\Lambda^j(\fn^*)}\[\Ad\(k^{-1}\)\]e^{(l-j)\<\alpha,a\>}\\
=&\sum_{j=0}^{2l}(-1)^j\Tr^{\Lambda^j(\fn^*)}\[\Ad\(k^{-1}\)\]e^{(l-j)|\alpha||a|}.
\end{aligned}
\end{align}
\end{prop}
\begin{proof}
We claim that
  \begin{align}\label{eq:det1/2}
  \left|\mathrm{det}\big(1-\Ad(\gamma)\big)|_{\fz_0^\bot}\right|^{1/2}=e^{l\<\alpha,a\>}\mathrm{det}\big(1-\Ad(\gamma)\big)|_{\ol{\fn}}.
\end{align}
Indeed, since $\dim \ol{\fn}$ is even, the right-hand side of \eqref{eq:det1/2} is positive. By \eqref{eq:zbnn}, we have
\begin{align}\label{eq:toot1}
\mathrm{det}\big(1-\Ad(\gamma)\big)|_{\fz_0^\bot}=\mathrm{det}\big(1-\Ad(\gamma)\big)|_{\fn}\mathrm{det}\big(1-\Ad(\gamma)\big)|_{\ol{\fn}}.
\end{align}
Since  $\ol{\fn}=\theta\fn$, we have
\begin{align}\label{eq:toot9}
\mathrm{det}\big(1-\Ad(\gamma)\big)|_{\fn}=\mathrm{det}\big(1-\Ad(\theta\gamma)\big)|_{\ol{\fn}}=\mathrm{det}\big(\Ad(\theta\gamma)\big)|_{\overline{\fn}}\mathrm{det}\big(\Ad(\theta\gamma)^{-1}-1\big)|_{\ol{\fn}}.
\end{align}
Since $\dim \ol{\fn}=2l$ is even, and since $(\theta\gamma)^{-1}=e^{a}k$ and $k$  acts unitarily on $\fn$,  by \eqref{eq:ab1} and \eqref{eq:toot9}, we have
\begin{align}\label{eq:toot2}
\mathrm{det}\big(1-\Ad(\gamma)\big)|_{\fn}=e^{2l\<\alpha,a\>}\mathrm{det}\big(1-\Ad(e^{a}k)\big)|_{\ol{\fn}}=e^{2l\<\alpha,a\>}\mathrm{det}\big(1-\Ad(\gamma)\big)|_{\ol{\fn}}.
\end{align}
By \eqref{eq:toot1} and \eqref{eq:toot2}, we get \eqref{eq:det1/2}.

Classically,
\begin{align}\label{eq:det1/2=n}
  \mathrm{det}\big(1-\Ad(\gamma)\big)|_{\ol{\fn}}=\sum_{j=0}^{2l}(-1)^j\Tr^{\Lambda^j(\ol{\fn})}\[\Ad\(k^{-1}\)\]e^{-j\<\alpha,a\>}.
\end{align}
Using the isomorphism of $K_{M}$-representations $\fn^*\simeq \ol{\fn}$,
by \eqref{eq:det1/2}, \eqref{eq:det1/2=n}, we get the first equation of \eqref{eq:det1/2fi} and the second equation of \eqref{eq:det1/2fi} if $a$ is positive in $\fb$. For the case $a$ is negative in $\fb$, it is enough to remark that
replacing $\gamma$ by $\theta\gamma$ does not change the left-hand side of \eqref{eq:det1/2fi}.
The proof of our proposition is completed.
\end{proof}

\subsection{A compact Hermitian symmetric space $Y_{\fb}$}\label{sec:Yb}
%


Let $\fu(\fb)\subset \fu$ and $\fu_\fm\subset \fu$ be  the compact form of $\fz(\fb)$ and $\fm$. \index{U@$\fu(\fb),\fu(\fm)$}Then,
\begin{align}\label{eq:99}
& \fu(\fb)=\sqrt{-1}\fb\oplus \fu_\fm,&\fu_{\fm}=\sqrt{-1}\fp_\fm\oplus \fk_\fm.
\end{align}
Since $\delta(M)=0$,  $M$ has compact center. By \cite[Proposition 5.3]{Knappsemi}, let $U_M$  be the compact form of $M$.

Let $U(\fb)\subset U, A_0\subset U$ \index{U@$U(\fb),U_M$}\index{A@$A_0$} be the  connected subgroups of $U$ associated with Lie algebras $\fu(\fb)$,$\sqrt{-1}\fb$.
By \eqref{eq:99}, $A_0$ is in the center of $U(\fb)$, and
\begin{align}\label{eq:Ub=S1UM}
U(\fb)=A_0 U_M.
\end{align}
By \cite[Corollaire 4.51]{KnappLie},   the stabilizer of $\fb$ in $U$ is a closed connected subgroup of $U$, and so it coincides with  $U(\fb)$.

\begin{prop}\label{prop:A0S1}
The group $A_0$ is closed in $U$, and is diffeomorphic to a
circle $\bbS^1$. 
\end{prop}
\begin{proof}
The proof of Proposition \ref{prop:A0S1}, based on the classification 
theory of real simple Lie algebras, will be given in Subsection \ref{sec:proofnn1}.
\end{proof}

Set
\begin{align}\label{eq:Yb0}
  Y_{\fb}=U/U(\fb).\index{Y@$Y_\fb$}
\end{align}
We will see that $Y_\fb$ is a compact Hermitian symmetric space.

Recall that the bilinear form $-B$ induces  an $\Ad(U)$-invariant metric on $\fu$. Let $\fu^{\bot}(\fb)$ be the orthogonal space to $\fu(\fb)$ in $\fu$,\index{U@$\fu^\bot(\fb)$} such that
\begin{align}\label{eq:deum}
  \fu=\fu(\fb)\oplus\fu^{\bot}(\fb).
\end{align}
Also, we have
\begin{align}
 \fu^{\bot}(\fb) =\sqrt{-1}\fp^{\bot}(\fb)\oplus \fk^{\bot}(\fb)
\end{align}

By \eqref{eq:fzbot1}-\eqref{eq:fzbot}, we have
\begin{align}
  \label{eq:fubot}
&\[\fu(\fb),\fu(\fb)\]\subset \fu(\fb),& \[\fu(\fb),\fu^\bot(\fb)\]\subset \fu^{\bot}(\fb),&&  \[\fu^{\bot}(\fb),\fu^\bot(\fb)\]\subset \fu(\fb).
\end{align}
Thus, $(\fu,\fu(\fb))$ is a symmetric pair.

Set
\begin{align}\label{eq:Ja0}
  J=\sqrt{-1}\ad(a_0)|_{\fu^{\bot}({\fb})}\in \End\(\fu^{\bot}({\fb})\).\index{J@$J$}
\end{align}
By \eqref{eq:ab1}-\eqref{eq:Bnn=0}, $J$ is a $U(\fb)$-invariant complex structure on $\fu^{\bot}(\fb)$, which preserves the restriction $B|_{\fu^\bot(\fb)}$.  Moreover, $\fn_\bC=\fn\otimes_\bR\bC$ and $\ol{\fn}_\bC=\ol{\fn}\otimes_\bR\bC$ are the eigenspaces of $J$ associated with the eigenvalues $\sqrt{-1}$ and $-\sqrt{-1}$, such that
\begin{align}\label{eq:ubotnn}
  \fu^\bot(\fb)\otimes_\bR\bC=\fn_\bC\oplus\ol{\fn}_\bC.
\end{align}
The bilinear form $-B$ induces a Hermitian metric on $\fn_\bC$ such that for $f_1,f_2\in \fn_\bC$,\index{1@$\langle,\rangle_{\fn_\bC}$}
\begin{align}\label{eq:Hern}
  \<f_1,f_2\>_{\fn_\bC}=-B(f_1,\ol{f_2}).
\end{align}

Since $J$ commutes with the action of $U(\fb)$, $U(\fb)$ preserves the splitting \eqref{eq:ubotnn}. Therefore, $U(\fb)$ acts on $\fn_\bC$ and $\ol{\fn}_\bC.$ In particular,  $U(\fb)$ acts on $\Lambda^\cdot(\ol{\fn}_\bC^*)$. If $S^{\fu^\bot(\fb)}$ \index{S@$S^{\fu^\bot(\fb)}$}is the spinor of $(\fu^\bot(\fb), -B)$, by \cite{Hitchin74}, we have the isomorphism of representations of $U(\fb)$,
\begin{align}\label{eq:Rnspinc}
 \Lambda^\cdot(\ol{\fn}_\bC^*)\simeq S^{\fu^\bot(\fb)}\otimes \det(\fn_\bC)^{1/2}.
\end{align}
Note that $M$ has compact center $Z_{M}$. By \cite[Proposition 
5.5]{Knappsemi}, $M$ is a product of a connected semisimple Lie group 
and the connected component of the identity in $Z_M$. Since both of 
these two groups act trivially on 
$\det(\fn)$, the same is true for  $M$.  Since the action of $U_{M}$ on  $\fn_\bC$ can be obtained by 
the restriction of the induced action of  $M_\bC$  on $\fn_\bC$, 
$U_M$ acts trivially on $\det(\fn_{\bC})$. By \eqref{eq:Rnspinc}, we have the isomorphism of  representations of $U_M$,
\begin{align}\label{eq:Rnspin}
   \Lambda^\cdot(\ol{\fn}_\bC^*)\simeq  S^{\fu^\bot(\fb)}.
\end{align}

As in Subsection \ref{sec:sym},
let $\omega^{\fu}$ be the canonical left invariant $1$-form on $U$ with values in $\fu$, and let $\omega^{\fu(\fb)}$ and $\omega^{\fu^{\bot}(\fb)}$ be the $\fu(\fb)$ and $\fu^{\bot}(\fb)$
components of $\omega^{\fu}$, so that
\begin{align}
  \omega^{\fu}=\omega^{\fu(\fb)}+\omega^{\fu^{\bot}(\fb)}.\index{O@$ \omega^{\fu},\omega^{\fu(\fb)},\omega^{\fu^{\bot}(\fb)}$}
\end{align}
Then, $U\to  Y_{\fb}$ is a $U(\fb)$-principle bundle, equipped with a connection form $\omega^{\fu(\fb)}$. Let $\Omega^{\fu(\fb)}$ \index{O@$\Omega^{\fu(\fb)}$} be the curvature form. As in \eqref{eq:Ok}, we have
\begin{align}\label{eq:Rub}
 \Omega^{\fu(\fb)}=-\frac{1}{2}\[\omega^{\fu^{\bot}(\fb)},\omega^{\fu^{\bot}(\fb)}\].
\end{align}

The real tangent bundle
\begin{align}\label{eq:TYb}
TY_{\fb}=U\times_{U(\fb)}\fu^{\bot}(\fb)
\end{align}
is equipped with a Euclidean metric and a Euclidean connection $\nabla^{TY_\fb}$, which coincides with the Levi-Civita connection. By \eqref{eq:Ja0},
$J$ induces an almost complex structure on $TY_\fb$.
Let $T^{(1,0)}Y_{\fb}$ and $T^{(0,1)}Y_{\fb}$ be the holomorphic and anti-holomorphic tangent bundles. Then
\begin{align}\label{eq:hah}
&T^{(1,0)}Y_{\fb}=U\times_{U(\fb)}\fn_\bC,  &T^{(0,1)}Y_{\fb}=U\times_{U(\fb)}\ol{\fn}_\bC.
\end{align}
By \eqref{eq:nnam} and \eqref{eq:hah}, $J$ is integrable.

The form $-B(\cdot,J\cdot)$ induces a K\"ahler form $\omega^{Y_\fb}\in \Omega^2(Y_\fb)$ on $Y_\fb$. Clearly, $\omega^{Y_\fb}$ is closed, therefore $(Y_\fb,\omega^{Y_\fb})$ is a K\"ahler manifold. Let $f_1,\ldots,f_{2l} \in \fn$ be such that
\begin{align}
  -B(f_i,\ol{f}_j)=\delta_{ij}.
\end{align}
Then $f_1,\ldots,f_{2l} $ is an orthogonal basis of $\fn_\bC$ with respect to $\<\cdot,\cdot\>_{\fn_\bC}$. Let $f^1,\ldots,f^{2l}$ be the dual base of $\fn_\bC^*$. The K\"ahler form $\omega^{Y_\fb}$ on $Y_\fb$ is given by\index{O@$ \omega^{Y_\fb}$}
\begin{align}\label{eq:kahher}
  \omega^{Y_\fb}=-\sum_{1\l i,j\l 2l}B(f_i,J\ol{f}_j)f^i\ol{f}^j=-\sqrt{-1}\sum_{1\l i\l 2l}f^i\ol{f}^i.
  \end{align}

Let us give a more explicit description of $Y_\fb$, although this description will not be needed in the following sections.
\begin{prop}\label{eq:propYb}The homogenous space  $Y_{\fb}$ is an irreducible compact  Hermitian symmetric space of  the type AIII or BDI.
\end{prop}
\begin{proof}
The proof of Proposition \ref{eq:propYb}, based on the classification 
theory of real simple Lie algebras, will be given in Subsection \ref{sec:propYb}.
\end{proof}

 Since $U_\fm$ acts on $\fu_\fm$ and $A_0$ acts trivially on $\fu_\fm$, by \eqref{eq:Ub=S1UM}, then $U(\fb)$ acts on $\fu_\fm$. Put\index{N@$N_\fb$}
\begin{align}\label{eq:defFb}
  N_\fb=U\times_{U(\fb)} \fu_{\fm}.
\end{align}
Then, $N_\fb$ is  a Euclidean vector bundle on $Y_\fb$ equipped  with a metric 
connection $\nabla^{N_\fb}$. We equip  the trivial connection 
$\nabla^{\sqrt{-1}\fb}$ on the trivial line bundle $\sqrt{-1}\fb$ on 
$Y_{\fb}$. 
%
%
Since 
$U(\fb)$ preserves the first splitting in \eqref{eq:99}, we have
\begin{align}\label{eq:AA4}
\sqrt{-1}\fb\oplus N_\fb=U\times_{U_{\fb}} U(\fb).
\end{align}
Moreover, the induced connection is given by 
\begin{align}\label{eq:AA5}
\nabla^{\sqrt{-1}\fb\oplus N_\fb}= \nabla^{\sqrt{-1}\fb}\oplus 
\nabla^{N_\fb}.
\end{align}

By  \eqref{eq:deum}, \eqref{eq:TYb}, and \eqref{eq:AA4}, we have
\begin{align}\label{eq:TGAMF1}
    TY_{\fb}\oplus \sqrt{-1}\fb\oplus N_\fb =\fu,
\end{align}
where $\fu$ stands for  the corresponding trivial bundle on $Y_\fb$.
\begin{prop}
  The following identity of closed forms holds on $Y_\fb$:
\begin{align}\label{eq:AA1}
\widehat{A}\(TY_{\fb},\nabla^{TY_{\fb}}\)\widehat{A}\(N_\fb,\nabla^{N_\fb}\)=1.
\end{align}
\end{prop}
\begin{proof}
Proceeding as in \cite[Proposition 7.1.1]{B09}, by \eqref{eq:deum}, 
\eqref{eq:TYb}, and \eqref{eq:AA4}, we have 
\begin{align}\label{eq:AA2}
\widehat{A}\(TY_{\fb},\nabla^{TY_{\fb}}\)\widehat{A}\(\sqrt{-1}\fb\oplus N_\fb,\nabla^{\sqrt{-1}\fb\oplus N_\fb}\)=1.
\end{align}
By  \eqref{eq:AA5}, we have
\begin{align}\label{eq:AA3}
\widehat{A}\(\sqrt{-1}\fb\oplus N_\fb,\nabla^{\sqrt{-1}\fb\oplus N_\fb}\)=\widehat{A}\(N_\fb,\nabla^{N_\fb}\).
\end{align}
By \eqref{eq:AA2} and (\ref{eq:AA3}), we get (\ref{eq:AA1}). The 
proof of our proposition is completed. 
\end{proof}

Recall that the curvature form $\Omega^{\fu(\fb)}$ is a $2$-form on $Y_\fb$
with values in $U\times_{U(\fb)}\fu(\fb)$. Recall that $a_0\in \fb$ is defined in \eqref{eq:ab3}. Let $ \Omega^{\fu_{\fm}}$ \index{O@$\Omega^{\fu_{\fm}}$} be the $\fu_\fm$-component of $\Omega^{\fu(\fb)}$.  By
\eqref{eq:ab3}, \eqref{eq:Rub} and \eqref{eq:kahher}, we have
\begin{align}\label{eq:Rububb}
  \Omega^{\fu(\fb)}=\sqrt{-1}\frac{a_0}{|a_0|^2}\otimes \omega^{Y_{\fb}} + \Omega^{\fu_{\fm}}.
\end{align}
By \eqref{eq:Rububb}, the curvature of $(N_\fb,\nabla^{N_\fb})$ \index{R@$R^{N_\fb}$} is given by
\begin{align}\label{eq:RN}
  R^{N_\fb}=\ad\(\Omega^{\fu(\fb)}\)\big|_{\fu_\fm}=\ad\(\Omega^{\fu_\fm}\)\big|_{\fu_\fm}.
\end{align}
Also, $B\(\Omega^{\fu(\fb)},\Omega^{\fu(\fb)}\)$ and 
$B\(\Omega^{\fu_\fm},\Omega^{\fu_\fm}\)$ are well defined $4$-forms 
on $Y_\fb$. We have an analogue of \cite[eq. (7.5.19)]{B09}.

\begin{prop}\label{prop:BOO=0}
  The following identities hold:
\begin{align}\label{eq:BOO=0}
  &B\(\Omega^{\fu(\fb)},\Omega^{\fu(\fb)}\)=0,&B\(\Omega^{\fu_\fm},\Omega^{\fu_\fm}\)=\frac{\omega^{Y_{\fb},2}}{|a_0|^2}.
\end{align}
\end{prop}
\begin{proof}
 If $e_1,\cdots,e_{4l}$ is an orthogonal  basis of $\fu^\bot(\fb)$, by \eqref{eq:Rub}, we have
\begin{multline}\label{eq:BOO1}
  B\(\Omega^{\fu(\fb)},\Omega^{\fu(\fb)}\)
=\frac{1}{4}\sum_{1\l i,j,i',j'\l 4l}B\big([e_i,e_j],[e_{i'},e_{j'}]\big)e^ie^je^{i'}e^{j'}\\
=\frac{1}{4}\sum_{1\l i,j,i',j'\l 4l}B\big([[e_i,e_j],e_{i'}],e_{j'}\big)e^ie^je^{i'}e^{j'}.
\end{multline}
Using the Jacobi identity and \eqref{eq:BOO1}, we get the first equation of \eqref{eq:BOO=0}.

The second equation of \eqref{eq:BOO=0} is a consequence of 
\eqref{eq:Rububb} and the first equation of \eqref{eq:BOO=0}. The 
proof of our proposition is completed.
\end{proof}


\subsection{Auxiliary virtual representations of $K$}\label{sec:varsi}
Let $RO(K_M)$ and $RO(K)$  \index{R@$RO(K_M),RO(K)$}be  the real 
representation rings of $K_M$ and $K$.  Let $\iota:K_M\to K$ 
\index{I@$\iota$}be the injection.  We denote by
\begin{align}\label{eq:irkkm}
  \iota^*: RO(K)\to RO(K_M)
\end{align}
the induced morphism of rings. Since $K_M$ and $K$ have the same 
maximal torus $T$, $\iota^{*}$ is injective.


\begin{prop}\label{prop:sumsum}
  The following identity in  $RO(K_M)$ holds:
\begin{align}\label{eq:sumsum}
\iota^*\( \sum_{i=1}^{m} (-1)^{i-1}i \Lambda^i(\fp^*)\)=\sum_{i=0}^{\dim \fp_\fm}\sum_{j=0}^{2l}(-1)^{i+j}\Lambda^{i}(\fp_\fm^*)\otimes \Lambda^j(\fn^*). \end{align}
\end{prop}
\begin{proof}
 For a representation $V$ of $K_M$, we use the multiplication notation introduced by Hirzebruch. Put
\begin{align}
\Lambda_y(V)=\sum_i y^i\Lambda^i(V)
\end{align}
a polynomial of $y$ with coefficients in $RO(K_M)$. In particular,
\begin{align}
&\Lambda_{-1}(V)=\sum_i (-1)^i\Lambda^i(V),&\Lambda'_{-1}(V)=\sum_i (-1)^{i-1}i\Lambda^i(V).
\end{align}
Denote by $\mathbf{1}$ the trivial representation.
Since $\Lambda_1(\mathbf{1})=0,\Lambda'_{-1}(\mathbf{1})=\mathbf{1}$, we get
\begin{align}\label{eq:hiz}
\Lambda'_{-1}(V\oplus \mathbf{1})=\Lambda_{-1}(V).
\end{align}

By \eqref{eq:mpk1}, \eqref{eq:n=pbK}, and by the fact that $K_M$ acts trivially on $\fb$, we have the isomorphism of $K_M$-representations
\begin{align}\label{eq:p1pn}
  \fp\simeq \mathbf{1}\oplus \fp_\fm\oplus \fn.
\end{align}
Taking $V=\fp_\fm\oplus \fn$, by \eqref{eq:hiz} and  \eqref{eq:p1pn},
we get \eqref{eq:sumsum}. The proof of our proposition is completed.
\end{proof}

The following theorem is crucial.
\begin{thm}\label{thm:keyrep}
The adjoint representations of $K_M$ on $\fn$  has a unique lift in $RO(K)$.
\end{thm}
\begin{proof}
The injectivity of $\iota^*$ implies the uniqueness. The proof of the 
existence of the lifting of $\fn$, based on  the  classification 
theorem of real simple Lie algebras, will be given in Subsection \ref{sec:liftn}.
\end{proof}

\begin{cor}\label{cor:key}
For $i,j\in\mathbf{N}$,  the adjoint representations of $K_M$ on $\Lambda^i(\fp^*_\fm)$ and $\Lambda^j(\fn^*)$  have unique lifts in $RO(K)$.
\end{cor}
\begin{proof}
	As before, it is enough to show the existence of lifts.
	Since the representation of $K_{M}$ on $\fn$ lifts to $K$, the same 
is true for the $\Lambda^{j}(\fn^{*})$. By \eqref{eq:p1pn}, 
this extends to the $\Lambda^{i}(\fp_{\fm}^{*})$. The proof of our 
corollary is completed. 
\end{proof}

Denote by $\eta_j$\index{E@$\eta_j$} the adjoint representation of $M$ on $\Lambda^j(\fn^*)$.
Recall that by \eqref{eq:ubotnn}, $U(\fb)$ acts on $\fn_\bC$. 
Recall also that $C^{\fu_\fm,\eta_j}\in \End\big(\Lambda^j(\fn^*_\bC)\big)$, \index{C@$C^{\fu_\fm,\eta_j}$} $C^{\fu(\fb),\fu^\bot(\fb)}\in \End(\fu^\bot(\fb))$ \index{C@$C^{\fu(\fb),\fu^\bot(\fb)}$} are defined in \eqref{eq:ckkckp}.

\begin{prop}\label{prop:bggbuu}For $0\l j\l 2l$,  $C^{\fu_\fm,\eta_j}$ is a scalar so that
\begin{align}\label{eq:BggBuu}
C^{\fu_\fm,\eta_j}=\frac{1}{8}\Tr^{\fu^\bot(\fb)}\[C^{\fu(\fb),\fu^\bot(\fb)}\]+(j-l)^2|\alpha|^2.
\end{align}
\end{prop}
\begin{proof}
 Equation \eqref{eq:BggBuu} was proved in   \cite[Lemma 
 2.5]{MStorsion}. 
 We give here  a more conceptual proof. 


Recall that 
$(\fu,\fu(\fb))$ is a compact symmetric pair. Let $S^{\fu^\bot(\fb)}$ 
be the $\fu^{\bot}(\fb)$-spinors \cite[Section 7.2]{B09}. Let 
$C^{\fu(\fb),S^{\fu^\bot(\fb)}}$ be the Casimir element  of $\fu(\fb)$ acting on  $S^{\fu^\bot(\fb)}$ defined as in \eqref{eq:ckkckp}.  By \cite[eq. (7.8.6)]{B09},
$C^{\fu(\fb),S^{\fu^\bot(\fb)}}$ is a scalar such that
\begin{align}\label{eq:fubfu}
 C^{\fu(\fb),S^{\fu^\bot(\fb)}}=\frac{1}{8}\Tr\[C^{\fu(\fb),\fu^\bot(\fb)}\].\end{align}

Let $C^{\fu_\fm,\Lambda^\cdot(\ol{\fn}^*_\bC)}$ be the Casimir element of $\fu_\fm$ acting on $\Lambda^\cdot(\ol{\fn}^*_\bC)$.
By \eqref{eq:Cg}, \eqref{eq:Rnspinc} and \eqref{eq:Rnspin}, we have
\begin{align}
 C^{\fu(\fb),S^{\fu^\bot(\fb)}}= C^{\fu_\fm,\Lambda^\cdot(\ol{\fn}^*_\bC)}-\(\Ad(a_1)|_{\Lambda^\cdot(\ol{\fn}_\bC^*)\otimes \det^{-1/2}(\fn_\bC)}\)^2.
\end{align}
By \eqref{eq:ab1}, we have
\begin{align}\label{eq:bkaka}
  \Ad(a_1)|_{\Lambda^j(\ol{\fn}_\bC^*)\otimes \det^{-1/2}(\fn_\bC)}=(j-l)|\alpha|.
\end{align}
By \eqref{eq:fubfu}-\eqref{eq:bkaka}, we get \eqref{eq:BggBuu}. The 
proof of our proposition is completed.
\end{proof}

Let $\gamma=e^{a}k^{-1}\in G$ be such that \eqref{eq:asr} holds. 
Since  $\Lambda^{\cdot}(\fp^{*}_{\fm})\in RO(K)$, for $Y\in \fk(\gamma)$, 
$\Trs^{\Lambda^{\cdot}(\fp^{*}_{\fm})}\[k^{-1}\exp(-iY)\]$ is well 
defined. 
 We have an analogue of \eqref{eq:vanishT}. 
\begin{prop}\label{prop:G1=0}If $\dim \fb(\gamma)\g2$, then for $Y\in \fk(\gamma)$, we have 
\begin{align}\label{eq:G1=0}
 \Trs^{\Lambda^{\cdot}(\fp^{*}_{\fm})}\[k^{-1}\exp(-iY)\]=0. 
\end{align}
\end{prop}
\begin{proof}
	The proof of Proposition \ref{prop:G1=0}, based on the classification 
theory of real simple Lie algebras, will be given in Subsection \ref{sec:G1=0}.
%
%
%
\end{proof}

\subsection{A classification of  real reductive Lie algebra $\fg$ with $\delta(\fg)=1$}\label{sec:class}
Recall that $G$ is a real reductive group with compact center, such 
that $\delta(G)=1$.


\begin{thm}\label{thm:cla}
  We have a decomposition of Lie algebras \index{G@$\fg_1,\fg_2$}
\begin{align}\label{eq:g=g1g2}
  \fg=\fg_1\oplus\fg_2,
\end{align}
where
\begin{align}
  \fg_1= \mathfrak{sl}_3(\bR) \hbox{ or } \mathfrak{so}(p,q),
\end{align}
with $pq>1$ odd, and  $\fg_2$ is real reductive  with $\delta(\fg_2)=0$.
\end{thm}
\begin{proof}
Since $G$ has compact center, by \eqref{eq:ZG}, $\fz_\fp=0$.
By \eqref{eq:gzgg}, we have
\begin{align}
 \delta([\fg,\fg])=1.
\end{align}
As in \cite[Remark 7.9.2]{B09}, by the classification theory  of real simple Lie algebras, we have
\begin{align}\label{eq:gssg1g2}
[\fg,\fg]=\fg_1\oplus \fg_2',
\end{align} where
\begin{align}
  \fg_1=\mathfrak{sl}_3(\bR) \hbox{ or } \mathfrak{so}(p,q),
\end{align}
with $pq>1$ odd, and where $\fg'_2$ is semisimple  with $\delta(\fg'_2)=0$. Take
\begin{align}\label{eq:gssg2}
  \fg_2=\fz_\fk\oplus \fg_2'.
\end{align}
By \eqref{eq:g=zgg}, \eqref{eq:gssg1g2}-\eqref{eq:gssg2}, we get \eqref{eq:g=g1g2}.
The proof of our theorem  is completed.
\end{proof}

%
%
%

\subsection{The group $\mathrm{SL}_3(\bR)$}\label{subsec:SL3}
In this subsection, we assume that $G=\mathrm{SL}_3(\bR)$, so that $K=\mathrm{SO}(3)$. We have
\begin{align}\label{eq:pksl3}
\begin{aligned}
 &\fp=\left\{\left(
          \begin{array}{ccc}
            x & a_1 & a_2 \\
            a_1 & y & a_3 \\
            a_2 & a_3 & -x-y \\
          \end{array}
        \right):x,y, a_1,a_2, a_3\in \bR
\right\},\\
 &\fk=\left\{\left(
          \begin{array}{ccc}
            0 & a_1 & a_2 \\
            -a_1 & 0 & a_3 \\
            -a_2 & -a_3 & 0 \\
          \end{array}
        \right): a_1,a_2, a_3\in \bR
\right\}.
\end{aligned}
\end{align}
Let
\begin{align}\label{eq:TSL3}
 T=\left\{\left(
          \begin{array}{cc}
            A & 0  \\
            0 & 1 \\
          \end{array}
        \right):A\in \mathrm{SO}(2)\right\}\subset  K
\end{align}
 be a maximal torus of $K$.

By  \eqref{eq:defb}, \eqref{eq:pksl3} and \eqref{eq:TSL3}, we have
\begin{align}\label{eq:fbsl3}
  \fb=\left\{\left(
          \begin{array}{ccc}
            x & 0 & 0 \\
            0 & x & 0 \\
            0 & 0 & -2x \\
          \end{array}
        \right): x\in \bR
\right\}\subset \fp.
\end{align}
By \eqref{eq:fbsl3}, we get
\begin{align}\label{eq:pmkmsl3}
 & \fp_\fm=\left\{\left(
          \begin{array}{ccc}
            x & a_1 & 0 \\
            a_1 & -x & 0 \\
            0 & 0 & 0 \\
          \end{array}
        \right):x, a_1\in \bR
\right\},&\fp^\bot(\fb)=\left\{\left(
          \begin{array}{ccc}
            0 & 0 & a_2 \\
            0 & 0 & a_3 \\
            a_2 & a_3 & 0 \\
          \end{array}
        \right):a_2, a_3\in \bR
\right\}.
\end{align}
Also,
\begin{align}\label{eq:MKMsl3}
 &\fk_\fm=\ft, &K_M=T,&
 &M=\left\{\left(
          \begin{array}{cc}
            A & 0  \\
            0 & 1 \\
          \end{array}
        \right):A\in \mathrm{SL}_2(\bR)\right\}.
\end{align}

By \eqref{eq:fbsl3}, we can orient $\fb$ by $x>0$. Thus,
\begin{align}\label{eq:nsl3}
  \fn=\left\{\left(
          \begin{array}{ccc}
            0 & 0 & a_2 \\
            0 & 0 & a_3 \\
            0 & 0 & 0 \\
          \end{array}
        \right):a_2, a_3\in \bR
\right\}.
\end{align}
By \eqref{eq:fbsl3} and \eqref{eq:nsl3}, since for $x\in \bR, a_{2}\in \bR, 
a_{3}\in \bR$, 
\begin{align}\label{eq:hasl3}
\[\left(
          \begin{array}{ccc}
            x & 0 & 0 \\
            0 & x & 0 \\
            0 & 0 & -2x \\
          \end{array}
        \right), \left(
          \begin{array}{ccc}
            0 & 0 & a_2 \\
            0 & 0 & a_3 \\
            0 & 0 & 0 \\
          \end{array}
        \right)\]=3x\left(
          \begin{array}{ccc}
            0 & 0 & a_2 \\
            0 & 0 & a_3 \\
            0 & 0 & 0 \\
          \end{array}
        \right),
	\end{align}
we find that $\fb$ acts on $\fn$ as a scalar. 

  Denote by $ \rm{Isom}^0(G/K)$ the connected component of the  identity of the isometric group of $X=G/K$. Since $G$ acts isometrically on $G/K$, we have the morphism of groups\index{I@$i_G$}
\begin{align}\label{eq:defig}
i_G:G\to \rm{Isom}^0(G/K).
\end{align}
\begin{prop}\label{prop:SLeff}
The morphism $i_G$ is an isomrphism, i.e.,
\begin{align}\label{eq:isoSL3}
\mathrm{SL}_3(\bR)\simeq\mathrm{Isom}^0\big(\mathrm{SL}_3(\bR)/\SO(3)\big).
\end{align}
\end{prop}
\begin{proof}By \cite[Theorem V.4.1]{HelgasonSymmetric}, it is enough 
	to show that $K$ acts on $\fp$ effectively. Assume that $k\in K$ 
	acts on $\fp$ as the identity. Thus, $k$ fixes the elements of $\fb$. As in \eqref{eq:MKMsl3}, there is $A\in \mathrm{GL}_2(\bR)$ such that
\begin{align}\label{eq:kAO2}
  k=\left(
      \begin{array}{cc}
        A & 0 \\
        0 & \det^{-1}(A) \\
      \end{array}
    \right).
\end{align}
Since $k$ fixes also the elements of $\fp^\bot(\fb)$, by \eqref{eq:pmkmsl3} and \eqref{eq:kAO2}, we get $A=1$.
Therefore, $k=1$. The proof of our proposition  is completed.
\end{proof}

\subsection{The group $G=\mathrm{SO}^0(p,q)$ with $pq>1$ odd}\label{subsec:SOPQ}
In this subsection, we assume that $G=\mathrm{SO}^0(p,q)$, so that $K=\mathrm{SO}(p)\times \mathrm{SO}(q)$,
with $pq>1$ odd.

In the sequel, if $l,l'\in \mathbf{N}^*$, let $\mathrm{Mat}_{l,l'}(\bR)$ be the space of real matrices  of $l$ raws and $l'$ columns. If $L\subset \mathrm{Mat}_{l,l}(\bR)$ is a matrix group, we denote by $\sigma_l$ the standard representation of $L$ on $\bR^l$. 
We have
\begin{align}
&  \fp=\left\{\left(
          \begin{array}{cc}
            0 & B \\
            B^t & 0 \\
          \end{array}
        \right)
:B\in \mathrm{Mat}_{p,q}(\bR)\right\},&\fk=\left\{\left(
          \begin{array}{cc}
            A & 0 \\
            0 & D \\
          \end{array}
        \right)
:A\in \mathfrak{so}(p),D\in \mathfrak{so}(q)\right\}.
\end{align}

Let
\begin{align}
  T_{p-1}=\left\{\left(
          \begin{array}{ccc}
            A_1
 & 0 & 0 \\
            0 & \ddots & 0 \\
            0 & 0 & A_{(p-1)/2}
 \\
          \end{array}
        \right):A_1,\ldots, A_{(p-1)/2}\in \SO(2)
\right\}\subset \SO(p-1)
\end{align}
 be a maximal torus of $\SO(p-1)$.
Then,
\begin{align}\label{eq:tsopq}
  T=\left\{\left(
               \begin{array}{ccc}
                 A & 0 & 0 \\
                 0 & \left(
                       \begin{array}{cc}
                         1 & 0 \\
                         0 & 1 \\
                       \end{array}
                     \right)
 & 0 \\
                 0 & 0 & B \\
               \end{array}
             \right)\in K:A\in T_{p-1},B\in T_{q-1}
\right\}\subset K
\end{align}
 is a maximal torus of $K$.

By \eqref{eq:defb} and \eqref{eq:tsopq}, we have
\begin{align}\label{eq:bpmpkmsopq}
  \fb&=\left\{\left(
               \begin{array}{ccc}
                 0 & 0 & 0 \\
                 0 & \left(
                       \begin{array}{cc}
                         0 & x \\
                         x & 0 \\
                       \end{array}
                     \right)
 & 0 \\
                 0 & 0 & 0 \\
               \end{array}
             \right)\in \fp:x\in \bR
\right\},\notag\\
\fp_\fm&=\left\{\left(
               \begin{array}{ccc}
                 0 & 0 & B \\
                 0 & \left(
                       \begin{array}{cc}
                         0 & 0 \\
                         0 & 0 \\
                       \end{array}
                     \right)
 & 0 \\
                 B^t & 0 & 0 \\
               \end{array}
             \right)\in \fp:B\in \mathrm{Mat}_{p-1,q-1}(\bR)\right\},\\
\fp^{\bot}(\fb)&=\left\{\left(
               \begin{array}{ccc}
                 0 & \begin{array}{cc}
                       0 & v_1
                     \end{array}
   & 0
 \\
                 \begin{array}{c}
       0 \\
       v_1^t
     \end{array} & \left(
                       \begin{array}{cc}
                         0 & 0 \\
                         0 & 0 \\
                       \end{array}
                     \right)
 & \begin{array}{c}
        v_2^t \\
      0
     \end{array} \\
                 0 & \begin{array}{cc}
                       v_2 & 0
                     \end{array} & 0 \\
               \end{array}
             \right)\in \fp:v_1\in \bR^{p-1},v_2\in \bR^{q-1}
\right\},\notag
\end{align}
where $v_1,v_2$ are considered as column vectors.  
Also,
\begin{align}\label{eq:bpmpkmsopq1}
\fk_\fm=\left\{\left(
               \begin{array}{ccc}
                 A & 0 & 0 \\
                 0 & \left(
                       \begin{array}{cc}
                         0 & 0 \\
                         0 & 0 \\
                       \end{array}
                     \right)
 & 0 \\
                 0 & 0 & D \\
               \end{array}
             \right)\in\fk:A\in \mathfrak{so}(p-1),D\in \mathfrak{so}(q-1)
\right\}.
\end{align}
By \eqref{eq:bpmpkmsopq} and \eqref{eq:bpmpkmsopq1}, we get
\begin{align}\label{eq:MKM}
\begin{aligned}
  M&=\left\{\left(
               \begin{array}{ccc}
                 A & 0 & B \\
                 0 & \left(
                       \begin{array}{cc}
                         1 & 0 \\
                         0 & 1 \\
                       \end{array}
                     \right)
 & 0 \\
                 C & 0 & D \\
               \end{array}
             \right)\in G: \left(
                             \begin{array}{cc}
                               A & B \\
                               C & D \\
                             \end{array}
                           \right)\in
\SO^0(p-1,q-1)\right\},\\
K_M&=\left\{\left(
               \begin{array}{ccc}
                 A & 0 & 0 \\
                 0 & \left(
                       \begin{array}{cc}
                         1 & 0 \\
                         0 & 1 \\
                       \end{array}
                     \right)
 & 0 \\
                 0 & 0 & D \\
               \end{array}
             \right)\in K: A\in \SO(p-1), D\in \SO(q-1)\right\}.
\end{aligned}
\end{align}

By \eqref{eq:bpmpkmsopq}, we can orient $\fb$ by $x>0$.
Then,
\begin{align}\label{eq:nsopq}
  \fn=\left\{\left(
               \begin{array}{ccc}
                 0 & \begin{array}{cc}
                       -v_1 & v_1
                     \end{array}
   & 0
 \\
                 \begin{array}{c}
       v_1^t \\
       v_1^t
     \end{array} & \left(
                       \begin{array}{cc}
                         0 & 0 \\
                         0 & 0 \\
                       \end{array}
                     \right)
 & \begin{array}{c}
       v_2^t \\
       v_2^t
     \end{array} \\
                 0 & \begin{array}{cc}
                       v_2 & -v_2
                     \end{array} & 0 \\
               \end{array}
             \right)\in \fg:v_1\in \bR^{p-1},v_2\in \bR^{q-1}
\right\}.
\end{align}
By \eqref{eq:bpmpkmsopq} and \eqref{eq:nsopq}, since for $x\in \bR$, $v_{1}\in \bR^{p-1}, v_{2}\in \bR^{q-1}$,  
\begin{align}\label{eq:hasp}
\[\left(
               \begin{array}{ccc}
                 0 & 0 & 0 \\
                 0 & \left(
                       \begin{array}{cc}
                         0 & x \\
                         x & 0 \\
                       \end{array}
                     \right)
 & 0 \\
                 0 & 0 & 0 \\
               \end{array}
             \right),\left(
               \begin{array}{ccc}
                 0 & \begin{array}{cc}
                       -v_1 & v_1
                     \end{array}
   & 0
 \\
                 \begin{array}{c}
       v_1^t \\
       v_1^t
     \end{array} & \left(
                       \begin{array}{cc}
                         0 & 0 \\
                         0 & 0 \\
                       \end{array}
                     \right)
 & \begin{array}{c}
       v_2^t \\
       v_2^t
     \end{array} \\
                 0 & \begin{array}{cc}
                       v_2 & -v_2
                     \end{array} & 0 \\
               \end{array}
             \right)\]
			 =x
			 \left(
               \begin{array}{ccc}
                 0 & \begin{array}{cc}
                       -v_1 & v_1
                     \end{array}
   & 0
 \\
                 \begin{array}{c}
       v_1^t \\
       v_1^t
     \end{array} & \left(
                       \begin{array}{cc}
                         0 & 0 \\
                         0 & 0 \\
                       \end{array}
                     \right)
 & \begin{array}{c}
       v_2^t \\
       v_2^t
     \end{array} \\
                 0 & \begin{array}{cc}
                       v_2 & -v_2
                     \end{array} & 0 \\
               \end{array}
             \right),
\end{align}
we find that $\fb$ acts on $\fn$ as a scalar.

\begin{prop}\label{prop:SOeff}
We have an isomorphism of Lie groups
\begin{align}\label{eq:isoSOpq}
\mathrm{SO}^0(p,q)\simeq \mathrm{Isom}^0\big(\mathrm{SO}^0(p,q)/\mathrm{SO}(p)\times \mathrm{SO}(q)\big),
\end{align}
where $pq> 1$ is odd.
\end{prop}
\begin{proof}As in the proof of Proposition \ref{prop:SLeff}, it 
	enough to show that $K$ acts effectively on $\fp$.  The 
	representation of $K\simeq \SO(p)\times \SO(q)$ on $\fp$ is 
	equivalent to $\sigma_p\boxtimes \sigma_q$. Assume that $(k_1,k_2)\in \SO(p)\times \SO(q)$ acts on  $\bR^p\boxtimes \bR^q$ as the identity. If $\lambda$ is any eigenvalue of $k_1$ and if $\mu$ is any eigenvalue of $k_2$, then
\begin{align}\label{eq:lm=1}
  \lambda\mu=1.
\end{align}
By \eqref{eq:lm=1}, both $k_1$ and $k_2$ are scalars. Using the fact 
that $\det(k_1)=\det(k_2)=1$ and that $p,q$ are odd, we deduce 
$k_1=1$ and $k_2=1$. The proof of our proposition  is completed.
\end{proof}

\subsection{The isometry group of $X$}\label{sec:i}
We return to the general case, where $G$ is only supposed to be such 
that  $\delta(G)=1$ and have compact center.

\begin{prop}
 The symmetric space $G/K$ is of the noncompact type.
\end{prop}

\begin{proof}\label{re:ss} Let $Z^0_G$ be the connected component of the identity in $Z_G$, and let $G_{ss}\subset G$ be the connected  subgroup of $G$ associated with the Lie algebra $[\fg,\fg]$. By \cite[Proposition 5.5]{Knappsemi}, $G_{ss}$ is closed in $G$, such that
\begin{align}
  G=Z^0_GG_{ss}.
\end{align}
Moreover, $G_{ss}$ is semisimple  with finite center, with
  maximal compact subgroup $K_{ss}=G_{ss}\cap K$. Also, the imbedding $G_{ss}\to G$ induces a diffeomorphism
\begin{align}
  G_{ss}/K_{ss}\simeq G/K.
\end{align}
Therefore,  $X$ is a symmetric space of the noncompact type.
\end{proof}
Put\index{G@$G_*$}
\begin{align}
  G_*=\mathrm{Isom}^0(X),
\end{align}
and let $K_*\subset G_*$ \index{K@$K_*$} be the stablizer of  $p1\in X$ fixed. Then $G_*$ is a semisimple Lie group with trivial center, and with maximal compact subgroup $K_*$. We denote by $\fg_*$ and $\fk_*$ the Lie algebras of $G_*$ and $K_*$.\index{G@$\fg_*$}\index{P@$\fp_*$}\index{K@$\fk_*$} Let
\begin{align}\label{eq:giso}
  \fg_*=\fp_*\oplus \fk_*
\end{align}
be the corresponding  Cartan decomposition. Clearly,
\begin{align}\label{eq:GK=GK}
  G_*/K_*\simeq X.
\end{align}

The morphism $i_G:G\to G_*$ defined in \eqref{eq:defig} induces a  morphism $i_\fg:\fg\to \fg_*$ of Lie algebras. By \eqref{eq:cartan2} and \eqref{eq:GK=GK},  $i_\fg$ induces an  isomorphism of vector spaces
\begin{align}\label{eq:pfp}
\fp\simeq \fp_*.
\end{align}
By the property  of $\fk_*$ and by \eqref{eq:pfp}, we have
\begin{align}\label{eq:kpp}
  \fk_*=[\fp_*,\fp_*]=i_\fg[\fp,\fp]\subset i_{\fg}\fk.
\end{align}
Thus $i_G,i_\fg$ are surjective.

\begin{prop}\label{cor:Gspli}
We have\index{G@$G_1,G_2$}
\begin{align}
  G_*=G_1\times G_2
\end{align}
where $G_1=\mathrm{SL}_3(\bR)$ or $G_1=\SO^0(p,q)$ with $pq>1$ odd, and where $G_2$ is a semisimple Lie group  with trivial center with $\delta(G_2)=0$.
\end{prop}
\begin{proof}By \cite[Theorem IV.6.2]{Kobayashi_Nomizu_I}, let $X=\prod_{i=1}^{l_1}X_l$ be the de Rham decomposition of $(X,g^{TX})$. Then every $X_i$ is an irreducible symmetric space of the noncompact type. By \cite[Theorem VI.3.5]{Kobayashi_Nomizu_I}, we have
\begin{align}\label{eq:Isopro}
  G_*=\prod_{i=1}^{l_1}\mathrm{Isom}^0(X_i),
\end{align}
By Theorem \ref{thm:cla},  \eqref{eq:isoSL3}, \eqref{eq:isoSOpq} and \eqref{eq:Isopro}, Proposition  \ref{cor:Gspli} follows.
\end{proof}

\subsection{Proof of Proposition \ref{prop:nnam1}}\label{sec:proofnn1}
  By \eqref{eq:g=g1g2} and by the definition of $\fb$ and $\fn$, we have
  \begin{align}\label{eq:bning1}
    \fb,\fn \subset \fg_1.
  \end{align}
  Proposition \ref{prop:nnam1} follows from \eqref{eq:hasl3} and 
  \eqref{eq:hasp}.
\qed
\subsection{Proof of Theorem \ref{thm:keyrep}}\label{sec:liftn}
%
\  

\underline{The case $G=\mathrm{SL}_3(\bR)$.}
By  \eqref{eq:MKMsl3} and \eqref{eq:nsl3},
the representation of $K_M\simeq \SO(2)$ on $\fn$ is just $\sigma_2$. Note that $K=\SO(3)$. We have the identity in $RO(K_M)$:
\begin{align}\label{eq:3to2}
 \iota^*\( \sigma_3-\mathbf{1}\)=\sigma_2,
\end{align}
which says $\fn$ lifts to $K$. 

\underline{The case  $G=\SO^0(p,q)$ with $pq>1$ odd.}
By \eqref{eq:MKM} and \eqref{eq:nsopq},  the representation of 
$K_M\simeq \SO(p-1)\times\SO(q-1)$ on $\fn$ is just  
$\sigma_{p-1}\boxtimes \mathbf{1}\oplus\mathbf{1}\boxtimes 
\sigma_{q-1}$.  Note that $K=\SO(p)\times \SO(q)$.  We have the identity in $RO(K_M)$:
\begin{align}\label{eq:ptoq}
\iota^* \big(( \sigma_{p}-\mathbf{1})\boxtimes \mathbf{1}\oplus 
\mathbf{1}\boxtimes   ( \sigma_{q}-\mathbf{1}) \big) =\sigma_{p-1}\boxtimes \mathbf{1}\oplus\mathbf{1}\boxtimes \sigma_{q-1},
\end{align}
which says $\fn$ lifts to $K$. 

\underline{The case for $G_*$. } This is a consequence of Proposition 
\ref{cor:Gspli}, \eqref{eq:bning1}-\eqref{eq:ptoq}.

\underline{The general case.}
Recall that $i_G:G\to G_*$ is a surjective morphism of Lie groups. 
Therefore, the restriction $i_K:K\to K_*$ of $i_G$ to $K$ is 
surjective. By \eqref{eq:pfp}, we have the identity in $RO(K)$:
\begin{align}\label{eq:ppull}
  \fp=i_K^*(\fp_*).
\end{align}

Set\index{T@$\ft_*$}
\begin{align}
  \ft_*=i_\fg(\ft)\subset \fk_*.
\end{align}
Since $i_K$  is surjective, by \cite[Theorem IV.2.9]{BrockerDieck}, $\ft_*$ is a Cartan subalgebra of $\fk_*$.
%
%
%

Let $\fb_*\subset \fp_*$ be the analogue of $\fb$ defined by $\ft_*$. \index{B@$\fb_*$}Thus,
\begin{align}
&\dim \fb_*=1,&  \fb_*=i_\fg(\fb).
\end{align}
We denote by $K_{*,M}, \fp_{*}^{\bot}(\fb_*), \fn_*$ the analogues of $K_M$, $\fp^\bot(\fb)$, $\fn$.
By \eqref{eq:pfp},  $i_\fg$ induces an isomorphism of vector spaces
\begin{align}\label{eq:ifgpp}
\fp^{\bot}({\fb}) \simeq \fp_{*}^{\bot}({\fb_*}).
\end{align}
Let $i_{K_M}:K_M\to K_{*,M}$ be the restriction of $i_G$ to $K_M$. We have the identity in $RO(K_M)$:
\begin{align}\label{eq:pbpull}
 \fp^{\bot}(\fb)=i^*_{K_M}\(\fp_{*}^{\bot}(\fb_*)\).
\end{align}

Let $\iota': K_{*,M}\to K_*$ be the embedding. Then the following diagram
\begin{align}\label{eq:diag}
\begin{aligned}
\xymatrix{
K_M \ar[d]^{i_{K_M}} \ar[r]^\iota &K\ar[d]^{i_K}\\
K_{*,M} \ar[r]^{\iota'} &K_*}
\end{aligned}
\end{align}
commutes. It was proved in the previous step that there is $E\in 
RO(K_*)$ such that the following identity in $RO(K_{*,M})$ holds:
\begin{align}\label{eq:lift}
  \iota^{\prime*}(E)=\fn_*.
\end{align}

By \eqref{eq:n=pbK}, \eqref{eq:pbpull}-\eqref{eq:lift}, we have the 
identity in $RO(K_M)$, 
\begin{align}
 \fn= \fp^\bot(\fb)=i_{K_M}^*\(\fp_{*}^{\bot}(\fb_*)\)= i_{K_M}^*(\fn_*)=i_{K_M}^*\iota^{\prime*}(E)=\iota^*i_K^*(E),
\end{align}
which completes  the proof of our theorem.\qed

\subsection{Proof of Proposition \ref{eq:propYb}}\label{sec:propYb}

If $n\in \mathbf{N}$, consider  the following closed subgroups:
\begin{align}\label{eq:exUU}
\begin{aligned}
& A\in  \mathrm{U}(2)\to \left(
                          \begin{array}{cc}
                            A & 0 \\
                            0 & \det^{-1}(A) \\
                          \end{array}
                        \right)\in \mathrm{SU}(3), \\
                        &(A,B)\in \SO(n)\times \SO(2)\to \left(
                          \begin{array}{cc}
                            A & 0 \\
                            0 & B \\
                          \end{array}
                        \right)\in \SO(n+2).
\end{aligned}
\end{align}

We state Proposition \ref{eq:propYb} in a more exact way.
\begin{prop}\label{eq:propYb1} We have the isomorphism of symmetric spaces
  \begin{align}\label{eq:Yb}
    Y_\fb\simeq \mathrm{SU}(3)/\mathrm{U}(2) \hbox{ or } \SO(p+q)/\SO(p+q-2)\times \SO(2),
  \end{align}
  with $pq>1$ odd.
\end{prop}
\begin{proof} Let $U_*$ and $U_{*}(\fb_*)$ be the analogues of  $U$ 
	and $U(\fb)$ when $G$ and $\fb$ are replaced by $G_*$ and 
	$\fb_*$. It is enough to show that
\begin{align}\label{eq:ybUUb}
  Y_\fb\simeq U_{*}/U_*(\fb_*).
\end{align}
Indeed, by the explicit constructions given  in Subsections \ref{subsec:SL3} and \ref{subsec:SOPQ}, by Proposition \ref{cor:Gspli}, and by \eqref{eq:exUU}, \eqref{eq:ybUUb}, we get \eqref{eq:Yb}.

Let  $Z_U\subset U$ be the center of $U$, and let $Z^0_U$ be the connected component of the identity in $Z_U$.
Let $U_{ss}\subset U$ be the connected subgroup of $U$ associated to the Lie algebra  $[\fu,\fu]\subset \fu$.
By \cite[Proposition 4.32]{Knappsemi}, $U_{ss}$ is compact, and  $U=U_{ss}Z^0_U$.

Let $U_{ss}(\fb)$ be the analogue of $U(\fb)$ when $U$ is replaced by $U_{ss}$. Then $U(\fb)=U_{ss}(\fb)Z^0_U$, and the imbedding $U_{ss}\to U$ induces an isomorphism of homogenous spaces
\begin{align}\label{eq:Yb1}
 U_{ss}/U_{ss}(\fb)\simeq U/U(\fb).
\end{align}

Let $\widetilde{U}_{ss}$ be the universal cover of $U_{ss}$. Since $U_{ss}$ is semisimple, $\widetilde{U}_{ss}$ is compact.
We define $\widetilde{U}_{ss}(\fb)$ similarly.
The canonical projection $\widetilde{U}_{ss}\to U_{ss}$ induces an isomorphism of homogenous spaces
\begin{align}\label{eq:Yb2}
\widetilde{U}_{ss}/  \widetilde{U}_{ss}(\fb)\simeq  U_{ss}/U_{ss}(\fb).
\end{align}
Similarly, since $U_*$ is semisimple,
if $\widetilde{U}_*$ is a universal cover of $U_*$, and if we define $\widetilde{U}_{*}(\fb)$ in the same way, we have
\begin{align}\label{eq:Yb200}
  \widetilde{U}_{*}/  \widetilde{U}_{*}(\fb)\simeq  U_{*}/U_{*}(\fb).
\end{align}

 The surjective morphism of Lie algebras $i_\fg:\fg\to \fg_*$ induces a surjective morphism
 of the compact forms $i_\fu:\fu\to \fu_*$. Since $\fu_*$ is semisimple, the restriction of $i_{\fu}$ to $[\fu,\fu]$ is still surjective. It
 lifts to a surjective morphism of simply connected Lie groups
\begin{align}\label{eq:tusstous}
  \widetilde{U}_{ss}\to \widetilde{U}_*.
\end{align}
Since any connected, simply connected, semisimple compact Lie group can be 
written as a product of connected, simply connected, simple compact Lie groups,  we can assume that there is a connected and simply connected semisimple compact Lie group $U'$ such that
 $\widetilde{U}_{ss}=\widetilde{U}_*\times U'$, and that the morphism \eqref{eq:tusstous} is  the   canonical projection.
Therefore,
\begin{align}\label{eq:Yb3}
  \widetilde{U}_{ss}/ \widetilde{U}_{ss}(\fb)\simeq \widetilde{U}_*/\widetilde{U}_*(\fb_*).
\end{align}

From \eqref{eq:Yb0}, \eqref{eq:Yb1}-\eqref{eq:Yb200} and 
\eqref{eq:Yb3}, we get \eqref{eq:ybUUb}.  The proof of our proposition  is completed.
\end{proof}
\begin{re}
  The Hermitian symmetric spaces on the right-hand side of 
  \eqref{eq:Yb} are irreducible and respectively of the type AIII and 
  the type BDI in the classification of Cartan \cite[p. 518 Table V]{HelgasonSymmetric}.
\end{re}

\subsection{Proof of Proposition \ref{prop:A0S1}}
We use the notation in Subsection \ref{sec:propYb}.
By definition, $A_0\subset U_{ss}$.   Let $\widetilde{A}_0\subset \widetilde{U}_{ss}$ and $A_{*0}\subset U_*$ be the analogues of $A_0$ when $U$ is replaced by $\widetilde{U}_{ss}$ and $U_*$.
As in the proof of Proposition \ref{eq:propYb}, we can show that $\widetilde{A}_0$ is a finite cover of $A_0$ and $A_{*0}$.

On the other hand, by the explicit constructions given  in Subsections \ref{subsec:SL3}, \ref{subsec:SOPQ}, and by Proposition \ref{cor:Gspli},  $A_{*0}$ is a circle $\bbS^1$. Therefore, both $\widetilde{A}_0, A_0$ are circles.

\subsection{Proof of Proposition \ref{prop:G1=0}}\label{sec:G1=0}
%
We use the notation in Subsection \ref{sec:t=0}. Let $\gamma\in G$ be 
such that $\dim \fb(\gamma)\g2$. As in \eqref{eq:trkpo}, we assume 
that $\gamma=e^ak^{-1}$ is such that
\begin{align}\label{eq:trkpo1}
&\ft(\gamma)\subset \ft, &k\in T.
\end{align}
It is enough to show \eqref{eq:G1=0} for $Y\in\ft(\gamma)$.

For $Y\in \ft(\gamma)$, since $k^{-1}\exp(-iY)\in T$ and $T\subset K^M$,  we have
\begin{align}\label{eq:TREhat}
  \Trs^{\Lambda^{\cdot}(\fp_{\fm})}\[k^{-1}\exp(-iY)\]
=\mathrm{det}\big(1-\Ad(k)\exp(i\ad(Y))\big)|_{\fp_\fm}.
\end{align}
It is enough to show
\begin{align}\label{eq:dbrp1}
  \dim \fb(\gamma)\cap \fp_\fm \g 1.
\end{align}

Note that $a\neq0$, otherwise $\dim \fb(\gamma)=1$. Let
\begin{align}
a=a^1+a^2+a^3\in \fb\oplus\fp_\fm\oplus\fp^\bot(\fb).
\end{align}
Since the decomposition $\fb\oplus\fp_\fm\oplus\fp^\bot(\fb)$ is 
preserved by $\ad(\ft)$ and $\Ad(T)$, it is also preserved by 
$\ad(\ft(\gamma))$ and $\Ad(k)$. Since $a \in \fb(\gamma)$, the $a_i$, $1 \l i \l 3$, all lie in $\fb(\gamma)$.  If $a^2\neq0$, we get \eqref{eq:dbrp1}.
If $a^2=0$ and $a^3=0$, we have $a\in \fb$. Since $a\neq 0$, then 
$\fb(\gamma) = \fb$, which is impossible since $\dim b (\gamma) \g 2$.

It remains to consider the case
\begin{align}\label{eq:asa2a3}
&a^2=0,&a^3\neq0.
\end{align}
We will follow the steps in the proof of Theorem \ref{thm:keyrep}.


\ul{The case $G=\mathrm{SL}_3(\bR)$.}  By \eqref{eq:TSL3} and \eqref{eq:pmkmsl3}, the representation of $T\simeq \SO(2)$ on $\fp^\bot(\fb)$ is equivalent to $\sigma_2$.
A nontrivial element of $T$ never fixes $a^3$. Therefore,
\begin{align}\label{eq:k=1sl3}
k=1.
\end{align}

Since $a\notin \fb$, $a$ does not commute  with all the elements of $\ft$. From \eqref{eq:trkpo1}, we get
\begin{align}
  \dim \ft(\gamma)< \dim \ft=1.
\end{align}
Therefore,
\begin{align}\label{eq:tr=0sl3}
   \ft(\gamma)=0.
\end{align}
By \eqref{eq:defbr}, \eqref{eq:k=1sl3} and \eqref{eq:tr=0sl3}, we see that $\fb(\gamma)=\fp$. Therefore,
\begin{align}\label{eq:babababab}
  \dim \fb(\gamma)\cap \fp_\fm =\dim \fp_\fm.
\end{align}
By \eqref{eq:pmkmsl3} and \eqref{eq:babababab}, we get \eqref{eq:dbrp1}.

\ul{The case $G=\SO^0(p,q)$ with $pq>1$ odd.} By \eqref{eq:bpmpkmsopq} and \eqref{eq:MKM},  the representations of $K_M\simeq\SO(p- 1)\times \SO(q-1)$ on $\fp_{\fm}$ and $\fp^\bot(\fb)$ are equivalent to
$\sigma_{p-1}\boxtimes \sigma_{q-1}$ and $\sigma_{p-1}\boxtimes \mathbf{1}\oplus  \mathbf{1}\boxtimes \sigma_{q-1}$. We identify $a^3\in \fp^\bot(\fb)$ with
\begin{align}\label{eq:a3a1a2}
v^1+v^2\in \bR^{p-1}\oplus \bR^{q-1}.
\end{align}
Then $v^1$ and $v^2$ are fixed by $\Ad(k)$ and commute with  $\ft(\gamma)$.

If $v^1\neq0$ and $v^2\neq0$, by \eqref{eq:defbr}, the nonzero element $v^1\boxtimes v^2\in \bR^{p-1}\boxtimes \bR^{q-1}\simeq\fp_\fm$ is in $\fb(\gamma)$. It implies \eqref{eq:dbrp1}.

If $v^2=0$, we will show that $\gamma$ can be conjugated into  $H$ by an element of $K$, which implies $\dim \fb(\gamma)=1$ and contradicts  $\dim \fb(\gamma)\g2$. (The proof for the case  $v^1=0$ is similar.)
Without loss of generality,  assume that there exist $s\in \mathbf{N}$ with $1\l s\l (p-1)/2$ and $\lambda_s, \cdots, \lambda_{(p-1)/2}\in \bC$ nonzero complex numbers such that,
\begin{align}\label{eq:v1}
v^1=\(0,\cdots,0,\lambda_{s},\cdots,\lambda_{(p-1)/2}\)\in \bC^{(p-1)/2}\simeq \bR^{p-1}.
\end{align}
Then there exists $x\in \bR$ such that
\begin{align}\label{eq:apva}
  a=\left(
               \begin{array}{ccc}
                 0 & \begin{array}{cc}
                       0 & v^1
                     \end{array}
   & 0
 \\
                 \begin{array}{c}
       0 \\
       v^{1t}
     \end{array} & \left(
                       \begin{array}{cc}
                         0 & x \\
                         x & 0 \\
                       \end{array}
                     \right)
 & \begin{array}{c}
        0 \\
      0
     \end{array} \\
                 0 & \begin{array}{cc}
                       0 & 0
                     \end{array} & 0 \\
               \end{array}
             \right)\in \fp.
\end{align}

By \eqref{eq:tsopq} and \eqref{eq:trkpo1}, there exist $A\in T_{p-1}$ and $D\in T_{q-1}$ such that
\begin{align}\label{eq:kAD}
  k=\left(
      \begin{array}{ccc}
        A & 0 & 0 \\
        0 & \left(
              \begin{array}{cc}
                1 & 0 \\
                0 & 1 \\
              \end{array}
            \right)
 & 0 \\
        0 & 0 & D \\
      \end{array}
    \right)\in T.
\end{align}
If we identify $T_{p-1}\simeq U(1)^{(p-1)/2}$, there are $\theta_1, \cdots, \theta_{(p-1)/2}\in \bR$ such that
\begin{align}\label{eq:Ator}
  A=\(e^{2i\pi \theta_1},\cdots, e^{2i\pi \theta_{(p-1)/2}}\).
\end{align}
Since $k$ fixes $a$, by \eqref{eq:v1}-\eqref{eq:Ator},
for $i=s,\cdots, (p-1)/2$, we have
\begin{align}\label{eq:e=0}
e^{2i\pi\theta_i}=1.
\end{align}

If $W\in \mathfrak{so}(p-2s+2)$, set
\begin{align}\label{eq:lW}
  l(W)=\left(
         \begin{array}{ccc}
           \overmat{p \hbox{col.}}{0 & 0} & 0 \\
           0 & W & 0 \\
           0&0 & 0 \\
         \end{array}
       \right)\in \fk.
\end{align}
By \eqref{eq:kAD}-\eqref{eq:lW}, we have
\begin{align}\label{eq:adklw}
  kl(W)=l(W)k.
\end{align}

Put $w=(\lambda_s,\cdots, \lambda_{(p-1)/2},x)\in \bC^{(p-2s+1)/2}\oplus \bR\simeq \bR^{p-2s+2}$.
There exists $W\in \mathfrak{so}(p-2s+2)$ such that
\begin{align}\label{eq:expWw}
  \exp(W)w= (0,\cdots, 0,|w|),
\end{align}
where $|w|$ is the Euclidean norm of $w$.

Put
\begin{align}
  k'=\exp(l(W))\in K.
\end{align}
By \eqref{eq:bpmpkmsopq}, \eqref{eq:adklw} and \eqref{eq:expWw}, we have
\begin{align}\label{eq:rtoH}
 & \Ad(k')a\in \fb,&k'kk^{\prime-1}=k.
\end{align}
Thus, $\gamma$ is conjugated by $k'$ into $H$.

\ul{The general case.} By \eqref{eq:g=g1g2}, $\fg=\fg_1\oplus \fg_2$ with $\fg_1=\mathfrak{sl}_3(\bR)$ or $\fg_1=\mathfrak{so}(p,q)$ with $pq>1$ odd. By \eqref{eq:bning1} and \eqref{eq:asa2a3}, we have $a\in \fg_1$. The arguments in \eqref{eq:k=1sl3}-\eqref{eq:a3a1a2} extend directly. We only need to take care
 of the case $\fg_1=\mathfrak{so}(p,q)$ and $a^2=0, v^1\neq0$ and $v^2=0$. In this case, the arguments in  \eqref{eq:apva}-\eqref{eq:expWw} extend  to the group of isometries $G_*$. In particular,  there is $W_*\in \fk_*$ such that
\begin{align}\label{eq:Ws}
 &\Ad\(\exp(W_*)\)i_\fg(a)\in \fb_*,  &\Ad(i_G(k))W_*=W_*.
\end{align}

By \eqref{eq:pfp}, $\ker(i_\fg)\subset \fk$. Let $\ker(i_\fg)^\bot$ be the orthogonal space of  $\ker(i_\fg)$ in $\fk$. Then,
\begin{align}
 & \fk=\ker(i_\fg)\oplus \ker(i_\fg)^{\bot},  &\ker(i_\fg)^{\bot}\simeq \fk_*.
\end{align}
Take $W=(0,W_*)\in \fk$. Put
\begin{align}\label{eq:kk2}
  k'=\exp(W)\in K.
\end{align}
By \eqref{eq:pfp}, \eqref{eq:Ws} and \eqref{eq:kk2}, we get 
\eqref{eq:rtoH}. Thus, $\gamma$ is conjugate  by $k'$ into $H$.
 The proof of  \eqref{eq:G1=0} is completed. \qed

%
%


\section{Selberg and Ruelle zeta functions}\label{Sec:qusi}
In this section, we assume that $\delta(G)=1$ and that $G$ has 
compact  center. 
The purpose of this section is to  establish  the first part of our 
main Theorem \ref{thm:01z}. 

%


In Subsection \ref{sec:evapri}, we introduce a class of 
representations $\eta$ of $M$, so that $\eta|_{K_{M}}$ lifts as an 
element of $RO(K)$. In particular, $\eta_{j}$ is in this class. Take 
$\widehat{\eta}=\Lambda^{\cdot}(\fp_{\fm}^{*})\otimes \eta\in 
RO(K)$. Using the explicit formulas for orbital integrals of Theorem \ref{thm:Bis}, we give 
an explicit geometric formula for  $\Trs^{[\gamma]}\[\exp(-tC^{\fg,X,\widehat{\eta}}/2)\]$, whose 
proof is given in Subsection \ref{sec:d11}. 

In Subsection \ref{sec:selzeta}, we introduce a Selberg zeta 
function $Z_{\eta,\rho}$  associated with $\eta$ and $\rho$. 
Using the result in Subsection 
\ref{sec:evapri}, we express $Z_{\eta,\rho}$ in terms of the regularized 
determinant of the resolvent of  $C^{\fg,Z, \widehat{\eta},\rho}$, 
and we prove that  $Z_{\eta,\rho}$ is meromorphic and  satisfies a functional 
equation.

Finally, in Subsection \ref{sec:6}, we show that the dynamical zeta 
function $R_\rho(\sigma)$ 
is equal to an alternating product of  $Z_{\eta_{j},\rho}$, from 
which we deduce the first part of Theorem \ref{thm:01z}.


\subsection{An explicit formula for $\Trs^{[\gamma]}\[\exp\(-tC^{\fg,X,\widehat{\eta}}/2\)\]$}\label{sec:evapri}
Now we introduce a class of representations of $M$.
\begin{as}\label{as:1}
Let $\eta$ \index{E@$\eta$}be a real finite dimensional representation of $M$ such that
\begin{enumerate}
   \item The restriction $\eta|_{K_M}$ on $K_M$ can be lifted into $RO(K)$;
   \item The action of  the Lie algebra  $\fu_\fm\subset \fm\otimes_\bR\bC$ on
$E_\eta\otimes_\bR\bC$, induced by complexification, can be lifted to an action of Lie group $U_M$;
   \item The Casimir element $C^{\fu_\fm}$ of $\fu_\fm$ acts on 
   $E_\eta\otimes_\bR\bC$ as the scalar $C^{\fu_\fm,\eta}\in \bR$.
 \end{enumerate}
\end{as}
By Corollary \ref{cor:key},  let
$\widehat{\eta}=\widehat{\eta}^+-\widehat{\eta}^-\in RO(K)$\index{E@$\widehat{\eta},\widehat{\eta}^+,\widehat{\eta}^-$} be the virtual real finite dimensional representation of $K$ on
$E_{\widehat{\eta}}=E^+_{\widehat{\eta}}-E^-_{\widehat{\eta}}$ \index{E@$E_{\widehat{\eta}},E^+_{\widehat{\eta}},E^-_{\widehat{\eta}}$}
such that the following identity in $RO(K_M)$ holds:
\begin{align}\label{eq:hatvar}
E_{\widehat{\eta}}|_{K_M}=  \sum_{i=0}^{\dim \fp_\fm}(-1)^i\Lambda^i(\fp^*_\fm)\otimes E_{\eta}|_{K_M}.
\end{align}

By Corollary  \ref{cor:key} and by Proposition \ref{prop:bggbuu}, 
 $\eta_j$ satisfies Assumption \ref{as:1}, so that the following identity in $RO(K)$ holds
 \begin{align}\label{eq:haba}
 \sum_{i=1}^{\dim \fp}(-1)^{i-1}i\Lambda^i(\fp^*)=\sum^{ 2l}_{j=0} (-1)^jE_{\widehat{\eta}_i}.
 \end{align}

As in Subsection \ref{sec:sym},
let $\cE_{\widehat{\eta}}=G\times_K E_{\widehat{\eta}}$  \index{E@$\cE_{\widehat{\eta}}$}be the induced
virtual  vector bundle on $X$.
Let $C^{\fg,X,\widehat{\eta}}$ be the corresponding Casimir element of $G$ acting on $C^\infty(X,\cE_{\widehat{\eta}})$. 
We will state an explicit  formula for  
$\Trs^{[\gamma]}\[\exp(-tC^{\fg,X,\widehat{\eta}}/2)\]$.

By \eqref{eq:Ub=S1UM}, the complex representation of $U_M$ on $E_\eta\otimes_\bR \bC$ extends to a complex  representation of
$U(\fb)$ such that $A_0$ acts trivially. Set\index{F@$F_{\fb,\eta}$}
\begin{align}
  F_{\fb,\eta}=U\times_{U(\fb)} \(E_\eta\otimes_\bR \bC\).
\end{align}
Then $F_{\fb,\eta}$ is a complex vector bundle on $Y_\fb$. It is equipped with a connection $\nabla^{F_{\fb,\eta}}$, induced by $\omega^{\fu(\fb)}$, with curvature $R^{F_{\fb,\eta}}$.\index{R@$R^{F_{\fb,\eta}}$}

\begin{re}
  When $\eta=\eta_j$, the above action of $U(\fb)$ on $\Lambda^j(\fn^*_\bC)$ is different from the adjoint  action
of $U(\fb)$ on  $\Lambda^j(\fn^*_\bC)$ induced by \eqref{eq:ubotnn}.
\end{re}

Recall that $T$ is the maximal torus of both $K$ and $U_M$. Put\index{C@$c_G$}
\begin{align}\label{eq:cG}
  c_G=(-1)^{\frac{m-1}{2}}\frac{|W(T,U_M)|}{|W(T,K)|}\frac{\vol(K/K_M)}{\vol(U_M/K_M)}.
\end{align}
Recall that $X_M=M/K_M$. By Bott's formula \cite[p. 175]{Bott1965},
\begin{align}\label{eq:Bott}
  \chi(K/K_M)=\frac{|W(T,K)|}{|W(T,K_M)|},
\end{align}
and by  \eqref{eq:eulWW}, \eqref{eq:cG}, we have a more geometric expression\begin{align}
  c_G=(-1)^l\frac{\[e\(TX_M,\nabla^{TX_M}\)\]^{\max}}{\[e\(T(K/K_M),\nabla^{T(K/K_M)}\)\]^{\max}}.
\end{align}
Note that $\dim \fu^{\bot}(\fb)=2\dim \fn=4l$.
If $\beta\in \Lambda^{\cdot}(\fu^{\bot,*}(\fb))$, let $[\beta]^{\max}\in \bR$ \index{B@$[\beta]^{\max}$}be such that
\begin{align}\label{eq:bmax}
  \beta-[\beta]^{\max}\frac{\omega^{Y_\fb,2l}}{(2l)!}
\end{align}
is of degree smaller than $4l$.

\begin{thm}\label{thm:d11}
 For $t>0$, we have
\begin{multline}\label{eq:d11}
  \Trs^{[1]}\[\exp\(-tC^{\fg,X,\widehat{\eta}}/2\)\]=\frac{c_G}{\sqrt{2\pi t}}\exp\(\frac{t}{16}\Tr^{{\fu^{\bot}(\fb)}}\[C^{\fu(\fb),\fu^\bot(\fb)}\]-\frac{t}{2}C^{\fu_\fm,\eta}\)\\
 \[\exp\(-\frac{\omega^{{Y_\fb},2}}{8\pi^2|a_0|^2 t}\)\widehat{A}\(TY_{\fb},\nabla^{TY_{\fb}}\)\mathrm{ch}\(F_{\fb,\eta},\nabla^{F_{\fb,\eta}}\)\]^{\max}.
\end{multline}
If $\gamma=e^ak^{-1}\in H$ with $a\in \fb, a\neq0, k\in T$, for  $t>0$, we have
\begin{multline}\label{eq:Trg}
  \Trs^{[\gamma]}\[\exp\(-t C^{\fg,X,\widehat{\eta}}/2\)\]
=\frac{1}{{\sqrt{2 \pi t}}}\[e\(TX_{M}(k),\nabla^{TX_M(k)}\)\]^{\max}\\
\exp\(-\frac{|a|^2}{2t}+\frac{t}{16}\Tr^{\fu^{\bot}(\fb)}\[C^{\fu(\fb),\fu^\bot(\fb)}\]-\frac{t}{2}C^{\fu_\fm,\eta}\)
 \frac{\Tr^{E_\eta}\[\eta(k^{-1})\]}{\left|\det\big(1-\Ad(\gamma)\big)|_{\fz_0^\bot}\right|^{1/2}}.
\end{multline}
If $\dim \fb(\gamma)\g2$, for $t>0$, we have
\begin{align}\label{eq:G1=00}
  \Trs^{[\gamma]}\[\exp\(-t C^{\fg,X,{\widehat{\eta}}}/2\)\]=0.
\end{align}
\end{thm}
\begin{proof}
 The proof of \eqref{eq:d11} and \eqref{eq:Trg} will be given in 
 Subsection \ref{sec:d11}.
 Equation  \eqref{eq:G1=00} is a consequence of \eqref{eq:trrJr}, 
 (\ref{eq:G1=0}),  and \eqref{eq:hatvar}.
\end{proof}

\subsection{The proof of Equations (\ref{eq:d11}) and (\ref{eq:Trg})}\label{sec:d11}
Let us recall some facts on Lie algebra. Let $\Delta(\ft,\fk)\subset\ft^*$ \index{D@$\Delta(\ft,\fk),\Delta^+(\ft,\fk)$} be the real root system \cite[Definition V.1.3]{BrockerDieck}. We fix a set of positive roots $\Delta^+(\ft,\fk)\subset \Delta(\ft,\fk)$. 
Set\index{R@$\rho^{\fk}$}
\begin{align}\label{eq:rhok88}
  \rho^\fk=\frac{1}{2}\sum_{\alpha\in \Delta^+(\ft,\fk)}\alpha.
\end{align}
By Kostant's strange formula \cite{Kostant76} or \cite[Proposition 7.5.1]{B09}, we have
\begin{align}\label{eq:kostant}
  4\pi^2|\rho^\fk|^2=-\frac{1}{24}\Tr^{\fk}\[C^{\fk,\fk}\].
\end{align}

Let $\pi_\fk:\ft\to \bC$ be the polynomial function such that for $Y\in \ft$,\index{P@$\pi_{\fk}(Y)$}
\begin{align}\label{eq:pisigma1}
  \pi_{\fk}(Y)=\prod_{\alpha\in  \Delta^+(\ft,\fk)}2i\pi \<\alpha,Y\>.
\end{align}
Let $\sigma_\fk:\ft\to \bC$ be the denominator in the Weyl character formula. \index{S@$\sigma_{\fk}(Y)$} For $Y\in \ft$, we have
\begin{align}\label{eq:pisigma2}
\sigma_{\fk}(Y)=\prod_{\alpha\in  \Delta^+(\ft,\fk)}\(e^{i\pi\<\alpha,Y\>}-e^{-i\pi\<\alpha,Y\>}\).
\end{align}
The Weyl group $W(T,K)$ acts isometrically on $\ft$. For $w\in W(T,K)$, set $\e_w=\mathrm{det}(w)|_{\ft}$.
The Weyl denominator formula asserts for $Y\in \ft$, we have
\begin{align}\label{eq:Weyldemno}
  \sigma_{\fk}(Y)=\sum_{w\in W(T,K)}\e_w \exp\(2i \pi \<\rho^\fk,wY\>\).
\end{align}

Let $\widehat{K}$ be the set of equivalence classes of complex irreducible representations of $K$.
There is a bijection between $\widehat{K}$ and the set of dominant and analytic integral elements in $\ft^*$ \cite[Section VI (1.7)]{BrockerDieck}. If $\lambda\in \ft^*$ is dominant and analytic integral, the  character $\chi_\lambda$\index{C@$\chi_\lambda$}
of the corresponding  complex irreducible representation
is given by the Weyl character  formula: for $Y\in \ft$,
\begin{align}\label{eq:Weylch}
\sigma_\fk\big(Y\big)  \chi_\lambda\big(\exp(Y)\big)=\sum_{w\in W(T,K)}\e_w\exp\(2i\pi\<\rho^\fk+\lambda,wY\>\).
\end{align}

Let us recall the Weyl integral formula for Lie algebras. Let $dv_{K/T}$ be the
Riemannian volume on  $K/T$ induced by $-B$,  and let  $dY$  be the Lebesgue measure on $\fk$ or $\ft$ induced by $-B$. By \cite[Lemma 11.4]{Knappsemi}, if $f\in C_c(\fk)$, we have
\begin{align}\label{eq:weylintlie}
  \int_{Y \in \fk} f(Y)dY=\frac{1}{|W(T,K)|}\int_{Y\in \ft}|\pi_\fk(Y)|^2\(\int_{k\in K/T} f\big(\Ad(k)Y\big)dv_{K/T}\)dY.
\end{align}
Clearly, the formula \eqref{eq:weylintlie} extends  to $L^1(\fk)$.

%

\begin{proof}[Proof of \eqref{eq:d11}] By \eqref{eq:mpnk}, \eqref{eq:trrJr} and \eqref{eq:weylintlie}, we have
\begin{multline}\label{eq:tr1si}
  \Trs^{[1]}\[\exp\(-t C^{\fg,X,\widehat{\eta}}/2\)\]=\frac{1}{(2\pi t)^{(m+n)/2}}\exp\(\frac{t}{16}\Tr^{\fp}[C^{\fk,\fp}]+\frac{t}{48}\Tr^{\fk}\[C^{\fk,\fk}\]\)\\
\frac{\vol(K/T)}{|W(T,K)|}\int_{Y\in\ft}|\pi_\fk(Y)|^2J_1(Y)\Trs^{E_{\widehat{\eta}}}\[\exp(-i\widehat{\eta}(Y))\]\exp(-|Y|^2/2t)dY.
\end{multline}
As $\delta(M)=0$, $\ft$ is also a Cartan subalgebra of $\fu_\fm$. 
We will use \eqref{eq:weylintlie} again to write the integral on the second line of \eqref{eq:tr1si} as an integral over $\fu_\fm$.


By \eqref{eq:n=pbK}, we have the isomorphism of representations of $K_M$,
\begin{align}\label{eq:pkpk}
\fp^\bot(\fb)\simeq \fk^\bot(\fb).
\end{align}
By \eqref{eq:J1} and \eqref{eq:pkpk}, for $Y\in \ft$, we have
\begin{align}\label{eq:J1t}
  J_1(Y)=\frac{\widehat{A}(i\ad(Y)|_{\fp_\fm})}{\widehat{A}(i\ad(Y)|_{\fk_\fm})}.
\end{align}
By \eqref{eq:hatvar}, for $Y\in \ft$, we have
\begin{align}\label{eq:TREhat99}
  \Trs^{E_{\widehat{\eta}}}\[\exp(-i\widehat{\eta}(Y))\]
=\mathrm{det}\big(1-\exp(i\ad(Y))\big)|_{\fp_\fm}\Tr^{E_\eta}\[\exp(-i\eta(Y))\].
\end{align}
By  \eqref{eq:pisigma1}, \eqref{eq:J1t} and \eqref{eq:TREhat99}, for $Y\in \ft$, we have
\begin{multline}\label{eq:ttoum}
  \frac{|\pi_\fk(Y)|^2}{|\pi_{\fu_\fm}(Y)|^2}J_1(Y)
  \Trs^{E_{\widehat{\eta}}}\[\exp\big(-i\widehat{\eta}(Y)\big)\]\\
  =(-1)^{\frac{\dim \fp_\fm}{2}}\mathrm{det}\big(\ad(Y)\big)|_{{\fk^\bot(\fb)}}\widehat{A}^{-1}\big(i\ad(Y)|_{\fu_\fm}\big)\Tr^{E_\eta}\[\exp\big(-i\eta(Y)\big)\].
\end{multline}
Using \eqref{eq:n=pbK}, for $Y\in \ft$, we have
\begin{align}\label{eq:detn}
  \mathrm{det}\big(\ad(Y)\big)|_{{\fk^\bot(\fb)}}=\mathrm{det}\big(\ad(Y)\big)|_{\fn_\bC}.
\end{align}
By the second condition of  Assumption \ref{as:1}
and by \eqref{eq:detn}, the function on the right-hand side of \eqref{eq:ttoum} extends naturally to an $\Ad(U_M)$-invariant function defined on $\fu_\fm$.
By \eqref{eq:cG}, \eqref{eq:weylintlie}, \eqref{eq:tr1si},
\eqref{eq:ttoum} and \eqref{eq:detn}, we have
\begin{multline}\label{eq:id3}
  \Trs^{[1]}\[\exp\(-t C^{\fg,X,\widehat{\eta}}/2\)\]
=\frac{(-1)^lc_G}{(2\pi t)^{(m+n)/2}}\exp\(\frac{t}{16}\Tr^{\fp}[C^{\fk,\fp}]+\frac{t}{48}\Tr^{\fk}\[C^{\fk,\fk}\]\)\\
 \int_{Y\in \fu_\fm}\mathrm{\det}\big(\ad(Y)\big)|_{\fn_\bC}\widehat{A}^{-1}\big(i\ad(Y)|_{\fu_\fm}\big)\Tr^{E_\eta}\[\exp(-i\eta(Y))\]\exp\(-|Y|^2/2t\)dY.
\end{multline}

It remains to evaluate the integral on the second line of \eqref{eq:id3}. We use the method in \cite[Section 7.5]{B09}.  For $Y\in \fu_\fm$, we have
\begin{align}
  |Y|^2=-B(Y,Y).
\end{align}
By \eqref{eq:Hern}, \eqref{eq:Rub} and \eqref{eq:Rububb}, for $Y\in \fu_\fm$, we have
\begin{align}\label{eq:w}
  B\(Y,\Omega^{\fu_\fm}\)=-\sum_{1\l i,j\l 2l}B\(\ad(Y)f_i,\ol{f}_j\)f^i\wedge \ol{f}^j=\sum_{1\l i,j\l 2l}\<\ad(Y)f_i,f_j\>_{\fn_\bC}f^i\wedge \ol{f}^j.
\end{align}
By \eqref{eq:kahher}, \eqref{eq:bmax} and \eqref{eq:w}, for $Y\in \fu_\fm$, we have
\begin{align}\label{eq:id4}
\frac{\mathrm{\det}(\ad(Y))|_{\fn_\bC}}{(2\pi t)^{2l}}=(-1)^l\[\exp\(\frac{1}{ t}B\(Y,\frac{\Omega^{\fu_\fm}}{2\pi}\)\)\]^{\max}.
\end{align}

As $\dim \fu_\fm=\dim \fm=m+n-2l-1$, from \eqref{eq:id3} and \eqref{eq:id4}, we get
\begin{multline}\label{eq:d113}
  \Trs^{[1]}\[\exp\(-t C^{\fg,X,\widehat{\eta}}/2\)\]=\frac{c_G}{\sqrt{2\pi t}}\exp\(\frac{t}{16}\Tr^{\fp}[C^{\fk,\fp}]+\frac{t}{48}\Tr^{\fk}\[C^{\fk,\fk}\]\)\\
\exp\(\frac{t}{2}\Delta^{\fu_\fm}\) \left\{{\widehat{A}^{-1}\big(i\ad(Y)|_{\fu_\fm}\big)}\Tr^{E_\eta}\big[\exp(-i\eta(Y))\big]\exp\(\frac{1}{ t}B\(Y,\frac{\Omega^{\fu_\fm}}{2\pi}\)\)\right\}^{\max}\bigg|_{Y=0}.
\end{multline}
Using
\begin{align}
  B\(Y,\frac{\Omega^{\fu_\fm}}{2\pi}\)+\frac{1}{2}B(Y,Y)=\frac{1}{2}B\(Y+\frac{\Omega^{\fu_\fm}}{2\pi},Y+\frac{\Omega^{\fu_\fm}}{2\pi}\)- \frac{1}{2} B\(\frac{\Omega^{\fu_\fm}}{2\pi},\frac{\Omega^{\fu_\fm}}{2\pi}\),
\end{align}
by \eqref{eq:BOO=0} and \eqref{eq:d113}, we have
\begin{multline}\label{eq:g730}
  \Trs^{[1]}\[\exp\(-t C^{\fg,X,\widehat{\eta}}/2\)\]=\frac{c_G}{\sqrt{2\pi t}}\exp\(\frac{t}{16}\Tr^{\fp}[C^{\fk,\fp}]+\frac{t}{48}\Tr^{\fk}\[C^{\fk,\fk}\]\)\\
  \left\{\exp\(-\frac{\omega^{Y_\fb,2}}{8\pi^2|a_0|^2t}\)\exp\(\frac{t}{2}\Delta^{\fu_\fm}\)\(\widehat{A}^{-1}\big(i\ad(Y)|_{\fu_\fm}\big)\Tr^{E_\eta}\[\exp(-i\eta(Y))\]\)\right\}^{\max}\bigg|_{Y=-\frac{\Omega^{\fu_\fm}}{2\pi}}.
\end{multline}

We claim that the $\Ad\big(U_M\big)$-invariant function
\begin{align}\label{eq:evar1}
  Y\in \fu_\fm\to  {\widehat{A}^{-1}\big(i\ad(Y)|_{\fu_\fm}\big)}{\Tr^{E_\eta}\[\exp(-i\eta(Y))\]}
\end{align}
 is an eigenfunction of $\Delta^{\fu_\fm}$ with eigenvalue
\begin{align}\label{eq:eig1}
-C^{\fu_\fm,\eta}-\frac{1}{24}\Tr^{\fu_{\fm}}\[C^{\fu_{\fm},\fu_{\fm}}\].
\end{align}
 Indeed,  if $f$ is an $\Ad(U_M)$-invariant function on $\fu_\fm$, 
 when restricted to $\ft$, it is well known, for example \cite[eq. (7.5.22)]{B09}, that
 \begin{align}
 \Delta^{\fu_\fm} f=\frac{1}{\pi_{\fu_\fm}}\Delta^{\ft}\pi_{\fu_\fm}f.
 \end{align}
 Therefore, it is enough to show that the function
\begin{align}\label{eq:funt1}
Y\in \ft\to  \pi_{\fu_\fm}(Y)  \widehat{A}^{-1}\big(i\ad(Y)\big)|_{\fu_\fm}{\Tr^{E_\eta}\[\exp(-i\eta(Y))\]}
\end{align}
is an eigenfunction of $\Delta^{\ft}$ with eigenvalue \eqref{eq:eig1}. For $Y\in \ft$, we have
\begin{align}\label{eq:EignA}
\widehat{A}^{-1}\big(i\ad(Y)\big)|_{\fu_\fm}=\frac{\sigma_{\fu_\fm}(iY)}{\pi_{\fu_\fm}(iY)}
\end{align}
By \eqref{eq:EignA}, for $Y\in \ft$, we have
\begin{align}\label{eq:Ain1}
  \pi_{\fu_\fm}(Y)\widehat{A}^{-1}\(i\ad(Y)\)|_{\fu_\fm}=i^{|\Delta^+(\ft,\fu_\fm)|}\sigma_{\fu_\fm}(-iY).
\end{align}
If $E_\eta\otimes_\bR\bC$ is an irreducible representation of $U_M$ with the highest weight $\lambda\in \ft^*$, by the Weyl character formula \eqref{eq:Weylch}, we have
\begin{align}\label{eq:Ain2}
 \sigma_{\fu_\fm}(-iY) \Trs^{E_\eta}\[\exp(-i\eta(Y))\]=\sum_{w\in W(T,U_M)}\e_w\exp(2\pi\<\rho^{\fu_\fm}+\lambda,wY\>).
\end{align}
By \eqref{eq:Ain1} and \eqref{eq:Ain2}, the function \eqref{eq:funt1} is an eigenfunction of $\Delta^\ft$ with eigenvalue \begin{align}\label{eq:eigen}
  4\pi^2|\rho^{\fu_\fm}+\lambda|^2.
\end{align}
  By  Assumption \ref{as:1}, the Casimir of $\fu_\fm$ acts   as the  scalar $C^{\fu_\fm,\eta}$. Therefore,
\begin{align}\label{eq:wgaga}
-C^{\fu_\fm,\eta}=  4\pi^2\(|\rho^{\fu_\fm}+\lambda|^2-|\rho^{\fu_\fm}|^2\).
\end{align}
By \eqref{eq:kostant} and \eqref{eq:wgaga}, the eigenvalue \eqref{eq:eigen} is equal to \eqref{eq:eig1}.
If $E_\eta\otimes_\bR\bC$ is not  irreducible, it is enough to decompose $E_\eta\otimes_\bR\bC$ as a sum of  irreducible representations of $U_M$.

Since the function \eqref{eq:funt1} and its derivations of any orders 
satisfy estimations similar to \eqref{eq:Jrexp},  by \eqref{eq:RN} 
and \eqref{eq:g730}, we get
\begin{multline}\label{eq:d116}
  \Trs^{[1]}\[\exp\(-t 
  C^{\fg,X,\widehat{\eta}}/2\)\]\\
  =\frac{c_G}{\sqrt{2\pi 
  t}}\exp\(\frac{t}{16}\Tr^{\fp}[C^{\fk,\fp}]+\frac{t}{48}\Tr^{\fk}\[C^{\fk,\fk}\]-\frac{t}{48}\Tr^{\fu_{\fm}}\[C^{\fu_{\fm},\fu_{\fm}}\]-\frac{t}{2}C^{\fu_\fm,\eta}\)\\
  \left\{\exp\(-\frac{\omega^{Y_\fb,2}}{8\pi^2|a_0|^2t}\){\widehat{A}^{-1}\(\frac{R^{N_\fb}}{2i\pi}\)}\Tr^{E_\eta}\[\exp\(-\frac{R^{F_{\fb,\eta}}}{2i\pi }\)\]\right\}^{\max}.
\end{multline}

Since $\widehat{A}$ is an even function, by \eqref{eq:Ahatg}, we have
\begin{align}\label{eq:d117}
  \widehat{A}\(\frac{R^{N_\fb}}{2i\pi}\)=\widehat{A}\(N_\fb,\nabla^{N_\fb}\).
\end{align}

We claim that
\begin{align}\label{eq:masi0}
\Tr^{\fp}\[C^{\fk,\fp}\]+\frac{1}{3}\Tr^{\fk}\[C^{\fk,\fk}\]-\frac{1}{3}\Tr^{\fu_{\fm}}\[C^{\fu_{\fm},\fu_{\fm}}\]=\Tr^{\fu^{\bot}(\fb)}\[C^{\fu(\fb),\fu^{\bot}(\fb)}\].
\end{align}
Indeed, by \cite[Proposition 2.6.1]{B09}, we have
\begin{align}
\begin{aligned}\label{eq:masi1}
\Tr^{\fp}\[C^{\fk,\fp}\]+\frac{1}{3}\Tr^{\fk}\[C^{\fk,\fk}\]=\frac{1}{3}\Tr^{\fu}\[C^{\fu,\fu}\],\\
\Tr^{\fu^{\bot}(\fb)}\[C^{\fu(\fb),\fu^{\bot}(\fb)}\]+\frac{1}{3}\Tr^{\fu(\fb)}\[C^{\fu(\fb),\fu(\fb)}\]=\frac{1}{3}\Tr^{\fu}\[C^{\fu,\fu}\].
\end{aligned}
\end{align}
By \eqref{eq:99}, it is trivial that
\begin{align}\label{eq:masi2}
\Tr^{\fu(\fb)}\[C^{\fu(\fb),\fu(\fb)}\]=\Tr^{\fu_\fm}\[C^{\fu_\fm,\fu_\fm}\].
\end{align}
From  \eqref{eq:masi1} and \eqref{eq:masi2}, we get \eqref{eq:masi0}

By \eqref{eq:chg}, \eqref{eq:AA1}, \eqref{eq:d116}-\eqref{eq:masi0}, we get \eqref{eq:d11}.
\end{proof}

Let $U_M(k)$ \index{U@$U_M(k),U_M^0(k)$}be the  centralizer of $k$ in $U_M$, and let $\fu_{\fm}(k)$ \index{U@$\fu_{\fm}(k)$}be its Lie algebra. Then
\begin{align}
  \fu_{\fm}(k)=\sqrt{-1}\fp_\fm(k)\oplus \fk_\fm(k).
\end{align}
Let $U^0_{M}(k)$ be the connected component of the identity in $U_M(k)$. Clearly, $U^0_{M}(k)$ is the compact form of $M^0(k)$.

%
\begin{proof}[Proof of \eqref{eq:Trg}]
Since $\gamma\in H$, $\ft\subset \fk(\gamma)$ is a Cartan subalgebra of $\fk(\gamma)$.
By \eqref{eq:trrJr}, \eqref{eq:Kr=kMk} and \eqref{eq:weylintlie}, we have
\begin{multline}\label{eq:trrsi}
  \Trs^{[\gamma]}\[\exp\(-t C^{\fg,X,\widehat{\eta}}/2\)\]=\frac{1 }{(2\pi t)^{\dim \fz(\gamma)/2}}\exp\(-\frac{|a|^2}{2t}+\frac{t}{16}\Tr^{\fp}[C^{\fk,\fp}]+\frac{t}{48}\Tr^{\fk}\[C^{\fk,\fk}\]\)\\
\frac{\vol(K^0_M(k)/T)}{|W(T,K^0_M(k))|}
\int_{Y\in \ft}|\pi_{\fk_{\fm}(k)}(Y)|^2J_\gamma(Y)
  \Trs^{E_{\widehat{\eta}}}
\[\widehat{\eta}\(k^{-1}\)\exp(-i\widehat{\eta}(Y))\]\exp\(-|{Y}|^2/2t\)dY.
\end{multline}
Since $\ft$ is also a Cartan subalgebra of $\fu_{\fm}(k)$, as in the proof of  (\ref{eq:d11}),
we will write the integral on the second line of \eqref{eq:trrsi} as an integral over $\fu_{\fm}(k)$.

As $k\in T$ and $T\subset K_M$, by \eqref{eq:hatvar}, for $Y\in \ft$, we have
\begin{multline}\label{eq:TREhat100}
  \Trs^{E_{\widehat{\eta}}}\[\widehat{\eta}\(k^{-1}\)\exp(-i\widehat{\eta}(Y))\]\\
=\mathrm{det}\big(1-\Ad(k)\exp(i\ad(Y))\big)|_{\fp_\fm}\Tr^{E_\eta}\[\eta\(k^{-1}\)\exp(-i\eta(Y))\].
\end{multline}
%
%
By \eqref{eq:J}, \eqref{eq:pisigma1} and \eqref{eq:TREhat100}, 
for $Y\in \ft$, we have
\begin{multline}\label{eq:trrre}
  \frac{|\pi_{\fk_{\fm}(k)}(Y)|^2}{|\pi_{\fu_\fm(k)}(Y)|^2}J_\gamma(Y)
  \Trs^{E_{\widehat{\eta}}}
\[\widehat{\eta}\(k^{-1}\)\exp(-i\widehat{\eta}(Y))\]
=\frac{(-1)^{\frac{\dim \fp_\fm(k)}{2}}}{\left|\det\big(1-\Ad(\gamma)\big)|_{\fz_0^\bot}\right|^{1/2}}\\
\widehat{A}^{-1}\(i\ad(Y)|_{\fu_\fm(k)}\)
\[\frac{\det\big(1-\exp(-i\ad(Y))\Ad(k^{-1})\big)|_{\fz_0^\bot(\gamma)}}{\det\big(1-\Ad(k^{-1})\big)|_{\fz_0^\bot(\gamma)}}\]^{1/2}\!\!\!\!\!
\Tr^{E_\eta}\[\eta\(k^{-1}\)\exp(-i\eta(Y))\]\!.
\end{multline}
Let $\fu^\bot_\fm(k)$ be the orthogonal space to $\fu_\fm(k)$ in $ \fu_\fm$. Then
\begin{align}\label{eq:fubotk}
  \fu^\bot_\fm(k)=\sqrt{-1}\fp^\bot_0(\gamma)\oplus \fk_0^\bot(\gamma).
\end{align}
By \eqref{eq:fubotk}, for $Y\in \ft$, we have
\begin{align}\label{eq:1/2}
\frac{\det\big(1-\exp(-i\ad(Y))\Ad(k^{-1})\big)|_{\fz_0^\bot(\gamma)}}{\det\big(1-\Ad(k^{-1})\big)|_{\fz_0^\bot(\gamma)}}=\frac{\det\big(1-\exp(-i\ad(Y))\Ad(k^{-1})\big)|_{\fu_\fm^\bot(k)}}{\det\big(1-\Ad(k^{-1})\big)|_{\fu_\fm^\bot(k)}}.
\end{align}
By   Assumption \ref{as:1} and \eqref{eq:1/2}, the right-hand side of \eqref{eq:trrre} extends naturally to an $\Ad(U^0_M(k))$-invariant function defined on $\fu_\fm(k)$.
By \eqref{eq:eulWW}, \eqref{eq:weylintlie}, \eqref{eq:trrsi} and \eqref{eq:trrre}, we have
\begin{multline}\label{eq:trrkh}
 \Trs^{[\gamma]}\[\exp\(-t C^{\fg,X,\widehat{\eta}}/2\)\]\\
 =
\frac{1}{\sqrt{2 \pi t}}
\frac{\[e\(TX_{M}(k),\nabla^{TX_M(k)}\)\]^{\max}}{\left|\det\big(1-\Ad(\gamma)\big)|_{\fz_0^\bot}\right|^{1/2}}
\exp\(-\frac{|a|^2}{2t}+\frac{t}{16}\Tr^{\fp}[C^{\fk,\fp}]+\frac{t}{48}\Tr^{\fk}\[C^{\fk,\fk}\]\)\\
\exp\(\frac{t}{2}\Delta^{\fu_{\fm}(k)}\)\Bigg\{\widehat{A}^{-1}\big(i\ad(Y)|_{\fu_{\fm}(k)}\big)
  \[\frac{\det\big(1-\exp(-i\ad(Y))\Ad(k^{-1})\big)|_{\fu_\fm^\bot(k)}}{\det\big(1-\Ad(k^{-1})\big)|_{\fu_\fm^\bot(k)}}\]^{1/2}\\ \Tr^{E_\eta}\[\eta\(k^{-1}\)\exp(-i\eta(Y))\]\Bigg\}\Bigg|_{Y=0}.
\end{multline}

As before, we claim that the function
\begin{multline}\label{eq:kuohao}
Y\in \fu_\fm(k)\to   \widehat{A}^{-1}\big(i\ad(Y)|_{\fu_{\fm(k)}}\big)
  \[\frac{\det\big(1-\exp(-i\ad(Y))\Ad(k^{-1})\big)|_{\fu_\fm^\bot(k)}}{\det\big(1-\Ad(k^{-1})\big)|_{\fu_\fm^\bot(k)}}\]^{1/2} \\ \Tr^{E_\eta}\[\eta\(k^{-1}\)\exp(-i\eta(Y))\]
\end{multline}
 is an eigenfunction of $\Delta^{\fu_{\fm}(k)}$ with  eigenvalue \eqref{eq:eig1}. Indeed, it is enough to remark that,
as in \eqref{eq:Ain2},
up to a sign, if $k=\exp(\theta_1)$ for some $\theta_1\in \ft$, we have
\begin{multline}
  \pi_{\fu_{\fm(k)}}(Y) \widehat{A}^{-1}\big(i\ad(Y)|_{\fu_{\fm(k)}}\big)
  \[\det\big(1-\exp(-i\ad(Y))\Ad(k^{-1})\big)|_{\fu_\fm^\bot(k)}\]^{1/2}\\
=\pm i^{|\Delta^+(\ft,\fu_\fm)|} \sigma_\fu(-iY-\theta_1).
\end{multline}
Also, if $E_\eta\otimes_\bR\bC$ is an irreducible representation of $U_M$ with the highest weight $\lambda\in \ft^*$,
\begin{align}\label{eq:Ain3}
 \sigma_{\fu_\fm}(-iY-\theta_1) \Trs^{E_\eta}\[\eta\(k^{-1}\)\exp(-i\eta(Y))\]=\sum_{w\in W(T,U_M)}\e_w\exp(2\pi\<\rho_{\fu_\fm}+\lambda,w(Y-i\theta_1)\>).
\end{align}
Proceeding as in the proof of  \eqref{eq:d11}, we get \eqref{eq:Trg}. \end{proof}

\subsection{Selberg   zeta functions}\label{sec:selzeta}
Recall that $\rho:\Gamma\to \mathrm{U}(r)$ is a unitary 
representation of $\Gamma$ and that $(F,\nabla^F,g^F)$ is the 
unitarily  flat vector bundle on $Z$ associated with $\rho$.

\begin{defin}
For $\sigma\in \bC$,	we define a formal sum 
\begin{align}\label{eq:Xisel}
\Xi_{\eta,\rho}(\sigma)=-\sum_{[\gamma]\in [\Gamma]- 
\{1\}}\Tr[\rho(\gamma)]\frac{\chi_{\rm 
orb}\big(\mathbb{S}^{1}\backslash B_{[\gamma]}\big)}{m_{[\gamma]}}\frac{\Tr^{E_\eta}\[\eta(k^{-1})\]}{\left|\det\big(1-\Ad(\gamma)\big)|_{\fz_0^\bot}\right|^{1/2}}e^{-\sigma|a|}
\end{align}	
and a formal Selberg zeta function 
\begin{align}
  Z_{\eta,\rho}(\sigma)=\exp\big(\Xi_{\eta,\rho}(\sigma)\big).
\end{align}
The formal Selberg zeta function is said to be well defined if 
the same conditions as in Definition \ref{def:Rr} hold. 
\end{defin}

\begin{re}
  When $G=\SO^0(p,1)$ with $p\g3$ odd, up to a shift on $\sigma$,
  $Z_{\eta,\rho}$ coincides with Selberg zeta function  in \cite[Section 3]{FriedRealtorsion}.
\end{re}

Recall that the Casimir operator $ C^{\fg,Z,\widehat{\eta},\rho}$ acting on
 $C^\infty(Z,\cF_{\widehat{\eta}}\otimes_\bC F)$ is a  formally  self-adjoint second order elliptic operator, which is  bounded from below. For $\lambda\in \bC$, set \index{M@$m_{\eta,\rho}(\lambda)$}
\begin{align}
  m_{\eta,\rho}(\lambda)=\dim_\bC \ker \(C^{\fg,Z,\widehat{\eta}^+\!\!\!,\rho}-\lambda\)-\dim_\bC \ker \(C^{\fg,Z,\widehat{\eta}^-\!\!\!,\rho}-\lambda\).
\end{align}
Write\index{R@$r_{\eta,\rho}$}
\begin{align}\label{eq:rrs}
  r_{\eta,\rho}=m_{\eta,\rho}(0).
\end{align}
As in Subsection \ref{sec:regdet}, for $\sigma\in \bR$ and $\sigma\gg1$, set\index{D@$ \det_{\rm gr}$}
\begin{align}\label{eq:detgr}
\mathrm{det}_{\rm gr}\(C^{\fg,Z,\widehat{\eta},\rho}+\sigma\)=\frac{\mathrm{det}(C^{\fg,Z,\widehat{\eta}^+\!\!\!,\rho}+\sigma)}{\mathrm{det}(C^{\fg,Z,\widehat{\eta}^-\!\!\!,\rho}+\sigma)}.
\end{align}
Then, $\mathrm{det}_{\rm gr}\(C^{\fg,Z,\widehat{\eta},\rho}+\sigma\)$ extends  meromorphically to $\sigma\in \bC$. Its zeros and poles  belong to the set $\{-\lambda : \lambda\in \Sp(C^{\fg,Z,\widehat{\eta},\rho})\}$. If   $\lambda\in \Sp(C^{\fg,Z,\widehat{\eta},\rho})$, the order of the zero at $\sigma=-\lambda$ is $m_{\eta,\rho}(\lambda)$.

Set\index{S@$\sigma_\eta$}
\begin{align}\label{eq:AshiftC}
  \sigma_{\eta}=\frac{1}{8}\Tr^{\fu^{\bot}(\fb)}\[C^{\fu(\fb),\fu^\bot(\fb)}\]-C^{\fu_\fm,\eta}.
\end{align}
Set\index{P@$P_{\eta}(\sigma)$}
\begin{align}\label{eq:polyP}
  P_{\eta}(\sigma)= c_G
\sum^l_{j=0}(-1)^j \frac{\Gamma(-j-\frac{1}{2})}{j!(4\pi)^{2j+\frac{1}{2}}|a_0|^{2j}} \[\omega^{{Y_\fb},2j}\widehat{A}\(TY_{\fb},\nabla^{TY_{\fb}}\)\mathrm{ch}\(\cF_{\fb,\eta},\nabla^{\cF_{\fb,\eta}}\)\]^{\max}\sigma^{2j+1}.
\end{align}
Then $P_{\eta}(\sigma)$ is an odd polynomial function of $\sigma$. As the notation indicates, $\sigma_\eta$ and  $P_\eta(\sigma)$ do not depend on $\Gamma$ or $\rho$.

\begin{thm}\label{thm:detfor}
There is $\sigma_0>0$ \index{S@$\sigma_0$}such that
\begin{align}\label{eq:see}
   \sum_{[\gamma]\in  [\Gamma]- \{1\}}\frac{\left| \chi_{\rm 
   orb}\big(\bbS^1\backslash B_{[\gamma]}\big)\right|}{m_{[\gamma]}}\frac{1}{\left|\det\big(1-\Ad(\gamma)\big)|_{\fz_0^\bot}\right|^{1/2}}e^{-\sigma_0|a|}<\infty.
\end{align}
	The Selberg zeta function $Z_{\eta,\rho}(\sigma)$ has a meromorphic extension to $\sigma\in \bC$ such that  the following identity of meromorphic functions on $\bC$ holds:
\begin{align}\label{eq:detfor}
  Z_{\eta,\rho}(\sigma)=\mathrm{det}_{\rm gr}\(C^{\fg,Z,\widehat{\eta},\rho}+\sigma_{\eta}+\sigma^2\)\exp\big(r\vol(Z) P_\eta(\sigma)\big).
\end{align}
The zeros and poles of $Z_{\eta,\rho}(\sigma)$
belong to the set $\{\pm i\sqrt{\lambda+\sigma_{\eta}}:\lambda\in \Sp(C^{\fg,Z,\widehat{\eta},\rho})\}$. If $\lambda\in \Sp(C^{\fg,Z,\widehat{\eta},\rho})$ and $\lambda\neq -\sigma_{\eta}$, the order of zero  at $\sigma=\pm i\sqrt{\lambda+\sigma_{\eta}}$ is $m_{\eta,\rho}(\lambda)$. The order of zero at $\sigma=0$ is $2m_{\eta,\rho}(-\sigma_{\eta})$.
Also,
\begin{align}\label{eq:funceq}
  Z_{\eta,\rho}(\sigma)=Z_{\eta,\rho}(-\sigma)\exp\big(2r\vol(Z) P_\eta(\sigma)\big).
\end{align}
\end{thm}
\begin{proof}
Proceeding as in the proof of Theorem \ref{thm:dg=1Rd1}, by 
Proposition \ref{prop:v=e},  Corollary \ref{cor:V=0}, and Theorem \ref{thm:d11}, we get 
the first two statements of our theorem. 
%
By \eqref{eq:detfor}, the zeros and poles of
$Z_{\eta,\rho}(\sigma)$
 coincide with that of $\mathrm{det}_{\rm 
 gr}\(C^{\fg,Z,\widehat{\eta},\rho}+\sigma_{\eta}+\sigma^2\)$, from 
 which we deduce the third statement of our theorem. 
Equation \eqref{eq:funceq} is a consequence of \eqref{eq:detfor} and 
of the fact that $P_\eta(\sigma)$ is an odd polynomial. The proof of  
our theorem  is completed.
\end{proof}

\subsection{The Ruelle dynamical zeta function}\label{sec:6}
We turn our attention to the Ruelle dynamical zeta function $R_\rho(\sigma)$.

\begin{thm}\label{cor:R}
  The dynamical zeta function $R_\rho(\sigma)$ is  holomorphic for $\Re(\sigma)\gg1$, and extends meromorphically to $\sigma\in \bC$ such that
\begin{align}\label{eq:RbyZ}
  R_\rho(\sigma)=\prod_{j=0}^{2l}Z_{\eta_j,\rho}\big(\sigma+(j-l)|\alpha|\big)^{(-1)^{j-1}}.
\end{align}
\end{thm}
\begin{proof}
Clearly, there is $C>0$ such that for all $\gamma\in \Gamma$,
\begin{align}\label{eq:det1/2E}
   \left|\mathrm{det}\big(1-\Ad(\gamma)\big)|_{\fz_0^\bot}\right|^{1/2}\l C\exp(C|a|).
\end{align}
By \eqref{eq:see} and \eqref{eq:det1/2E}, for $\sigma\in \bC$ and $\Re(\sigma)>\sigma_0+C$, the sum
in \eqref{eq:Rrho} converges absolutely to a holomorphic function.  By \eqref{eq:V=e}, \eqref{eq:V=0}, \eqref{eq:Rrho}, \eqref{eq:det1/2fi} and \eqref{eq:Xisel}, for $\sigma\in \bC$ and $\Re(\sigma)>\sigma_0+C$, we have
\begin{align}
  \Xi_{\rho}(\sigma)=\sum_{j=0}^{2l}(-1)^{j-1}\Xi_{\eta_j,\rho}(\sigma+(j-l)|\alpha|).
\end{align}
By taking exponentials, we get \eqref{eq:RbyZ} for 
$\Re(\sigma)>\sigma_0+C$. Since the right-hand side of 
\eqref{eq:RbyZ} is meromorphic, $R_\rho(\sigma)$ has a meromorphic 
extension to $\bC$, such that \eqref{eq:RbyZ} holds. The proof of our 
theorem  is completed.
\end{proof}

Remark that for $0\l j\l 2l$, we have the isomorphism of $K_M$-representations of $ \eta_j\simeq \eta_{2l-j}$. By \eqref{eq:hatvar}, we have the isomorphism of $K$-representations,
\begin{align}\label{eq:njl}
  \widehat{\eta}_j\simeq \widehat{\eta}_{2l-j}
\end{align}
Note that by \eqref{eq:BggBuu} and \eqref{eq:AshiftC}, we have
\begin{align}\label{eq:teshu}
\sigma_{\eta_j}=-(j-l)^2|\alpha|^2.
\end{align}

By \eqref{eq:detfor}, \eqref{eq:njl} and \eqref{eq:teshu}, we have
\begin{multline}\label{eq:Z1}
   Z_{\eta_j,\rho}\(-\sqrt{\sigma^2+(l-j)^2|\alpha|^2}\)Z_{\eta_{2l-j},\rho}\(\sqrt{\sigma^2+(l-j)^2|\alpha|^2}\)\\
   =Z_{\eta_j,\rho}\(-\sqrt{\sigma^2+(l-j)^2|\alpha|^2}\)Z_{\eta_{j},\rho}\(\sqrt{\sigma^2+(l-j)^2|\alpha|^2}\)\\
   =\mathrm{det}_{\rm gr}\(C^{\fg,Z,\widehat{\eta}_j,\rho}+\sigma^2\)^2=\mathrm{det}_{\rm gr}\(C^{\fg,Z,\widehat{\eta}_j,\rho}+\sigma^2\)\mathrm{det}_{\rm gr}\(C^{\fg,Z,\widehat{\eta}_{2l-j},\rho}+\sigma^2\).
\end{multline}
 Recall that $T(\sigma)$ is defined in \eqref{eq:Tsigma}. 
\begin{thm}
  \label{prop:RT}
  The following identity of meromorphic functions on $\bC$ holds:
  \begin{multline}\label{eq:RT}
    R_\rho(\sigma)=T(\sigma^2) \exp\((-1)^{l-1}r\vol(Z)P_{\eta_l}(\sigma)\)\\
\prod_{j=0}^{l-1}  \Bigg(\frac{Z_{\eta_j,\rho}\big(\sigma+(j-l)|\alpha|\big)Z_{\eta_{2l-j},\rho}\big(\sigma+(l-j)|\alpha|\big)}{Z_{\eta_j,\rho}\(-\sqrt{\sigma^2+(l-j)^2|\alpha|^2}\)Z_{\eta_{2l-j},\rho}\(\sqrt{\sigma^2+(l-j)^2|\alpha|^2}\)}\Bigg)^{(-1)^{j-1}}.
  \end{multline}
\end{thm}
\begin{proof}
By  \eqref{eq:Tsigma}, \eqref{eq:D=C}, and \eqref{eq:detgr}, we have the identity of meromorphic functions,
\begin{align}\label{eq:detfortor}
  T(\sigma)=\prod_{j=0}^{2l}{\rm det}_{\rm gr}\(C^{\fg,Z,\widehat{\eta}_j,\rho}+\sigma\)^{(-1)^{j-1}}.
\end{align}

By \eqref{eq:detfor}, \eqref{eq:Z1}, and \eqref{eq:detfortor}, we have
\begin{multline}\label{eq:thtoZ}
  T(\sigma^2)= Z_{\eta_l,\rho}(\sigma)^{(-1)^{l-1}}\exp\((-1)^{l}r\vol(Z)P_{\eta_l}(\sigma)\)\\
\prod_{j=0}^{l-1}  \Bigg(Z_{\eta_j,\rho}\(-\sqrt{\sigma^2+(l-j)^2|\alpha|^2}\)Z_{\eta_{2l-j},\rho}\(\sqrt{\sigma^2+(l-j)^2|\alpha|^2}\)\Bigg)^{(-1)^{j-1}}.
\end{multline}
By  \eqref{eq:RbyZ} and \eqref{eq:thtoZ}, we get \eqref{eq:RT}. The 
proof of our theorem is completed.
\end{proof}

 For $0\l j\l 2l$, as in \eqref{eq:rrs}, we write $r_j=r_{\eta_j,\rho}$. \index{R@$r_j$} By  \eqref{eq:njl} and \eqref{eq:detfortor}, we have
\begin{align}\label{eq:chid}
  \chi'(Z,F)=2\sum_{j=0}^{l-1}(-1)^{j-1}r_j+(-1)^{l-1}r_l.
\end{align}
Set\index{C@$C_\rho$}\index{R@$r_\rho$}
\begin{align}\label{eq:Cr}
 & C_\rho=\prod_{j=0}^{l-1}\big(-4(l-j)^2|\alpha|^2\big)^{(-1)^{j-1}r_j},&r_\rho=2\sum_{j=0}^{l}(-1)^{j-1}r_j.
\end{align}



\begin{proof}[Proof of \eqref{eq:t01z}]
By Proposition \ref{prop:bggbuu} and Theorem \ref{thm:detfor}, for $0\l j\l l-1$,
the orders of the zero  at $\sigma=0$ of the functions $Z_{\eta_j,\rho}\big(\sigma+(j-l)|\alpha|\big)$ and $ Z_{\eta_{2l-j},\rho}\big(\sigma+(l-j)|\alpha|\big)$
are equal to $r_j$. Therefore,   for $0\l j\l l-1$, there are $A_j\neq0, B_j\neq0$ such that  as $\sigma\to 0$, 
\begin{align}\label{eq:P7}
\begin{aligned}
Z_{\eta_j,\rho}\big(\sigma+(j-l)|\alpha|\big)&=A_j\sigma^{r_j}+\cO\(\sigma^{r_j+1}\),\\
Z_{\eta_{2l-j},\rho}\big(\sigma+(l-j)|\alpha|\big)&=B_j\sigma^{r_j}+\cO\(\sigma^{r_j+1}\),
\end{aligned}
\end{align}
and
\begin{align}\label{eq:P6}
\begin{aligned}
Z_{\eta_j,\rho}\(-\sqrt{\sigma^2+(l-j)^2|\alpha|^2}\)&=A_j\(\frac{-\sigma^2}{2(l-j)|\alpha|}\)^{r_j}+\cO\(\sigma^{2r_j+2}\)\\
Z_{\eta_{2l-j},\rho}\(\sqrt{\sigma^2+(l-j)^2|\alpha|^2}\)&=B_j\(\frac{\sigma^2}{2(l-j)|\alpha|}\)^{r_j}+\cO\(\sigma^{2r_j+2}\).
\end{aligned}
\end{align}
By \eqref{eq:P7} and \eqref{eq:P6}, as $\sigma\to0$,
\begin{multline}\label{eq:725}
  \frac{Z_{\eta_j,\rho}\big(\sigma+(j-l)|\alpha|\big)Z_{\eta_{2l-j},\rho}\big(\sigma+(l-j)|\alpha|\big)}{Z_{\eta_j,\rho}\(-\sqrt{\sigma^2+(l-j)^2|\alpha|^2}\)Z_{\eta_{2l-j},\rho}\(\sqrt{\sigma^2+(l-j)^2|\alpha|^2}\)}\\
  \to \(-4(l-j)^2|\alpha|^2\)^{r_j}\sigma^{-2r_j}+\cO(\sigma^{-2r_j+1}).
\end{multline}
By  \eqref{eq:polyP}, \eqref{eq:RT}, \eqref{eq:chid}, \eqref{eq:Cr}, and \eqref{eq:725},
we get  \eqref{eq:t01z}.
\end{proof}
%
%

\begin{re}
  When $G=\SO^0(p,1)$ with $p\g3$ odd, we recover   \cite[Theorem 3]{FriedRealtorsion}.
\end{re}

\begin{re}
If we scale the form $B$ with the factor $a > 0$, $R_\rho (\sigma)$ is replaced by $R_\rho ( \sqrt{a}\sigma)$. By \eqref{eq:t01z}, as $\sigma\to 0$,
\begin{align}
R_\rho(\sqrt{a}\sigma)=a^{r_\rho/2}C_\rho T(F)^2\sigma^{r_\rho}+\cO(\sigma^{r_\rho+1}).
\end{align}
On the other hand, $C_\rho$ should become $a^{\sum_{j=0}^{l-1}(-1)^jr_j} C_\rho$ , and $T(F)$ should scale by $a^{\chi'(Z,F)/2}$. This is only possible if
\begin{align}\label{eq:3B}
r_\rho = 2\sum_{j=0}^{l-1} (-1)^j r_j + 2\chi'(Z, F ),
\end{align}
which just \eqref{eq:chid}.
\end{re}

\section{A cohomological formula for $r_j$}\label{sec:proofthm1}
The purpose of  this section is to establish  \eqref{eq:t02z} when $G$ is such that $\delta(G)=1$ and has compact center.  We rely on some deep results from the representation theory of reductive Lie groups.

This section is organized as follows. In Subsection \ref{sec:rep}, we recall the constructions of the infinitesimal and global characters of Harish-Chandra modules. We also recall some properties of $(\fg,K)$-cohomology and $\fn$-homology of Harish-Chandra modules.

In Subsection \ref{sec:7.2}, we give a formula relating 
$r_j$ with an 
alternating sum of the dimensions of Lie algebra cohomologies of 
certain Harish-Chandra modules, and we establish Equation \eqref{eq:t02z}.



%
%
%

\subsection{Some results from representation theory}\label{sec:rep}
In this Subsection,  we do not assume that $\delta(G)=1$. We use  the 
notation in Section \ref{Sec:Pred} and the convention of real root 
systems  introduced in Subsection \ref{sec:d11}.
\subsubsection{Infinitesimal characters}
Let $\cZ(\fg_\bC)$ \index{Z@$\cZ(\fg_\bC)$}be the center of the 
enveloping algebra $U(\fg_\bC)$\index{U@$U(\fg_\bC)$} of the 
complexification $\fg_\bC$ of $\fg$. A morphism of algebras 
$\chi:\cZ(\fg_\bC)\to \bC$ will be called a character of $\cZ(\fg_\bC)$.

Recall that $\fh_1,\cdots,\fh_{l_0}$ form all the nonconjugated 
$\theta$-stable Cartan subalgebras of $\fg$. Let  
$\fh_{i\bC}=\fh_{i}\otimes_\bR\bC$ and $\fh_{i\bR}=\sqrt{-1}\fh_{i\fp}\oplus \fh_{i\fk}$ \index{H@$\fh_{i\bC},\fh_{i\bR}$}are the complexfication and real form of $\fh_i$. For $\alpha\in \fh_{i\bR}^*$, we extend $\alpha$ to a $\bC$-linear form on $\fh_{i\bC}$ by $\bC$-linearity. In this way, we identify $\fh_{i\bR}^*$ to a subset of $\fh_{i\bC}^*$.

For $1\l i\l l_0$, let $S(\fh_{i\bC})$ be the symmetric algebra of $\fh_{i\bC}$.
 The algebraic Weyl group $W(\fh_{i\bR},\fu)$ acts isometrically on $\fh_{i\bR}$. By $\bC$-linearity, $W(\fh_{i\bR},\fu)$ acts on $\fh_{i\bC}$.  Therefore,  $W(\fh_{i\bR},\fu)$ acts on $S(\fh_{i\bC})$. Let $S(\fh_{i\bC})^{W(\fh_{i\bR},\fu)}\subset S(\fh_{i\bC})$ be the $W(\fh_{i\bR},\fu)$-invariant subalgebra of $S(\fh_{i\bC})$.
 Let
\begin{align}
  \gamma_i:\cZ(\fg_\bC)\simeq S(\fh_{i\bC})^{W(\fh_{i\bR},\fu)}
\end{align}
be the Harish-Chandra isomorphism \cite[Section V.5]{KnappLie}. For $\Lambda\in \fh^*_{i\bC}$, we can associate to it a character $\chi_\Lambda$ of $\cZ(\fg_\bC)$ as follows: for $z\in \cZ(\fg_\bC)$,\index{C@$\chi_{\Lambda}$}
\begin{align}
  \chi_\Lambda(z)=\<\gamma_i(z),2\sqrt{-1}\pi \Lambda\>.
\end{align}
By \cite[Theorem 5.62]{KnappLie}, every character of $\cZ(\fg_\bC)$ is of the form $\chi_\Lambda$, for some $\Lambda\in \fh^*_{i\bC}$. Also, $\Lambda$ is uniquely determined up to an action of $W(\fh_{i\bR}, \fu)$.
Such an element $\Lambda\in \fh^*_{i\bC}$ is called  the Harish-Chandra parameter of the character. In particular, $\chi_{\Lambda}=0$ if and only if there is $w\in W(\fh_{i\bR},\fu)$ such that
\begin{align}
w\Lambda=\rho_i^{\fu},
\end{align}
where $\rho_i^{\fu}$ is defined as in \eqref{eq:rhok88} with respect 
to  $(\fh_{i\bR},\mathfrak u)$.

\begin{defin}
A complex representation of $\fg_\bC$ is said to have  infinitesimal character $\chi$,
if $z\in \cZ(\fg_\bC)$ acts as a scalar $\chi(z)\in\bC$.

A complex representation of $\fg_\bC$ is said to have  generalized infinitesimal character $\chi$,
if $z-\chi(z)$ acts nilpotently for all
$z\in \cZ(\fg_\bC)$, i.e., $(z-\chi(z))^i$ acts like $0$ for $i\gg1$.
\end{defin}

If $\lambda\in \fh_{i\bR}^*$ is algebraically integral and dominant, let $V_\lambda$ be the complex finite dimensional irreducible representation of the Lie algebra $\fg_\bC$ with the highest weight $\lambda$. Then $V_\lambda$ possesses  an infinitesimal character with Harish-Chandra parameter $\lambda+\rho_i^\fu\in \fh_{i\bR}^*$.

\subsubsection{Harish-Chandra $(\fg_\bC,K)$-modules and admissible representations of $G$}
We follow  \cite[p. 54-55]{HechtSchmid} and \cite[p. 207]{Knappsemi}.
\begin{defin}
We will say that a complex $U(\fg_\bC)$-module $V$, equipped with an  action of $K$,  is a Harish-Chandra $(\fg_\bC,K)$-module, if
the following conditions hold:
\begin{enumerate}
  \item The space $V$ is finitely generated as a $U(\fg_\bC)$-module;
  \item Every $v\in V$ lies in a finite dimensional, $\fk_\bC$-invariant subspace;
  \item The action of $\fg_\bC$ and $K$ are compatible;
  \item Each irreducible $K$-module occurs only finitely many times in $V$.
\end{enumerate}
\end{defin}

Let $V$ be a Harish-Chandra $(\fg_\bC,K)$-module. For  a character $\chi$ of $\cZ(\fg_\bC)$, let
$V_\chi\subset V$
be the largest submodule of $V$ on which $z-\chi(z)$ acts nilpotently for all $z\in \cZ(\fg_\bC)$.
Then $V_\chi$ is a Harish-Chandra $(\fg_\bC,K)$-submodule of $V$ with generalized infinitesimal character $\chi$.
By \cite[eq. (2.4)]{HechtSchmid}, we can decompose $V$ as a finite  sum of Harish-Chandra $(\fg_\bC,K)$-submodules
\begin{align}\label{eq:V=Vchi}
  V=\bigoplus_{\chi} V_\chi.
\end{align}

Any Harish-Chandra $(\fg_\bC,K)$-module $V$ has a finite  composition series in the following sense: there exist finitely many Harish-Chandra  $(\fg_\bC,K)$-submodules
\begin{align}
  V=V_{n_1}\supset V_{n_1-1}\supset \cdots\supset V_{0}\supset V_{-1}=0
\end{align}
such that each quotient $V_i/V_{i-1}$, for $0\l i\l n_1$, is an irreducible Harish-Chandra $(\fg_\bC,K)$-module. Moreover, the set of all irreducible quotients and their multiplicities are the same for all the composition series.

\begin{defin}
We say that a representation $\pi$ of $G$ on a Hilbert space is admissible if the followings hold:
\begin{enumerate}
    \item when restricted to $K$, $\pi|_K$ is unitary;
    \item each $\tau\in \widehat{K}$ occurs with only finite multiplicity in $\pi|_K$.
  \end{enumerate}
\end{defin}

Let $\pi$ be a finitely generated  admissible representation of  $G$ on the Hilbert space $V_\pi$. If $\tau\in \widehat{K}$, let $V_\pi(\tau)\subset V_\pi$  be the $\tau$-isotopic subspace of $V_\pi$. Then $V_\pi(\tau)$ is the image of the evaluation map
\begin{align}
 (f,v)\in \Hom_K(V_\tau,V_\pi)\otimes V_\tau\to f(v)\in V_\pi.
\end{align}
Let
\begin{align}\label{eq:V_piK}
 V_{\pi,K}=\bigoplus_{\tau\in \widehat{K}}V_\pi(\tau)\subset V_\pi
\end{align}
 be the algebraic sum of representations of $K$. By \cite[Proposition 
 8.5]{Knappsemi}, 
 $V_{\pi,K}$ is a Harish-Chandra $(\fg_\bC,K)$-module.
It is explained in \cite[Section4]{VoganRepCompAna} that,
by  results of Casselman, Harish-Chandra, Lepowsky and Wallach, any Harish-Chandra $(\fg_\bC,K)$-module $V$ can be constructed in this way and  the corresponding $V_\pi$ is called a Hilbert globalization of $V$.  Moreover, %
 $V$ is an irreducible Harish-Chandra $(\fg_\bC,K)$-module if and only if $V_\pi$ is an irreducible admissible  representation of $G$. In this case, $V$ or $V_\pi$ has an infinitesimal character.

 We note that a Hilbert globalization of a Harish-Chandra $(\fg_\bC,K)$-module is not unique.

\subsubsection{Global characters}
We recall the definition of the space of rapidly decreasing functions $\cS(G)$ on $G$ \index{S@$\cS(G)$} \cite[Section 7.1.2]{WallachI}.

For $z\in U(\fg)$, we denote by $z_{L}$ and $z_R$ respectively the corresponding left and right invariant differential operators on $G$.
For  $r\g0$, $z_1\in U(\fg)$, $z_2\in U(\fg)$,
and $f\in C^\infty(G)$, put\index{1@$\lVert\cdot\rVert_{r,z_1,z_2}$}
\begin{align}\label{eq:defsg}
  \|f\|_{r,z_1,z_2}=\sup_{g\in G} e^{rd_{X}(p1,pg)}|z_{1L}z_{2R}f(g)|.
\end{align}
Let $\cS(G)$ be the space of all
$f\in C^\infty(G)$ such that, for all $r\g0$, $z_1\in U(\fg)$, $z_2\in U(\fg)$,  $\|f\|_{r,z_1,z_2}<\infty$. We endow $\cS(G)$ with the topology given by the above semi-norms. By \cite[Theorem 7.1.1]{WallachI}, $\cS(G)$ is a  Fr\'echet space which contains $C^\infty_c(G)$ as a dense subspace.


Let $\pi$ be a finitely generated admissible representation of  $G$ 
on the Hilbert space $V_\pi$. By \cite[Lemma 2.A.2.2]{WallachI}, there exists $C>0$ such that for $g\in G$, we have
\begin{align}\label{eq:pig}
  \|\pi(g)\|\l Ce^{Cd_X(p1,pg)},
\end{align}
where $\|\cdot\|$ is the operator norm.
By \eqref{eq:pig}, if $f\in \cS(G)$,
\begin{align}\label{eq:pig3}
  \pi(f)=\int_G f(g)\pi(g)dg
\end{align}
is a  bounded operator on $V_\pi$.
By \cite[Lemma 8.1.1]{WallachI}, $\pi(f)$ is trace class.  The global character $\Theta^G_\pi$ \index{T@$\Theta^G_\pi$} of $\pi$ is a continuous linear functional on  $\cS(G)$ such that for $f\in \cS(G)$,
\begin{align}\label{eq:trpi}
  \Tr[\pi(f)]=\< \Theta^G_\pi,f\>.
\end{align}
If $V$ is a Harish-Chandra $(\fg_\bC,K)$-module, we can define
the global character $\Theta^G_V$ of $V$ by
the global character of its Hilbert globalization. We note that the global character does not depend on the choice of
Hilbert globalization  \cite[p.56]{HechtSchmid}.

By Harish-Chandra's regularity theorem \cite[Theorems 10.25] {Knappsemi}, there is an $L^1_{loc}$ and $\Ad(G)$-invariant function $\Theta^G_\pi(g)$ on $G$, whose restriction to the regular set $G'$ is analytic, such that for $f\in C_c^\infty(G)$, we have
\begin{align}
  \left\<\Theta^G_{\pi},f\right\>=\int_{g\in G} \Theta^G_\pi(g)f(g)dv_G.
\end{align}


\begin{prop}\label{prop:HH}
  If $f\in \mathcal{S}(G)$, then $\Theta^G_{\pi}(g)f(g)\in L^1(G)$ such that
  \begin{align}\label{eq:HH1}
    \left\<\Theta^G_\pi,f\right\>=\int_{g\in G} \Theta^G_\pi(g)f(g)dv_G.
  \end{align}
\end{prop}
\begin{proof}It is enough to show that there exist $C>0$ and a seminorm $\|\cdot\|$ on $\mathcal{S}(G)$ such that
\begin{align}\label{eq:HH2}
  \int_G |\Theta^G_\pi(g)f(g)|dg\l C\|f\|.
\end{align}
Recall that $H'$ is defined in \eqref{eq:Mul1}.
By \eqref{eq:weylnoncom}, we need to show that there exist $C>0$ and a semi-norm  $\|\cdot\|$ on $\mathcal{S}(G)$ such that
   for $1\l i\l l_0$, we have
  \begin{align}\label{eq:HH}
    \int_{\gamma\in H'_i}\left| \Theta^G_\pi(\gamma)\right| \(\int_{g\in H_i\backslash G} |f(g^{-1}\gamma g)|dv_{H_i\backslash G}\)\left|\det\big(1-\Ad(\gamma)\big)|_{\fg/\fh_i}\right|dv_{H_i}\l C\|f\|.
  \end{align}

  By \cite[Theorem 10.35]{Knappsemi},
  there exist $ C>0$ and $r_0>0$ such that, for $\gamma=e^ak^{-1}\in H_i'$ with $a\in \fh_{i\fp}$, $k\in H_i\cap K$, we have
\begin{align}\label{eq:estgol}
 | \Theta^G_{\pi}(\gamma)|\left|\mathrm{det}\big(1-\Ad(\gamma)\big)|_{\fg/\fh_i}\right|^{1/2}\l Ce^{r_0|a|}.
\end{align}

We claim that there exist  $r_1>0$ and $C>0$, such that for $\gamma\in H'_i$, we have
  \begin{align}\label{eq:es22}
   \left|\det\big(1-\Ad(\gamma)\big)|_{\fg/\fh_i}\right|^{1/2} \int_{g\in H_i\backslash G} \exp\(-r_1d_X(p1, g^{-1}\gamma g\cdot p1)\) dv_{H_i\backslash G}\l C.
  \end{align}
Indeed, let $\Xi(g)$ \index{X@$\Xi(g)$} be the Harish-Chandra's $\Xi$-function \cite[Section II.8.5]{VaradarajanHAred}. By \cite[Section II.12.2, Corollary 5]{VaradarajanHAred}, there exist  $r_2>0$ and $C>0$, such that for $\gamma\in H'_i$, we have
  \begin{align}\label{eq:es223}
   \left|\det\big(1-\Ad(\gamma)\big)|_{\fg/\fh_i}\right|^{1/2} \int_{g\in H_i\backslash G} \Xi(g^{-1}\gamma g )\(1+d_X(p1, g^{-1}\gamma g\cdot p1)\)^{-r_2} dv_{H_i\backslash G}\l C.
  \end{align}
  By \cite[Proposition 7.15 (c)]{Knappsemi} and by \eqref{eq:es223}, we get \eqref{eq:es22}.

   By \cite[(3.1.10)]{B09},  for $g\in G$ and $\gamma=e^ak^{-1}\in H_i$, we have
  \begin{align}\label{eq:dxaa}
    d_X\(p1,g^{-1}\gamma g\cdot p1\)\g |a|.
  \end{align}
  Take $r= 2r_0+r_1$, $z_1=z_2=1\in U(\fg)$.
  Since  $f\in \mathcal{S}(G)$, by \eqref{eq:defsg} and \eqref{eq:dxaa},  for $\gamma=e^ak^{-1}\in H_i'$, we have
  \begin{multline}\label{eq:es33}
    |f(g^{-1}\gamma g)|\l \|f\|_{r,z_1,z_2} \exp\(-rd_X(p1, g^{-1}\gamma g \cdot p1)\)\\
    \l \|f\|_{r,z_1,z_2} \exp(-2r_0|a|)\exp\(-r_1d_X(p1, g^{-1}\gamma g \cdot p1)\).
  \end{multline}
  By \eqref{eq:estgol}, \eqref{eq:es22}, and \eqref{eq:es33}, for $\gamma\in H_i'$, we have
  \begin{align}\label{eq:720}
    \left| \Theta^G_\pi(\gamma)\right| \(\int_{g\in H_i\backslash G} |f(g^{-1}\gamma g)|dv_{H_i\backslash G}\)\left|\det\big(1-\Ad(\gamma)\big)|_{\fg/\fh_i}\right|
    \l C\|f\|_{r,z_1,z_2} \exp(-r_0|a|).
  \end{align}

  By \eqref{eq:720}, we get \eqref{eq:HH}. The proof of our proposition is completed.
\end{proof}

Let $V$ be a Harish-Chandra $(\fg_\bC,K)$-module, and let
 $\tau$ be a real finite dimensional orthogonal representation of $K$ on the real Euclidean space $E_\tau$.
Then the invariant subspace $(V\otimes_\bR E_\tau)^K\subset V\otimes_\bR E_\tau$ has a finite dimension.
We will describe an integral formula for $\dim_\bC(V\otimes_\bR E_\tau)^K$, which extends \cite[Corollary 2.2]{MoscoviciL2}.

Recall that $p^{X,\tau}_t(g)$ is the smooth integral  kernel of $\exp\(-tC^{X,\tau}/2\)$. By the estimation
on the heat kernel or
by \cite[Proposition 2.4]{MoscoviciL2}, $p^{X,\tau}_{t}(g) \in \mathcal{S}(G)\otimes \End(E_\tau)$. Recall that $dv_{G}$ \index{d@$dv_G$} is the Riemannian volume on $G$ induced by $-B(\cdot, \theta\cdot)$. 

\begin{prop}\label{prop:int}
  Let $f\in C^\infty(G,E_\tau)^K$. Assume that there exist  $C>0$ and $r>0$ such that
  \begin{align}\label{eq:fas}
    |f(g)|\l C\exp\big(r d_X(p1,pg)\big).
  \end{align}
  The integral
  \begin{align}\label{eq:fas1}
   \int_{g\in G} p^{X,\tau}_t(g)f(g)dv_G\in E_\tau
  \end{align}
  is well defined such that
  \begin{align}\label{eq:fas2}
 \begin{split}
    \frac{\p}{\p t} \int_{g\in G} &p^{X,\tau}_t(g)f(g)dv_G=-\frac{1}{2}\int_{g\in G} C^{\fg}p^{X,\tau}_t(g)f(g)dv_G,\\
   &\frac{1}{\vol(K)}\lim_{t\to 0} \int_{g\in G} p^{X,\tau}_t(g)f(g)dv_G=f(1).
\end{split}
\end{align}
\end{prop}
\begin{proof}
By \eqref{eq:fas}, by the property of $\mathcal{S}(G)$ and by $\frac{\p}{\p t}p_t^{X,\tau}(g)=-\frac{1}{2}C^\fg p_t^{X,\tau}(g)$, the left-hand side of \eqref{eq:fas1} and the right-hand side of the first equation of \eqref{eq:fas2} are well defined such that the first equation of \eqref{eq:fas2} holds true.

It remains to show the second equation of \eqref{eq:fas2}.
Let $\phi_1\in C_c^\infty(G)^K$
such that $0\l \phi_1(g)\l 1$ and that
\begin{align}
  \phi_1(g)=\left\{
            \begin{array}{ll}
              1, & d_X(p1,pg)\l 1, \\
              0, & d_X(p1,pg)\g 2.
            \end{array}
          \right.
\end{align}
Set $\phi_2=1-\phi_1$.

Since $\phi_1f$ has compact support, it descends to an $L^2$-section on $X$ with values in $G\times_K E_\tau$.
We have
\begin{align}\label{eq:ll1}
\frac{1}{\vol(K)}  \lim_{t\to 0}    \int_{g\in G} p^{X,\tau}_t(g)\phi_1(g)f(g)dv_G=f(1).
\end{align}

By \eqref{eq:estheat}, there exist $c>0$ and $C>0$ such that
for $g\in G$ with $d_X(p1,pg)\g 1$ and for $t\in (0,1]$,  we have
\begin{align}\label{eq:dff}
  \left|p^{X,\tau}_t(g)\right|\l C\exp\(-c\frac{d^2_{X}(p1,pg)}{t}\)\l  Ce^{-c/2t}\exp\(-c\frac{d^2_{X}(p1,pg)}{2t}\).
\end{align}
By \eqref{eq:fas} and \eqref{eq:dff}, there exist $c>0$ and $C>0$ such that for $t\in (0,1]$, we have
\begin{align}\label{eq:ll2}
   \int_{g\in G} \left|p^{X,\tau}_t(g)\phi_2(g)f(g)dv_G\right|\l Ce^{-c/2t}.
\end{align}

By \eqref{eq:ll1} and \eqref{eq:ll2}, we get
the second equation of \eqref{eq:fas2}. The proof of our proposition  is completed.
\end{proof}

\begin{prop}\label{prop:dimeq}
  Let $V$ be a Harish-Chandra $(\fg_\bC,K)$-module with generalized infinitesimal character  $\chi$. For $t>0$, we have
\begin{align}\label{eq:dimeq}
  \dim_\bC \(V\otimes_\bR E_\tau\)^K= \vol(K)^{-1}e^{t\chi(C^{\fg})/2} \int_{g\in G}\Theta^G_V(g)\Tr\[p^{X,\tau}_{t}(g)\]dv_G.
\end{align}
\end{prop}
\begin{proof} Let $V_\pi$ be a Hilbert globalization of $V$. Then,
\begin{align}\label{eq:vvp}
  \(V\otimes_\bR E_\tau\)^K=\(V_{\pi}\otimes_\bR E_\tau\)^K.
\end{align}%
As in \eqref{eq:pig3}, set
\begin{align}\label{eq:pj}
  \pi\(p^{X,\tau}_t\)=\frac{1}{\vol(K)}\int_{g\in G}\pi(g)\otimes_{\bR} p^{X,\tau}_t(g)dv_G.
\end{align}
Then, $\pi\(p^{X,\tau}_t\)$ is a bounded operator acting on $V_\pi\otimes_\bR  E_\tau$.

We follow \cite[p. 160-161]{MoscoviciL2}.  Let $\(V_\pi\otimes_\bR  E_\tau\)^{K,\bot}$ be the orthogonal space to $\(V_\pi\otimes_\bR  E_\tau\)^K$ in $V_\pi\otimes_\bR  E_\tau$, such that
\begin{align}\label{eq:decom}
  V_\pi\otimes_\bR  E_\tau=\(V_\pi\otimes_\bR  E_\tau\)^K\oplus \(V_\pi\otimes_\bR  E_\tau\)^{K,\bot}.
\end{align}
Let $Q_{\pi,\tau}$ be the orthogonal projection from $V_\pi\otimes_\bR  E_\tau$ to $\(V_\pi\otimes_\bR  E_\tau\)^K$. Then,
\begin{align}\label{eq:QQ}
  Q_{\pi,\tau} =\frac{1}{\vol(K)}\int_{k\in K} \pi\otimes \tau (k) dv_K.
\end{align}
By \eqref{eq:qtautype}, \eqref{eq:pj} and \eqref{eq:QQ}, we get
\begin{align}\label{eq:QPQ}
  Q_{\pi,\tau}\pi\(p^{X,\tau}_t\)Q_{\pi,\tau}=\pi\(p^{X,\tau}_t\).
\end{align}
In particular, $\pi(p^{X,\tau}_t)$ is of finite rank.

Take $u\in \(V_\pi\otimes_{\bR} E_\tau\)^K$ and $v\in V_\pi$. Define $\<u,v\>\in E_\tau$ be such that for any $w\in E_\tau$,
\begin{align}
\big\<\<u,v\>,w\big\>=\< u,v\otimes_{\bR} w\>.
\end{align}
 By \eqref{eq:pig}, the function $g\in G\to \<\pi(g)\otimes_{\bR}{\rm id}\cdot u,v\>\in E_\tau$ is of class $C^\infty(G,E_\tau)^K$ such that \eqref{eq:fas} holds.
By \eqref{eq:pj}, we have
\begin{align}\label{eq:xb}
 \left\< \pi\(p^{X,\tau}_t\)u, v\right\>=\frac{1}{\vol(K)}\int_{g\in G}p^{X,\tau}_t(g)\<\pi(g)\otimes_{\bR}{\rm id}\cdot u,v\>dv_G.
\end{align}
By Proposition \ref{prop:int} and \eqref{eq:xb}, we have
\begin{align}\label{eq:P12}
&\frac{\p}{\p t}\left\<\pi\(p^{X,\tau}_t\)u,v\right\>=-\frac{1}{2}\left\<\pi\(C^\fg\)\pi\(p^{X,\tau}_t\)u,v\right\>,
  &\lim_{t\to 0}\left\< \pi\(p^{X,\tau}_t\)u, v\right\>=\<u,v\>
\end{align}

Since $C^\fg\in \mathcal{Z}(\fg)$ and since $\pi(C^{\fg})$ preserves the splitting \eqref{eq:decom},
by \eqref{eq:QPQ} and \eqref{eq:P12}, under the splitting \eqref{eq:decom},
we have
\begin{align}\label{eq:trcc}
   \pi\(p^{X,\tau}_t\)=\left(
                       \begin{array}{cc}
                         e^{-t\pi(C^{\fg})/2} & 0 \\
                         0 & 0 \\
                       \end{array}
                     \right).
\end{align}
Since $V$ has a generalized infinitesimal character  $\chi$, by \eqref{eq:trcc}, we have
\begin{align}\label{eq:tronpi}
  \Tr\[\pi\big(p^{X,\tau}_t\big)\]=e^{-t\chi(C^{\fg})/2} \dim_\bC \(V_\pi\otimes_\bR E_\tau\)^K.
\end{align}

Let $(\xi_i)_{i=1}^\infty$ and $(\eta_j)_{j=1}^{\dim E_\tau}$ be orthogonal basis of $V_\pi$ and $E_\tau$. Then
\begin{align}\label{eq:sd}
\begin{aligned}
  \Tr\[\pi\big(p^{X,\tau}_t\big)\]=&\frac{1}{\vol(K)}\sum_{i=1}^\infty\sum_{j=1}^{\dim E_\tau}\int_{g\in G} \<p_{t}^{X,\tau}(g)\eta_j,\eta_j\> \<\pi(g)\xi_i,\xi_i\>dv_G\\
  =&\frac{1}{\vol(K)}\sum_{i=1}^\infty\int_{g\in G} \Tr\[p_{t}^{X,\tau}(g)\] \<\pi(g)\xi_i,\xi_i\>dv_G.
\end{aligned}
\end{align}
Since $\Tr\[p_{t}^{X,\tau}(g)\]\in \cS(G)$, by \eqref{eq:HH1} and \eqref{eq:sd}, we have
\begin{align}\label{eq:pidot}
\Tr\[\pi\big(p^{X,\tau}_t\big)\]=\frac{1}{\vol(K)}\int_{g\in G} \Tr\big[p^{X,\tau}_t(g)\big] \Theta^G_\pi(g)dv_G.
\end{align}

From  \eqref{eq:vvp}, \eqref{eq:tronpi} and \eqref{eq:pidot}, we get 
\eqref{eq:dimeq}. The proof of  our proposition  is completed.
\end{proof}


\begin{prop}\label{prop:test}
For $1\l i\l l_0$, the function
\begin{align}\label{eq:test2}
\gamma\in H_i' \to  \Tr^{[\gamma]}\[\exp\(-tC^{\fg,X,\tau}/2\)\] \Theta^G_\pi(\gamma)\left|\mathrm{det}\big(1-\Ad(\gamma)\big)|_{\fg/\fh_i}\right|
\end{align}
is almost everywhere well defined and integrable on $H_i'$, so that
\begin{multline}\label{eq:test1}
  \int_{g\in G} \Tr\big[p^{X,\tau}_t(g)\big] \Theta^G_\pi(g)dv_G
=\sum_{i=1}^{l_0}\frac{\vol(K\cap H_i\backslash K)}{|W(H_i,G)|}\\
\int_{\gamma\in H'_i} \Tr^{[\gamma]}\[\exp\(-tC^{\fg,X,\tau}/2\)\]  \Theta^G_\pi(g)\left|\mathrm{det}\big(1-\Ad(\gamma)\big)|_{\fg/\fh_i}\right|dv_{H_i}.
\end{multline}
\end{prop}
\begin{proof}Since $\Tr\big[p^{X,\tau}_t(g)\big] \Theta^G_\pi(g)\in L^1(G)$, by \eqref{eq:weylnoncom} and by Fubini Theorem, the function
\begin{align}\label{eq:weylob2}
  \gamma\in H_i\to\(\int_{g\in H_i\backslash G}\Tr^{E_\tau}\[p_t^{X,\tau}\(g^{-1}\gamma g\)\]dv_{H_i\backslash G} \)\Theta^G_\pi(\gamma)\left|\mathrm{det}(1-\Ad(\gamma))|_{\fg/\fh_i}\right|
\end{align}
is  almost everywhere well defined and integrable on $H_i$.

Take $\gamma\in H_i'$. Since $H_i$ is abelian, we have
\begin{align}\label{eq:62}
Z^0(\gamma)=H_i^0\subset H_i\subset  Z(\gamma).
\end{align}
We have a finite covering space $H^0_i\backslash  G\to H_i\backslash  G$. Note that
\begin{align}\label{eq:coverHK}
\[H_i:H_i^0\]=\[K\cap H_i:K\cap H^0_i\].
\end{align}
 By \eqref{eq:TRrz}, \eqref{eq:62} and \eqref{eq:coverHK}, if $\gamma\in H'_i$, we have
\begin{multline}\label{eq:weylob}
  \int_{H_i\backslash G}\Tr^{E_\tau}\[p_t^{X,\tau}(g^{-1}\gamma g)\]dv_{H_i\backslash G} =\frac{\vol(K^0(\gamma)\backslash K)}{[H_i:H_i^0]}\Tr^{[\gamma]}\[\exp\(-tC^{\fg,X,\tau}/2\)\]\\
  =\vol(K\cap H_i\backslash K)\Tr^{[\gamma]}\[\exp\(-tC^{\fg,X,\tau}/2\)\].
\end{multline}

Since $H_i-H_i'$ has zero measure, and by \eqref{eq:weylob2} and 
\eqref{eq:weylob}, the function \eqref{eq:test2} defines an 
$L^1$-function on $H_i'$. By \eqref{eq:weylnoncom} and 
\eqref{eq:weylob}, we get \eqref{eq:test1}. The proof of our  proposition is completed.
\end{proof}


\subsubsection{The $(\fg,K)$-cohomology}
If $V$ is a Harish-Chandra $(\fg_\bC,K)$-module, let $H^\cdot(\fg,K;V)$ \index{H@$H^\cdot(\fg,K;V)$} be the $(\fg,K)$-cohomology of $V$ \cite[Section I.1.2]{BW}. The following two theorems are the essential algebraic ingredients in our proof of \eqref{eq:t02z}.

\begin{thm}\label{thm:vankey}
Let $V$ be a Harish-Chandra $(\fg_\bC,K)$-module with generalized infinitesimal  character $\chi$.
Let $W$ be a finite dimensional $\fg_\bC$-module with infinitesimal character.
 Let $\chi^{W^*}$ be the  infinitesimal character of $W^*$. If $\chi\neq \chi^{W^*}$, then
\begin{align}\label{eq:key1}
   H^\cdot(\fg,K;V\otimes W)=0.
\end{align}
\end{thm}
\begin{proof}
  If  $\chi$ is  the infinitesimal character of $V$,  then  \eqref{eq:key1} is  a consequence of  \cite[Theorem I.5.3 (ii)]{BW}.

In general, let
\begin{align}
V=V_{n_1}\supset V_{n_1-1}\supset \cdots \supset V_0\supset V_{-1}=0
\end{align}
be the composition series of $V$. Then for $0\l i \l n_1$, $V_{i}/V_{i-1}$ is an irreducible Harish-Chandra $(\fg_\bC,K)$-module with infinitesimal character $\chi$. Therefore, for all $0\l i \l n_1$, we have
\begin{align}\label{eq:gk0}
  H^\cdot(\fg,K;(V_{i}/V_{i-1})\otimes W)=0.
\end{align}

We will show by induction that, for all $0\l i\l n_1$,
\begin{align}\label{eq:gk1}
  H^\cdot(\fg,K;V_i\otimes W)=0.
\end{align}
 By \eqref{eq:gk0}, Equation \eqref{eq:gk1} holds for $i=0$. Assume that
\eqref{eq:gk1} holds for some $i$ with $0\l i\l n_1$. Using the short exact  sequence of Harish-Chandra $(\fg_\bC,K)$-modules
 \begin{align}
   0\to V_{i}\to V_{i+1}\to V_{i+1}/V_{i}\to 0,
 \end{align}
  we get the long exact sequence of cohomologies
  \begin{align}\label{eq:loex}
   \cdots \to H^j(\fg,K;V_{i}\otimes W)\to H^j(\fg,K;V_{i+1}\otimes W)\to H^j(\fg,K;(V_{i+1}/V_{i})\otimes W)\to \cdots.
 \end{align}
By \eqref{eq:gk0}, \eqref{eq:loex} and by the induction hypotheses, 
Equation \eqref{eq:gk1} holds for $i+1$, which completes the proof of 
\eqref{eq:gk1}. The proof of  our theorem is completed.
\end{proof}

We denote by $\widehat{G}_u$ the unitary dual of $G$, that is  the set of equivalence classes  of complex irreducible unitary representations $\pi$ of $G$ on Hilbert spaces $V_\pi$. If $(\pi,V_\pi)\in \widehat{G}_u$, by \cite[Theorem 8.1]{Knappsemi}, $\pi$ is irreducible admissible.  Let $\chi_\pi$  be the corresponding infinitesimal character. 

\begin{thm}\label{thm:vankey1}
If  $(\pi,V_\pi)\in \widehat{G}_u$, then
\begin{align}\label{eq:key2}
 \chi_{\pi}\neq0 \iff H^\cdot(\fg,K;V_{\pi,K})=0.
\end{align}
\end{thm}
\begin{proof}
  The  direction $\implies$ of \eqref{eq:key2} is \eqref{eq:key1}.
The  direction  $\impliedby$  of \eqref{eq:key2} is a consequence of Vogan-Zuckerman \cite{VoganZuckerman}, Vogan \cite{Vogan2} and  Salamanca-Riba \cite{Salamanca}. Indeed, the irreducible unitary representations with nonvanishing $(\fg,K)$-cohomology are classified
and constructed in \cite{VoganZuckerman, Vogan2}.
By \cite{Salamanca}, the irreducible unitary representations with vanishing infinitesimal character is
in the class specified by  Vogan and Zuckerman, which implies that their $(\fg,K)$-cohomology are nonvanishing.
\end{proof}

\begin{re} The condition that $\pi$ is unitary is crucial in the \eqref{eq:key2}. See \cite[Section 9.8.3]{WallachI} for a counterexample.
\end{re}
%

\subsubsection{The Hecht-Schmid character formula}
Let us recall the main result of \cite{HechtSchmid}. Let $Q\subset G$ be a standard parabolic subgroup of $G$ with  Lie algebra $\fq\subset \fg$. Let
\begin{align}
&    Q=M_QA_QN_Q,&\fq=\fm_\fq\oplus \fa_\fq\oplus \fn_\fq
\end{align}
be the corresponding Langlands decompositions \cite[Section V.5]{Knappsemi}.

Put $ \Delta^+(\fa_\fq,\fn_\fq)$ to be the set of all linear forms $\alpha\in \fa^*_\fq$ such that there exists a nonzero element $Y\in \fn_\fq$ such that for all $a\in \fa_\fq$,
\begin{align}
  \ad(a)Y=\<\alpha,a\>Y.
\end{align}
Set
\begin{align}
  \fa_\fq^-=\{a\in \fa_\fq: \<\alpha,a\><0, \hbox{ for all } \alpha\in \Delta^+(\fa,\fn)\}.
\end{align}
Put  $\(M_QA_Q\)^{-}$ to be the interior in $M_QA_Q$  of the set
\begin{align}\label{eq:HSMA-}
  \{g\in M_QA_Q: \mathrm{det}\big(1-\Ad(ge^a)\big)|_{\fn_\fq}\g 0 \hbox{ for all } a\in \fa_\fq^-\}.
\end{align}

If $V$ is a Harish-Chandra $(\fg_\bC,K)$-module, let $H_\cdot(\fn_{\fq},V)$ be the $\fn_{\fq}$-homology of $V$.  By \cite[Proposition 2.24]{HechtSchmid}, $H_\cdot(\fn_\fq,V)$ \index{H@$H_\cdot(\fn_\fq,V)$} is a Harish-Chandra  $(\fm_{\fq\bC}\oplus\fa_{\fq\bC}, K\cap M_Q)$-module. We denote by $\Theta^{M_QA_Q}_{H_\cdot(\fn_{\fq},V)}$ the corresponding global character. Also, $M_QA_Q$ acts on $\fn_\fq$. We denote by $\Theta^{M_QA_Q}_{\Lambda^\cdot\(\fn_{\fq}\)}$ the character of $\Lambda^\cdot(\fn_\fq)$.
By \cite[Theorem 3.6]{HechtSchmid}, the following identity of analytic functions on $\(M_QA_Q\)^{-}\cap G'$ holds:
\begin{align}\label{eq:smid}
  \Theta^G_{V}|_{\(M_QA_Q\)^{-}\cap G'}=\frac{\sum^{\dim \fn_\fq}_{i=0}(-1)^i\Theta^{M_QA_Q}_{H_i(\fn_{\fq},V)}}{\sum^{\dim \fn_\fq}_{i=0}(-1)^i\Theta^{M_QA_Q}_{\Lambda^i\(\fn_{\fq}\)}}\bigg|_{\(M_QA_Q\)^{-}\cap G'}.
\end{align}

Take a $\theta$-stable
 Cartan subalgebra $\fh^{\fm_\fq}$ of $\fm_\fq$. Set  $\fh_\fq=\fh^{\fm_\fq}\oplus \fa_\fq$. Then  $\fh_\fq$ is a $\theta$-stable
 Cartan subalgebra of both $\fm_\fq\oplus\fa_\fq$ and $\fg$. Put $\fu_\fq$ to be the compact form of $\fm_{\fq}\oplus\fa_{\fq}$. Then
$\fh_{\fq\bR}$, the real form of $\fh_\fq$, is a Cartan subalgebra of both $\fu_\fq$ and $\fu$. The real root system of $\Delta(\fh_{\fq\bR},\fu_\fq)$ is a subset of $ \Delta(\fh_{\fq\bR},\fu)$ consisting of the elements whose restriction to $\fa_\fq$  vanish. The set of positive real roots $\Delta^+(\fh_{\fq\bR},\fu)\subset \Delta(\fh_{\fq\bR},\fu)$  determines a set of positive real roots $\Delta^+(\fh_{\fq\bR}, \fu_\fq)\subset \Delta(\fh_{\fq\bR}, \fu_\fq)$. Let $\rho^{\fu}_\fq$ and $\rho^{\fu_\fq}_\fq$ be the corresponding half sums of positive real roots.

 If $V$ possesses  an infinitesimal character with Harich-Chandra parameter $\Lambda\in \fh^*_{\fq\bC}$, by \cite[Corollary 3.32]{HechtSchmid},  $H_\cdot(\fn_\fq,V)$ can be decomposed in the sense of \eqref{eq:V=Vchi}, where the generalized infinitesimal characters are given by
\begin{align}
\chi_{w\Lambda+\rho_\fq^\fu-\rho_\fq^{\fu_\fq}},
\end{align}
for some $w\in W(\fh_{\fq\bR},\fu)$.

Also, $H_\cdot(\fn,V)$ is a Harish-Chandra $(\fm_{\fq\bC},K\cap M_Q)$-module. For $\nu\in \fa_{\fq\bC}^*$, let $H_\cdot(\fn,V)_{[\nu]}$ be the largest submodule of $H_\cdot(\fn,V)$ on which $z-\<2\sqrt{-1}\pi\nu,z\>$ acts nilpotently for all $z\in \fa_{\fq\bC}$. Then,
\begin{align}
  H_\cdot(\fn,V)=\bigoplus_{\nu}H_\cdot(\fn,V)_{[\nu]},
\end{align}
where $\nu=(w\Lambda+\rho_\fq^\fu-\rho_\fq^{\fu_\fq})|_{\fa_{\fq\bC}}$, for some $w\in W(\fh_{\fq\bR},\fu)$. Let $\Theta^{M_Q}_{H_\cdot(\fn,V)}$ and $\Theta^{M_Q}_{H_\cdot(\fn,V)_{[\nu]}}$ be the corresponding global characters. We have the identities of $L^1_{loc}$-functions: for $m\in M_Q$ and $a\in \fa_\fq$,
\begin{align}\label{eq:MAtoM}
&\Theta^{M_QA_Q}_{H_\cdot(\fn,V)}\(me^a\)=\sum_{\nu}e^{2\sqrt{-1}\pi\<\nu,a\>}  \Theta^{M_Q}_{H_\cdot(\fn,V)_{[\nu]}}(m),&\Theta^{M_Q}_{H_\cdot(\fn,V)}\(m\)=\sum_{\nu}  \Theta^{M_Q}_{H_\cdot(\fn,V)_{[\nu]}}(m),
\end{align}
where $\nu=(w\Lambda+\rho_\fq^\fu-\rho_\fq^{\fu_\fq})|_{\fa_{\fq\bC}}$, for some $w\in W(\fh_{\fq\bR},\fu)$.


Consider now $G$ is such that $\delta(G)=1$ and has compact center. 
Use the notation in Subsection \ref{sec:dG=1}. Take $\fq=\fm\oplus \fb\oplus \fn$, and let $Q=M_QA_QN_Q\subset G$ be the corresponding parabolic subgroup. Then $M$ is the connected component of the identity in $M_Q$. Since $K\cap M_Q$ has  a finite  number of connected components, $H_\cdot(\fn,V)$ is still a Harish-Chandra $(\fm_\bC\oplus \fb_\bC,K_M)$-module.  Also, it is a Harish-Chandra $(\fm_\bC,K_M)$-module. Let $\Theta^{MA_Q}_{H_\cdot(\fn,V)}$ and $\Theta^{M}_{H_\cdot(\fn,V)}$ be the respective global characters.

Recall that $H=\exp(\fb)T\subset MA_Q$ is the Cartan subgroup of $MA_Q$.

\begin{prop}\label{prop:G1mq}
  We have
\begin{align}\label{eq:G1mq}
 \bigcup_{g\in MA_Q} g H' g^{-1}\subset (M_QA_Q)^{-}\cap G'.
\end{align}
\end{prop}
\begin{proof}Put $L'=\bigcup_{g\in MA_Q} g H' g^{-1}\subset MA_Q\cap G'$. Then $L'$ is an open subset of $MA_Q$. It is enough to show that $L'$ is a subset of \eqref{eq:HSMA-}.

By \eqref{eq:det1/2} and \eqref{eq:toot2}, for $\gamma=e^ak^{-1}\in H$ with $a\in \fb$ and $k\in T$, we have
$  \mathrm{det} \big(1-\Ad(\gamma)\big)|_\fn\g 0.$ Therefore, $L'$ is 
a subset of \eqref{eq:HSMA-}.  The proof of our proposition is completed.
\end{proof}

\subsection{Formulas for $r_{\eta,\rho}$ and $r_j$}\label{sec:7.2}
Recall that $\widehat{p}:\Gamma\backslash G\to Z$ is the natural projection. The group $G$ acts unitarily  on the right on $L^2(\Gamma\backslash G,\widehat{p}^*F)$. By \cite[p.23, Theorem]{GMP}, we can decompose $L^2(\Gamma\backslash G,\widehat{p}^*F)$ into a
direct sum of unitary representations of $G$,
\begin{align}\label{eq:Geldecomp}
 L^2\(\Gamma\backslash G,\widehat{p}^*F\)=\bigoplus_{\pi\in \widehat{G}_u}^{\mathrm{Hil}}n_{\rho}(\pi)V_\pi,
\end{align}
with $n_\rho(\pi)<\infty$.

Recall that $\tau$  is a real finite dimensional orthogonal representation of $K$ on the real Euclidean space $E_\tau$, and that $C^{\fg,Z,\tau,\rho}$ is the Casimir element of $G$ acting on $C^\infty(Z,\cF_\tau\otimes_\bC F)$.
By \eqref{eq:Geldecomp}, we have
\begin{align}\label{eq:kerC=sum}
  \ker C^{\fg,Z,\tau,\rho}=\bigoplus_{\pi\in \widehat{G}_u, \chi_\pi(C^{\fg})=0}n_\rho(\pi) \big(V_{\pi,K}\otimes_\bR E_\tau\big)^K.
\end{align}
By the properties of elliptic operators, the sum on right-hand side of \eqref{eq:kerC=sum} is finite.

We will give two applications of \eqref{eq:kerC=sum}.
In our first application, we take $E_\tau=\Lambda^\cdot(\fp^*)$.
\begin{prop}\label{prop:mushi}
We have
\begin{align}\label{eq:mushifi}
H^\cdot(Z,F)=\bigoplus_{\pi\in \widehat{G}_u,\chi_\pi=0}n_{\rho}(\pi)H^\cdot(\fg,K;V_{\pi,K}).
\end{align}
If $H^\cdot(Z,F)=0$, then for any $\pi\in \widehat{G}_u$ such that $\chi_\pi=0$, we have
\begin{align}\label{eq:mushiva}
  n_\rho(\pi)=0.
\end{align}
\end{prop}
\begin{proof}By Hodge theory, and by \eqref{eq:D=C}, \eqref{eq:kerC=sum}, we have
\begin{align}\label{eq:Hodv_Gecla}
  H^\cdot(Z,F)=\bigoplus_{\pi\in \widehat{G}_u,\chi_\pi(C^{\fg})=0}n_{\rho}(\pi)\big(V_{\pi,K}\otimes_\bR\Lambda^\cdot(\fp^*)\big)^K.
\end{align}
By Hodge theory for Lie algebras \cite[Proposition II.3.1]{BW}, if $\chi_\pi(C^{\fg})=0$, we have
\begin{align}\label{eq:HogLie}
  \big(V_{\pi,K}\otimes_\bR\Lambda^\cdot(\fp^*)\big)^K=H^\cdot(\fg,K;V_{\pi,K}).
\end{align}
From \eqref{eq:Hodv_Gecla} and \eqref{eq:HogLie}, we get
\begin{align}\label{eq:Mushi}
  H^\cdot(Z,F)=\bigoplus_{\pi\in \widehat{G}_u,\chi_\pi(C^{\fg})=0}n_{\rho}(\pi)H^\cdot(\fg,K;V_{\pi,K}).
\end{align}
By \eqref{eq:key1} and \eqref{eq:Mushi},
we get \eqref{eq:mushifi}.

By Theorem \ref{thm:vankey1}, and by \eqref{eq:mushifi}, we get 
\eqref{eq:mushiva}. The proof of our  proposition is completed.
\end{proof}
\begin{re}
	Equation \eqref{eq:mushifi} is 
	\cite[Proposition VII.3.2]{BW}.	When $\rho$ is a trivial 
	representation,  \eqref{eq:Mushi} is originally due to 
	Matsushima \cite{MatsushimaBetti}. 
\end{re}

In the rest of this section,  $G$ is supposed to be  $\delta(G)=1$ and has compact center. Recall that $\eta$ is a real finite dimensional representation of $M$ satisfying Assumption \ref{as:1}, and that $\widehat{\eta}$ is defined in \eqref{eq:hatvar}.
In our second application of \eqref{eq:kerC=sum}, we take $\tau=\widehat{\eta}$.

\begin{thm}\label{thm:711}
If $(\pi,V_\pi)\in \widehat{G}_u,$ then
\begin{multline}\label{eq:711}
 \dim_\bC\(V_{\pi,K} \otimes_\bR \widehat{\eta}^+\)^K-\dim_\bC\(V_{\pi,K} \otimes_\bR \widehat{\eta}^-\)^K\\
= \frac{1}{\chi(K/K_M)}\sum_{i=0}^{\dim \fp_\fm}\sum_{j=0}^{2l} (-1)^{i+j}\dim_\bC H^i\big(\fm, K_M; H_j(\fn,V_{\pi,K})\otimes_\bR E_\eta\big).
\end{multline}
\end{thm}
\begin{proof}Let $\Lambda(\pi)\in \fh^*_\bC$ be the Harish-Chandra parameter of the infinitesimal character of $\pi$.
By \eqref{eq:dimeq}, for $t>0$, we have
\begin{multline}\label{eq:68}
  \dim_\bC\(V_{\pi,K} \otimes_\bR \widehat{\eta}^+\)^K-\dim_\bC\(V_{\pi,K} \otimes_\bR \widehat{\eta}^-\)^K\\
=\vol(K)^{-1} e^{t\chi_\pi(C^{\fg})/2}\int_{g\in G} \Theta^G_\pi(g)\Trs\[p^{X,\widehat{\eta}}_{t}(g)\]dv_G.
\end{multline}
By \eqref{eq:G1=00}, by Proposition \ref{prop:test} and by $H\cap K=T$, we have
\begin{multline}\label{eq:wel1}
  \int_G \Theta^G_\pi(g)\Trs\[p^{X,\widehat{\eta}}_{t}(g)\]dv_G=\frac{\vol(T\backslash K)}{|W(H,G)|}\\
\int_{\gamma\in H'}\Theta^G_\pi(\gamma)\Trs^{[\gamma]}\[\exp(-tC^{\fg,X,\widehat{\eta}}/2)\]\left|\mathrm{det}\big(1-\Ad(\gamma)\big)|_{\fg/\fh}\right|dv_{H}.
\end{multline}
Since $\gamma=e^ak^{-1}\in H'$ implies $T=K_M(k)=M^0(k)$,
by \eqref{eq:Trg}, \eqref{eq:68} and \eqref{eq:wel1}, we have
\begin{multline}\label{eq:dim=int}
\dim_\bC\(V_{\pi,K} \otimes_\bR \widehat{\eta}^+\)^K-\dim_\bC\(V_{\pi,K} \otimes_\bR \widehat{\eta}^-\)^K=
\frac{1}{|W(H,G)|\vol(T)}\\
\frac{1}{\sqrt{2 \pi t}}\exp\(
\frac{t}{16}\Tr^{\fu^{\bot}(\fb)}\[C^{\fu(\fb),\fu^{\bot}(\fb)}\]-\frac{t}{2}C^{\fu_\fm,\eta}+\frac{t}{2}\chi_\pi(C^{\fg})\)\\
\int_{\gamma=e^ak^{-1}\in H'}\Theta^G_\pi(\gamma)\exp\(-|a|^2/2t\)\Tr^{E_\eta}\[\eta\(k^{-1}\)\]
\frac{\left|\mathrm{det}\big(1-\Ad(\gamma)\big)|_{\fg/\fh}\right|}{\left|\mathrm{det}\big(1-\Ad(\gamma)\big)|_{\fz_0^\bot}\right|^{1/2}}dv_{H}.
\end{multline}

%
%
%
%
%

By \eqref{eq:det1/2}, for $\gamma=e^ak^{-1}\in H'$, we have
\begin{align}\label{eq:nzs}
\frac{\left|\mathrm{det}\big(1-\Ad(\gamma)\big)|_{\fg/\fh}\right|}{\mathrm{det}\big(1-\Ad(\gamma)\big)|_\fn\left|\mathrm{det}\big(1-\Ad(\gamma)\big)|_{\fz_0^\bot}\right|^{1/2}}=e^{-l\<\alpha,a\>}\left|\mathrm{det}\big(1-\Ad(k^{-1})\big)|_{\fm/\ft}\right|.
\end{align}
By \eqref{eq:smid}, \eqref{eq:G1mq}, \eqref{eq:dim=int}, and \eqref{eq:nzs}, we have
\begin{multline}\label{eq:hs1}
  \dim_\bC\(V_{\pi,K} \otimes_\bR \widehat{\eta}^+\)^K-\dim_\bC\(V_{\pi,K} \otimes_\bR \widehat{\eta}^-\)^K=\frac{1}{|W(H,G)|\vol(T)}\\
\frac{1}{\sqrt{2 \pi t}}\exp\(\frac{t}{16}\Tr^{\fu^{\bot}(\fb)}\[C^{\fu(\fb),\fu^{\bot}(\fb)}\]-\frac{t}{2}C^{\fu_\fm,\eta}+\frac{t}{2}\chi_\pi(C^{\fg})\)\\
\sum_{j=0}^{2l}(-1)^j\int_{\gamma=e^ak^{-1}\in H'}\Theta^{MA_Q}_{H_j(\fn,V_{\pi,K})}(\gamma)
\exp\(-|a|^2/2t-l\<\alpha,a\>\)\\
\Tr^{E_\eta}\[\eta\(k^{-1}\)\]\left|\mathrm{det}\big(1-\Ad(k^{-1})\big)|_{\fm/\ft}\right|dv_H.
\end{multline}

By \eqref{eq:estgol}, there exist  $C>0$ and $c>0$ such that for $\gamma=e^ak^{-1}\in H'$, we have
\begin{align}\label{eqfg}
\left|  \Theta^{MA_Q}_{H_j(\fn,V_{\pi,K})}(\gamma)\right|\left|
\mathrm{det}\(1-\Ad(k^{-1})\)|_{\fm/\ft}\right|^{1/2}\l Ce^{c|a|}.
\end{align}
By  \eqref{eq:MAtoM}, \eqref{eq:hs1}, \eqref{eqfg}, and by  letting $t\to0$, we get
\begin{multline}\label{eq:dimGtoT}
 \dim_\bC\(V_{\pi,K} \otimes_\bR \widehat{\eta}^+\)^K-\dim_\bC\(V_{\pi,K} \otimes_\bR \widehat{\eta}^-\)^K=\frac{1}{|W(H,G)|\vol(T)}\\
\sum_{j=0}^{2l}(-1)^j\int_{\gamma\in T'}\Theta^{M}_{H_j(\fn,V_{\pi,K})}(\gamma)
\Tr^{E_\eta}\[\eta(\gamma)\]\left|\mathrm{det}\big(1-\Ad(\gamma)\big)|_{\fm/\ft}\right|dv_T,
\end{multline}
where $T'$ is the set of the regular elements of $M$ in $T$.

We claim that, for $0\l j\l 2l$, we have
\begin{multline}\label{eq:hs2}
\sum_{i=0}^{\dim \fp_\fm}(-1)^i\dim_\bC\(H_j(\fn,V_{\pi,K})\otimes_\bR\Lambda^i(\fp_\fm^*)\otimes_\bR E_\eta\)^{K_M}\\
=\frac{1}{|W(T,M)|\vol(T)}
 \int_{\gamma\in T'} \Theta^M_{H_j(\fn,V_{\pi,K})}(\gamma)\Tr^{E_\eta}\[\eta\(\gamma\)\]\left|\mathrm{det}\big(1-\Ad(\gamma)\big)|_{{\fm/\ft}}\right|dv_T.
\end{multline}
Indeed, consider $H_j(\fn,V_{\pi,K})$  as a  Harish-Chandra $(\fm_\bC,K_M)$-module. We can decompose $H_j(\fn,V_{\pi,K})$ in the sense of \eqref{eq:V=Vchi}, where the  generalized infinitesimal characters are given by \begin{align}\label{eq:sf1}
 \chi_{ (w\Lambda(\pi)+\rho^\fu-\rho^{\fu(\fb)})|_{\ft_{\bC}}},
\end{align}
for some $w\in W(\fh_\bR,\fu)$. Therefore, it is enough to show \eqref{eq:hs2} when $H_j(\fn,V_{\pi,K})$ is replaced by any Harish-Chandra
$(\fm_\bC,K_M)$-module with  generalized infinitesimal  character
$\chi_{(w\Lambda(\pi)+\rho^\fu-\rho^{\fu(\fb)})|\ft_{\bC}}$. Let $(\pi^M,V_{\pi^M})$ be a Hilbert globalization of such a Harish-Chandra $(\fm_\bC,K_M)$-module. As before,
let $C^{\fm,X_M,\Lambda^\cdot(\fp_\fm)\otimes E_\eta}$ be the Casimir element of $M$ acting on
$C^\infty(M,\Lambda^\cdot(\fp_\fm)\otimes E_\eta)^{K_M}$, and let 
$p_t^{X_M,\Lambda^\cdot(\fp_\fm)\otimes E_\eta}(g)$ be the  smooth 
integral kernel of the heat operator  $\exp(-tC^{\fm,X_M,\Lambda^\cdot(\fp_\fm)\otimes E_\eta}/2)$. Remark that by \cite[Proposition 8.4]{BMZ}, $C^{\fm,X_M,\Lambda^\cdot(\fp_\fm)\otimes E_\eta}-C^{\fm,E_{\eta}}$ is the Hodge Laplacian on $X_M$ acting  on the differential forms with values in the homogenous flat vector bundle $M\times_{K_M} E_\eta$. Proceeding as in \cite[Theorem 7.8.2]{B09}, if $\gamma\in M$ is semisimple and nonelliptic, we have
\begin{align}\label{eq:trd0Fe22}
\Tr^{[\gamma]}\[\exp\(-t(C^{\fm,X_M,\Lambda^\cdot(\fp_\fm)\otimes E_\eta}-C^{\fm,E_{\eta}})/2\)\]=0.
\end{align}
Also, if $\gamma=k^{-1}\in K_M$, then
\begin{align}\label{eq:trd0Fi33}
\Tr^{[\gamma]}\[\exp\(-t(C^{\fm,X_M,\Lambda^\cdot(\fp_\fm)\otimes E_\eta}-C^{\fm,E_{\eta}})/2\)\]=\Tr^{E_\eta}\[\eta\(k^{-1}\)\]e\(X_M(k),\nabla^{TX_M(k)}\).
\end{align}
Using \eqref{eq:trd0Fe22}, proceeding as in \eqref{eq:68} and \eqref{eq:wel1},  we have
\begin{align}\label{eq:jiji}
  &\hspace{10mm}\sum_{i=0}^{\dim \fp_\fm}(-1)^i\dim_\bC\(V_{\pi^M}\otimes_\bR\Lambda^i(\fp_\fm^*)\otimes_\bR E_\eta\)^{K_M}\notag\\
&=\vol(K_M)^{-1}\exp\big(t\chi_{\pi^M}(C^{\fm})/2\big)\int_{g\in M} \Theta^M_{\pi^M}(g) \Trs\[p_t^{X_M,\Lambda^\cdot(\fp^*_\fm)\otimes E_\eta}(g)\]dv_M\\
&\begin{aligned}
=\frac{\exp\big(t\chi_{\pi^M}(C^{\fm})/2\big)}{|W(T,M)|\vol(T)}
\int_{\gamma\in T'} \Theta^M_{\pi^M}(\gamma) \Trs^{[\gamma]}\[\exp(-tC^{\fm,X_M,\Lambda^\cdot(\fp_\fm)\otimes E_\eta}/2)\]\\
\left|\mathrm{det}\big(1-\Ad(\gamma)\big)|_{\fm/\ft}\right|dv_T.\notag
\end{aligned}
\end{align}
By \eqref{eq:trd0Fi33}, \eqref{eq:jiji}, proceeding as in \eqref{eq:dim=int}, and letting $t\to 0$, we get the desired equality \eqref{eq:hs2}.

The Euler formula asserts
\begin{multline}\label{eq:dimeu}
\sum_{i=0}^{\dim \fp_\fm} (-1)^i \dim_\bC\( H_j(\fn,V_{\pi,K})\otimes_\bR\Lambda^i(\fp_\fm^*)\otimes_\bR E_\eta\)^{K_M}\\
=\sum_{i=0}^{\dim \fp_\fm}(-1)^i\dim_\bC  H^i\big(\fm,K_M;H_j(\fn,V_{\pi,K})\otimes_\bR E_\eta\big).
\end{multline}

By \eqref{eq:Weylgroup}, we have
\begin{align}\label{eq:dimeu12}
&W(H,G)=W(T,K),&W(T,M)=W(T,K_M).
\end{align}

By \eqref{eq:Bott}, \eqref{eq:dimGtoT}, \eqref{eq:hs2}, 
\eqref{eq:dimeu}-\eqref{eq:dimeu12}, we get \eqref{eq:711}. The proof 
of our theorem  is completed.
\end{proof}

\begin{cor}\label{cor:forr}
The following identity holds:
\begin{align}\label{eq:jul}
  r_{\eta,\rho}=\frac{1}{\chi(K/K_M)}\!\sum_{\pi\in \widehat{G}_{u},\chi_\pi(C^{\fg})=0}\!\!\!\!n_\rho(\pi)\sum_{ i=0}^{ \dim \fp_\fm}\sum_{ j=0}^{ 2l}(-1)^{i+j}\dim_\bC H^i\big(\fm,K_M;H_j(\fn,V_{\pi,K})\otimes_\bR E_\eta\big).
\end{align}
\end{cor}
\begin{proof}
  This is a consequence of \eqref{eq:rrs}, \eqref{eq:kerC=sum}, and \eqref{eq:711}.
\end{proof}

\begin{re}
  When $G=\SO^0(p,1)$ with $p\g 3$ odd, the formula \eqref{eq:jul} is compatible with \cite[Theorem 3.11]{Juhldymamicalzeta}.
\end{re}

We will apply \eqref{eq:jul} to $\eta_j$. The following proposition allows  us to reduce the first  sum in \eqref{eq:jul} to the one over $\pi\in \widehat{G}_u$ with $\chi_\pi=0$.

\begin{prop}\label{prop:vani3}
Let $(\pi,V_\pi)\in \widehat{G}_u$. Assume $\chi_\pi(C^{\fg})=0$ and
\begin{align}\label{eq:Hn0}
H^\cdot\big(\fm, K_M; H_\cdot(\fn,V_\pi)\otimes_\bR \Lambda^{j}(\fn^{*})\big)\neq0.
\end{align}
Then the infinitesimal character $\chi_\pi$ vanishes.
\end{prop}
\begin{proof}
Recall that $\Lambda(\pi)\in \fh^*_\bC$ is a Harish-Chandra parameter of $\pi$. We need to show that there is $w\in W(\fh_\bR,\fu)$ such that
\begin{align}\label{eq:hou}
  w\Lambda(\pi)=\rho^\fu.
\end{align}
Let $B^*$ be the bilinear form on $\fg^*$ induced by $B$. It extends to $\fg^*_\bC$ and $\fu^*$ in an obvious way. 
Since $\chi_\pi(C^{\fg,\pi})=0$, we have
\begin{align}\label{eq:C=0}
  B^*(\Lambda(\pi),\Lambda(\pi))=B^*(\rho^\fu,\rho^\fu).
\end{align}
%

We identify $\fh_\bR^*=\sqrt{-1}\fb^*\oplus \ft^*$.
By definition,
\begin{align}
  \rho^{\fu}=\(\frac{l\alpha}{2\sqrt{-1}\pi},\rho^{\fu_{\fm}}\)\in  \sqrt{-1}\fb^*\oplus \ft^* \hbox{\quad and \quad} \rho^{\fu(\fb)}=(0, \rho^{\fu_\fm})\in \sqrt{-1}\fb^*\oplus \ft^*.
\end{align}
%
By \eqref{eq:key1}, \eqref{eq:sf1} and \eqref{eq:Hn0}, there exist $w\in W(\fh_\bR,\fu)$, $w'\in W(\ft,\fu_\fm)\subset W(\fh_\bR,\fu)$ and the highest real weight $\mu_j\in \ft^*$ of an irreducible subrepresentation of $\fm_\bC$ on $\Lambda^j(\fn_\bC)\simeq \Lambda^j(\ol{\fn}^*_\bC)$ such that
\begin{align}\label{eq:sf3}
 w\Lambda(\pi)|_{\ft_\bC}=w'(\mu_j+\rho^{\fu_\fm}).
\end{align}

By \eqref{eq:BggBuu}, \eqref{eq:C=0} and \eqref{eq:sf3}, there exists $w''\in  W(\fh_\bR,\fu)$ such that
\begin{align}\label{eq:Cpi=0}
  w''\Lambda(\pi)=\(\pm\frac{(l-j)\alpha}{2\sqrt{-1}\pi}, \mu_j+\rho^{\fu_\fm}\)=\(\pm\frac{(l-j)\alpha}{2\sqrt{-1}\pi}, \mu_j\)+\rho^{\fu(\fb)}.
\end{align}
In particular, $w''\Lambda(\pi)\in \fh^*_\bR$.

Clearly, $\big((j-l)\alpha/2\sqrt{-1}\pi, \mu_j\big)\in \fh_\bR^*$ is the highest
real weight of an irreducible subrepresentation of
$\fm_\bC\oplus \fb_\bC$ on
$\Lambda^j(\ol{\fn}^*_\bC)\otimes_\bC(\det(\fn_\bC))^{-1/2}$.
By \eqref{eq:Rnspinc}, $\big((j-l)\alpha/2\sqrt{-1}\pi, \mu_j\big)\in \fh_\bR^*$ is the highest
real weight of an irreducible subrepresentation of  $\fm_\bC\oplus \fb_\bC$ on $S^{\fu^\bot(\fb)}$.
By \cite[Lemma II.6.9]{BW},  there exists $w_1\in W(\fh_\bR,\fu)$ such that
\begin{align}\label{eq:la2}
  \(\,\frac{(j-l)\alpha}{2\sqrt{-1}\pi}, \mu_j\)=w_1\rho^{\fu}-\rho^{\fu(\fb)}.
\end{align}
Similarly, $\big((l-j)\alpha/2\sqrt{-1}\pi, \mu_j\big)\in \fh_\bR^*$ is the highest
real weight of an irreducible subrepresentation of $\fm_\bC\oplus \fb_\bC$ on both $\Lambda^{2l-j}(\ol{\fn}^*_\bC)\otimes_\bC(\det(\fn_\bC))^{-1/2}$ and $S^{\fu^\bot(\fb)}$. Therefore,  there exists $w_2\in W(\fh_\bR,\fu)$ such that
\begin{align}\label{eq:la3}
  \(\,\frac{(l-j)\alpha}{2\sqrt{-1}\pi}, \mu_j\)=w_2\rho^{\fu}-\rho^{\fu(\fb)}.
\end{align}

%
%

%
%
%
By \eqref{eq:Cpi=0}-\eqref{eq:la3}, we get \eqref{eq:hou}.
The proof of our proposition  is completed.
\end{proof}

\begin{cor}\label{cor:forr1}
 For $0\l j\l 2l$, we have
\begin{align}
  r_j=\frac{1}{\chi(K/K_M)}\sum_{\pi\in 
  \widehat{G}_{u},\chi_\pi=0}n_\rho(\pi)\sum_{i=0}^{\dim 
  \fp_\fm}\sum_{k=0}^{2l}(-1)^{i+k}\dim_\bC 
  H^i\(\fm,K_M;H_k(\fn,V_{\pi,K})\otimes_\bR \Lambda^{j}(\fn^{*})\).
\end{align}
If $H^\cdot(Z,F)=0$, then for all $0\l j\l 2l$,
\begin{align}\label{eq:yep}
  r_j=0.
\end{align}
\end{cor}
\begin{proof}
  This is a consequence of  Proposition \ref{prop:mushi}, Corollary \ref{cor:forr} and Proposition \ref{prop:vani3}.
\end{proof}

\begin{re}By \eqref{eq:Cr} and \eqref{eq:yep},  we get  
	\eqref{eq:t02z} when $G$ has compact center and $\delta(G)=1$.
\end{re}

\def\cprime{$'$}
\providecommand{\bysame}{\leavevmode\hbox to3em{\hrulefill}\thinspace}
\providecommand{\MR}{\relax\ifhmode\unskip\space\fi MR }
\providecommand{\MRhref}[2]{%
  \href{http://www.ams.org/mathscinet-getitem?mr=#1}{#2}
}
\providecommand{\href}[2]{#2}

\begin{theindex}

  \item $\lVert\cdot\rVert_{r,z_1,z_2}$, \hyperpage{56}
  \item $\langle,\rangle_{\fn_\bC}$, \hyperpage{31}
  \item $\lvert\cdot\rvert$, \hyperpage{12}

  \indexspace

  \item $A_0$, \hyperpage{31}
  \item $[\alpha ]^{\qopname  \relax m{max}}$, \hyperpage{17}
  \item $\alpha$, \hyperpage{29}
  \item $\widehat{A}$, \hyperpage{10}
  \item $\widehat{A}(E,\nabla^E)$, \hyperpage{10}
  \item $a_0$, \hyperpage{29}

  \indexspace

  \item $ \mathfrak  {b}(\gamma )$, \hyperpage{22}
  \item $B$, \hyperpage{6}, \hyperpage{12}
  \item $B_{[\gamma]}$, \hyperpage{4}, \hyperpage{24}
  \item $[\beta]^{\max}$, \hyperpage{45}
  \item $\Box^Z$, \hyperpage{3}, \hyperpage{11}
  \item $\fb_*$, \hyperpage{40}
  \item $\mathfrak  {b}$, \hyperpage{14}

  \indexspace

  \item $C^{\fg,X,\tau}$, \hyperpage{17}
  \item $C^{\fg,X}$, \hyperpage{17}
  \item $C^{\fk}$, \hyperpage{13}
  \item $C^{\fu(\fb),\fu^\bot(\fb)}$, \hyperpage{34}
  \item $C^{\fu_\fm,\eta_j}$, \hyperpage{34}
  \item $C^{\mathfrak  {g}}$, \hyperpage{13}
  \item $C^{\mathfrak  {k},V},C^{\mathfrak  {k},\tau }$, \hyperpage{13}
  \item $C_\rho$, \hyperpage{53}
  \item $\chi (Z,F),\chi '(Z,F)$, \hyperpage{11}
  \item $\chi(Z)$, \hyperpage{11}
  \item $\chi_\lambda$, \hyperpage{46}
  \item $\chi_{\Lambda}$, \hyperpage{55}
  \item $\chi_{\rm orb}$, \hyperpage{4}, \hyperpage{24}
  \item $\mathrm  {ch}(E',\nabla ^{E'})$, \hyperpage{10}
  \item $c_G$, \hyperpage{45}

  \indexspace

  \item $ \det_{\rm gr}$, \hyperpage{51}
  \item $D^Z$, \hyperpage{11}
  \item $\Delta(\ft,\fk),\Delta^+(\ft,\fk)$, \hyperpage{46}
  \item $\delta(G)$, \hyperpage{6}, \hyperpage{14}
  \item $\delta(\fg)$, \hyperpage{15}
  \item $\det(P+\sigma)$, \hyperpage{10}
  \item $\qopname  \relax m{det}(P)$, \hyperpage{10}
  \item $d^Z$, \hyperpage{11}
  \item $d^{Z,*}$, \hyperpage{11}
  \item $d_X(x,y)$, \hyperpage{17}
  \item $dv_X$, \hyperpage{17}
  \item $dv_{H_i},dv_{H_i\backslash G}$, \hyperpage{15}
  \item $dv_{K^0(\gamma)\backslash K},dv_{Z^0(\gamma)\backslash G}$, 
		\hyperpage{18}
  \item $dv_G$, \hyperpage{58}

  \indexspace

  \item $E_\tau$, \hyperpage{17}
  \item $E_{\widehat{\eta}},E^+_{\widehat{\eta}},E^-_{\widehat{\eta}}$, 
		\hyperpage{45}
  \item $\cE_\tau$, \hyperpage{17}
  \item $\cE_{\widehat{\eta}}$, \hyperpage{45}
  \item $\eta$, \hyperpage{44}
  \item $\eta_j$, \hyperpage{34}
  \item $\widehat{\eta},\widehat{\eta}^+,\widehat{\eta}^-$, 
		\hyperpage{45}
  \item $e(E,\nabla ^E)$, \hyperpage{10}

  \indexspace

  \item $F$, \hyperpage{3}, \hyperpage{6}, \hyperpage{11}, 
		\hyperpage{20}
  \item $F_{\fb,\eta}$, \hyperpage{45}
  \item $\cF_\tau$, \hyperpage{20}

  \indexspace

  \item $G$, \hyperpage{6}, \hyperpage{12}
  \item $G'$, \hyperpage{16}
  \item $G_*$, \hyperpage{39}
  \item $G_1,G_2$, \hyperpage{39}
  \item $G_\bC$, \hyperpage{12}
  \item $G_i'$, \hyperpage{16}
  \item $[\Gamma]$, \hyperpage{4}, \hyperpage{21}
  \item $[\gamma]$, \hyperpage{4}, \hyperpage{21}
  \item $\Gamma (\gamma )$, \hyperpage{21}
  \item $\Gamma$, \hyperpage{4}, \hyperpage{6}, \hyperpage{11}, 
		\hyperpage{20}
  \item $\fg$, \hyperpage{6}, \hyperpage{12}
  \item $\fg_*$, \hyperpage{39}
  \item $\fg_1,\fg_2$, \hyperpage{35}
  \item $\fg_\bC$, \hyperpage{12}
  \item $g^{F}$, \hyperpage{3}, \hyperpage{11}
  \item $g^{TX}$, \hyperpage{17}
  \item $g^{TZ}$, \hyperpage{3}, \hyperpage{11}

  \indexspace

  \item $ \mathfrak  {h}_{i\mathfrak  {p}}, \mathfrak  {h}_{i\mathfrak  {k}}$, 
		\hyperpage{14}
  \item $H^\cdot(Z,F)$, \hyperpage{11}
  \item $H^\cdot(\fg,K;V)$, \hyperpage{61}
  \item $H_\cdot(\fn_\fq,V)$, \hyperpage{62}
  \item $H_i'$, \hyperpage{16}
  \item $\fh(\gamma)$, \hyperpage{15}
  \item $\fh_i,H_i$, \hyperpage{14}
  \item $\fh_{i\bC},\fh_{i\bR}$, \hyperpage{54}
  \item $\mathfrak  {h}$, \hyperpage{14}

  \indexspace

  \item $\iota$, \hyperpage{34}
  \item $i_G$, \hyperpage{36}

  \indexspace

  \item $J$, \hyperpage{31}
  \item $J_{\gamma }$, \hyperpage{19}

  \indexspace

  \item $K$, \hyperpage{6}, \hyperpage{12}
  \item $K_*$, \hyperpage{39}
  \item $K_{M}$, \hyperpage{23}
  \item $\fk$, \hyperpage{6}, \hyperpage{12}
  \item $\fk_*$, \hyperpage{39}
  \item $\fk_0^{\bot}$, \hyperpage{19}
  \item $\fk_0^{\bot}(\gamma)$, \hyperpage{19}
  \item $\fk_\fm(k)$, \hyperpage{30}
  \item $\fk_{\fm}$, \hyperpage{23}
  \item $\mathfrak  {k}(\gamma )$, \hyperpage{13}
  \item $\mathfrak  {k}_0$, \hyperpage{19}

  \indexspace

  \item $l$, \hyperpage{29}
  \item $l_0$, \hyperpage{14}
  \item $l_{[\gamma]}$, \hyperpage{4}

  \indexspace

  \item $M$, \hyperpage{23}
  \item $M(k),M^0(k)$, \hyperpage{30}
  \item $\fm$, \hyperpage{23}
  \item $\fm(k)$, \hyperpage{30}
  \item $m$, \hyperpage{11, 12}
  \item $m_P(\lambda )$, \hyperpage{10}
  \item $m_{[\gamma]}$, \hyperpage{4}, \hyperpage{24}
  \item $m_{\eta,\rho}(\lambda)$, \hyperpage{51}

  \indexspace

  \item $N^{\Lambda^\cdot(T^*Z)}$, \hyperpage{3}
  \item $N_\fb$, \hyperpage{32}
  \item $\mathfrak  {n},\overline  {\mathfrak  {n}}$, \hyperpage{28}
  \item $\nabla^{TX}$, \hyperpage{17}
  \item $n$, \hyperpage{12}

  \indexspace

  \item $ \omega ^\mathfrak  {g},\omega ^\mathfrak  {p},\omega ^\mathfrak  {k}$, 
		\hyperpage{16}
  \item $ \omega ^{\mathfrak  {u}},\omega ^{\mathfrak  {u}(\mathfrak  {b})},\omega ^{\mathfrak  {u}^{\bot }(\mathfrak  {b})}$, 
		\hyperpage{32}
  \item $ \omega^{Y_\fb}$, \hyperpage{32}
  \item $\Omega ^\mathfrak  {k}$, \hyperpage{17}
  \item $\Omega^\cdot(Z,F)$, \hyperpage{3}, \hyperpage{11}
  \item $\Omega^{\fu(\fb)}$, \hyperpage{32}
  \item $\Omega^{\fu_{\fm}}$, \hyperpage{33}
  \item $o(TX)$, \hyperpage{17}

  \indexspace

  \item $P_{\eta}(\sigma)$, \hyperpage{51}
  \item $\fp$, \hyperpage{12}
  \item $\fp^{a,\bot}(\gamma)$, \hyperpage{13}
  \item $\fp_*$, \hyperpage{39}
  \item $\fp_0^{\bot}$, \hyperpage{19}
  \item $\fp_0^{\bot}(\gamma)$, \hyperpage{19}
  \item $\fp_\fm(k)$, \hyperpage{30}
  \item $\fp_{\fm}$, \hyperpage{23}
  \item $\mathfrak  {p}(\gamma )$, \hyperpage{13}
  \item $\mathfrak  {p}_0$, \hyperpage{19}
  \item $\pi_{\fk}(Y)$, \hyperpage{46}
  \item $\widehat{p},\widehat{\pi}$, \hyperpage{20}
  \item $p_t^{X,\tau }(g)$, \hyperpage{18}
  \item $p_t^{X,\tau}(x,x')$, \hyperpage{18}
  \item $P$, \hyperpage{16}

  \indexspace

  \item $\fq$, \hyperpage{29}

  \indexspace

  \item $RO(K_M),RO(K)$, \hyperpage{7}, \hyperpage{34}
  \item $R^{F_{\fb,\eta}}$, \hyperpage{45}
  \item $R^{N_\fb}$, \hyperpage{33}
  \item $R_\rho(\sigma)$, \hyperpage{4}, \hyperpage{25}
  \item $\mathrm{rk}_\bC$, \hyperpage{14}
  \item $\rho$, \hyperpage{6}, \hyperpage{11}, \hyperpage{20}
  \item $\rho^{\fk}$, \hyperpage{46}
  \item $r$, \hyperpage{11}, \hyperpage{20}
  \item $r_\rho$, \hyperpage{53}
  \item $r_j$, \hyperpage{53}
  \item $r_{\eta,\rho}$, \hyperpage{51}

  \indexspace

  \item $S^{\fu^\bot(\fb)}$, \hyperpage{31}
  \item $\cS(G)$, \hyperpage{56}
  \item $\sigma_0$, \hyperpage{26}, \hyperpage{52}
  \item $\sigma_\eta$, \hyperpage{51}
  \item $\sigma_{\fk}(Y)$, \hyperpage{46}

  \indexspace

  \item $T$, \hyperpage{14}
  \item $T(F)$, \hyperpage{3}, \hyperpage{11}
  \item $T(\sigma )$, \hyperpage{8}, \hyperpage{11}
  \item $\Theta^G_\pi$, \hyperpage{56}
  \item $\Tr  ^{[\gamma ]}[\cdot ]$, \hyperpage{18}
  \item $\Trs^{[\gamma]}[\cdot]$, \hyperpage{18}
  \item $\ft(\gamma)$, \hyperpage{22}
  \item $\ft_*$, \hyperpage{40}
  \item $\tau$, \hyperpage{13}, \hyperpage{17}
  \item $\theta (s)$, \hyperpage{3}
  \item $\theta _P(s)$, \hyperpage{10}
  \item $\theta$, \hyperpage{6}, \hyperpage{12}

  \indexspace

  \item $U$, \hyperpage{12}
  \item $U(\fb),U_M$, \hyperpage{31}
  \item $U(\fg)$, \hyperpage{13}
  \item $U(\fg_\bC)$, \hyperpage{54}
  \item $U_M(k),U_M^0(k)$, \hyperpage{49}
  \item $\fu$, \hyperpage{12}
  \item $\fu(\fb),\fu(\fm)$, \hyperpage{31}
  \item $\fu^\bot(\fb)$, \hyperpage{31}
  \item $\fu_{\fm}(k)$, \hyperpage{49}

  \indexspace

  \item $\vol(\cdot)$, \hyperpage{17, 18}

  \indexspace

  \item $W(H_i,G)$, \hyperpage{14}
  \item $W(T,K)$, \hyperpage{17}

  \indexspace

  \item $X$, \hyperpage{6}, \hyperpage{16}
  \item $X(\gamma )$, \hyperpage{17}
  \item $X^{a,\bot }(\gamma )$, \hyperpage{18}
  \item $X_M(k)$, \hyperpage{30}
  \item $X_{M}$, \hyperpage{23}
  \item $\Xi(g)$, \hyperpage{57}
  \item $\Xi_{\rho}$, \hyperpage{25}

  \indexspace

  \item $Y_\mathfrak  {b}$, \hyperpage{31}

  \indexspace

  \item $Z$, \hyperpage{3}, \hyperpage{11}, \hyperpage{20}
  \item $Z(\gamma)$, \hyperpage{13}
  \item $Z(a)$, \hyperpage{13}
  \item $Z^0(\fb)$, \hyperpage{28}
  \item $Z^0(\gamma)$, \hyperpage{13}
  \item $Z_G$, \hyperpage{12}
  \item $\cZ(\fg_\bC)$, \hyperpage{54}
  \item $\fz(\gamma)$, \hyperpage{13}
  \item $\fz(a)$, \hyperpage{13}
  \item $\fz^{a,\bot}(\gamma)$, \hyperpage{13}
  \item $\fz_0^{\bot}$, \hyperpage{19}
  \item $\fz_0^{\bot}(\gamma)$, \hyperpage{19}
  \item $\fz_\fg$, \hyperpage{12}
  \item $\mathcal{Z}(\fg)$, \hyperpage{13}
  \item $\mathfrak  {z}_0$, \hyperpage{19}
  \item $\mathfrak  {z}_{\mathfrak  {p}},\mathfrak  {z}_{\mathfrak  {k}}$, 
		\hyperpage{12}

\end{theindex}


\end{document}